\documentclass[11pt]{amsart}

\usepackage{amsmath,amssymb,latexsym, amsbsy, amsfonts, amscd, amsthm, mathrsfs, xy,comment, xcolor}
\usepackage{multirow}
\usepackage{hyperref}
\usepackage{enumitem}
 \usepackage[bibtex-style]{amsrefs}
\xyoption{all}

\theoremstyle{plain}

\newtheorem{ithm}{Theorem}

\newtheorem{theorem}{Theorem}[section]

\newtheorem{prop}[theorem]{Proposition}
\newtheorem{corollary}[theorem]{Corollary}

\newtheorem{lemma}[theorem]{Lemma}

\newcommand{\longhookrightarrow}{\lhook\joinrel\longrightarrow}

\theoremstyle{definition}
\newtheorem{definition}[theorem]{Definition}
\newtheorem{remark}[theorem]{Remark}

\long\def\symbolfootnote[#1]#2{\begingroup
\def\thefootnote{\fnsymbol{footnote}}\footnote[#1]{#2}\endgroup}

\def\SL{{\bf SL}}
\def\GL{{\bf GL}}

\def\A{\mathbf{A}}

\def\sgn{\mathrm{sgn}}
\def\N{\mathrm{N}}
\def\1{\mf{1}}

\DeclareMathOperator{\Sch}{Sch}
\DeclareMathOperator{\AlgSp}{AlgSp}
\DeclareMathOperator{\AnSp}{AnSp}

\DeclareMathOperator{\Sym}{Sym}

\DeclareMathOperator{\Hom}{Hom} \DeclareMathOperator{\End}{End}
\DeclareMathOperator{\Stab}{Stab}

\DeclareMathOperator{\coker}{coker} 
\DeclareMathOperator{\ord}{ord}

 \DeclareMathOperator{\Cl}{Cl}

\DeclareMathOperator{\Spec}{Spec}
\DeclareMathOperator{\tr}{tr}

\newcommand{\Res}{\mathrm{\,Res\,}}

\newcommand{\bpsi}{\pmb{\psi}}

\newcommand{\mat}[4]{\left( \begin{array}{cc} {#1} & {#2} \\ {#3} & {#4}
\end{array} \right)}

\newcommand{\mf}{\mathfrak }

\def\fn{\mathfrak{n}}

\def\ft{\mathfrak{t}}

\def\fP{\mathfrak{P}}

\def\fd{\mathfrak{d}}

\def\Z{\mathbf{Z}}
\def\F{\mathbf{F}}
\def\Q{\mathbf{Q}}
\def\C{\mathbf{C}}
\def\G{\mathbf{G}}
\def\R{\mathbf{R}}

\def\bdf{\begin{definition}}
\def\edf{\end{definition}}

\def\cL{\mathcal{L}}
\def\cH{\mathcal{H}}

\def\cO{\mathcal{O}}

\def\Gal{{\rm Gal}}

\def\mc{\mathcal}

\usepackage{tikz}
\usepackage{tikz-cd}
\usetikzlibrary{matrix,arrows,backgrounds}
\usepackage[colorinlistoftodos, textsize=tiny]{todonotes} 

\usepackage{stmaryrd}
\newcommand{\myfatslash}{\mathbin{\mkern-6mu\fatslash}\,}

\newtheorem{eg}[theorem]{Example}

\def\bP{\mathbf{P}}
\newcommand{\from}{\colon}
\newcommand{\isom}{\cong}
\newcommand{\into}{\hookrightarrow}
\newcommand{\surj}{\twoheadrightarrow}
\newcommand{\nmto}[1]{\stackrel{#1}{\to}}
\newcommand{\nmleft}[1]{\stackrel{#1}{\leftarrow}}
\newcommand{\nmisom}[1]{\stackrel{#1}{\isom}}
\newcommand{\tensor}{\otimes}
\newcommand{\wh}[1]{\widehat{#1}}
\newcommand{\wt}[1]{\widetilde{#1}}
\DeclareMathOperator{\Aut}{Aut}
\DeclareMathOperator{\Proj}{Proj} 
\DeclareMathOperator{\Spf}{Spf} 
\DeclareMathOperator{\Lie}{Lie} 
\newcommand{\mr}[1]{\mathrm{#1}}
\newcommand{\dual}{\vee}
\def\M{{\mathcal M}}
\def\N{{\mathbf N}}

\begin{document}
\baselineskip 14pt

\title{Group Ring Valued Hilbert Modular Forms}
\author{Jesse Silliman}

\maketitle

\begin{abstract}
In this paper, we study the action of diamond operators on Hilbert modular forms with coefficients in a general commutative ring. In particular, we generalize a result of Chai on the surjectivity of the constant term map for Hilbert modular forms with nebentype to the setting of group ring valued modular forms. As an application, we construct certain Hilbert modular forms required for Dasgupta--Kakde's proof of the Brumer--Stark conjecture at odd primes. Since the forms required for the Brumer--Stark conjecture live on the non-PEL Shimura variety associated to the reductive group $G = \Res_{F/\Q}(\GL_2)$, as opposed to the PEL Shimura variety associated to the subgroup $G^* \subset G$ studied by Chai, we give a detailed explanation of theory of algebraic diamond operators for $G$, as well as how the theory of toroidal and minimal compactifications for $G$ may be deduced from the analogous theory for $G^*$.
\end{abstract}

\setcounter{tocdepth}{1}

\section{Introduction}

\subsection{The Brumer--Stark conjecture}
Ribet's method is a powerful technique that allows one to relate special values of $L$-functions to certain associated Galois cohomology classes and thereby make progess on Bloch--Kato type conjectures.  Beyond Ribet's original work \cite{ribet}, the technique was used by Mazur--Wiles \cite{mw} and Wiles \cite{wiles}  to prove the Iwasawa main conjecture for Hecke characters of totally real fields; by Skinner--Urban \cite{su} to prove the Iwasawa main conjecture for elliptic curves over $\Q$; and by Darmon, Dasgupta, Kakde, Pollack, and Ventullo to prove the Gross--Stark conjecture \cite{ddp}, \cite{dkv}.  In each case one begins with a suitable automorphic form whose constant terms at a certain cusp are equal to, or at least divisible by, the $L$-function under study.  This automorphic form is typically an Eisenstein series.  One then proves that this Eisenstein series is congruent to a cuspidal form modulo the $L$-value.  The construction of this cusp form is a delicate process, and provides the motivation for the study undertaken in this paper.

 In recent work \cite{dk}, Dasgupta and Kakde prove the Brumer--Stark conjecture away from $p=2$. In a sequel, they prove an exact $p$-adic analytic formula for Brumer--Stark units that had been conjectured earlier by Dasgupta, thereby making progress on Hilbert's 12th problem \cite{dk2}.  Both of these papers apply suitable generalizations of Ribet's method.  In this context, the relevant automorphic forms are {\em group ring valued Hilbert modular forms} for $F$. The purpose of the present paper is to prove the existence of certain Hilbert modular forms that are applied in those works.
 
  One result of this paper is the existence of a group ring valued modular form with constant terms 1---or more generally any possible system of constant terms---at all $p$-unramified cusps.
This is a generalization of a result of Chai \cite{chai} used by Wiles to prove the Iwasawa Main Conjecture \cite{wiles}. Another result is the existence of a Hilbert modular form $V$ lifting the Hasse invariant, a mild generalization of a result of Hida  \cite{wilesrep}*{Lemma 1.4.2}.

\medskip

\bigskip

\subsection{Statement of results}

Let $F$ be a totally real field and $\fn \subset \cO_F$ an ideal. We consider Hilbert modular forms for the group $G = \Res_{F/\Q}(\GL_2)$. More precisely, following Shimura \cite{shim}, let $M_k(\mf{n}, \C)$ denote the space of Hilbert modular forms of weight $k$ and level $\Gamma_1(\mf{n})$. Let $G_\fn^+$ denote the narrow ray class group of conductor $\fn$, and consider a character $\psi \from G^+_{\mf{n}} \to \C^*$.  There is an action of $G^+_{\mf{n}}$ on $M_k(\mf{n}, \C)$ such that that $\psi$-isotypic component, $M_k(\mf{n}, \psi, \C) := M_k(\mf{n}, \C)^{\psi}$, consists of those modular forms of nebentype $\psi$.

There is a $G^+_{\mf{n}}$-equivariant constant term map
\begin{equation}
\label{e:conk}
\mr{const} \colon M_k(\fn, \C) \longrightarrow C_k(\C), \end{equation}
where $C_k(\C) \cong  \bigoplus_{\mr{Cusp}(\fn)^*} \C$, although this identification is not entirely canonical, as we will discuss. Here $\mr{Cusp}(\fn)^* \subset \mr{Cusp}(\fn)$ are the {\em admissible} cusps for the weight $k$, defined in \S\ref{subsec:type-of-cusp}. Using the theory of algebraic modular forms and their $q$-expansions, one can define for any ring $R$ the modules $M_k(\fn, R), C_k(R)$, as well as a constant term map $\mr{const} \colon M_k(\fn, R) \longrightarrow C_k(R)$.

Now fix a prime $p$, let $R$ be a $\Z_{(p)}$-algebra, and consider a character $\psi \from G^+_{\mf{n}} \to R^*$. By restricting the constant term map to only the admissible cusps $\mr{Cusp}(\fn)^*_p$  which are ``$p$-unramified" (\S\ref{subsec:type-of-cusp}), we obtain a $G^+_{\mf{n}}$-equivariant constant term map:
\[ \mr{const}_p \from M_k(\mf{n}, R) \to C_{p, k}(R) \isom \bigoplus_{\mr{Cusp}(\fn)^*_p} R. \]

In \S\ref{subsec:equivar-const}, we prove:
\begin{ithm} \label{t:main}
For $k$ sufficiently large, $\mr{const}_p$ is surjective. Moreover, the restriction of $\mr{const}_p$ to the $\psi$-isotypic components, 
\[ M_k(\mf{n}, R)^{\psi} \to C_{p, k}(R)^{\psi}, \]
is also surjective.
\end{ithm}

This theorem is the main input to \cite{dk}, \cite{dk2} proven in this paper. The proof of Theorem~\ref{t:main} follows that of an analogous theorem of Chai \cite{chai}*{Proposition in \S4.5}, using the interpretation of 
$M_k(\fn, R)^{\psi}$ as the space of global sections of a certain line bundle on a projective scheme.

Let us highlight three differences between our result and that of Chai.
 (1) We work with $R$-valued modular forms, rather than classical forms over $\Z$.  (2) We work with Hilbert modular forms associated to the reductive group $G = \Res_{F/\Q} \GL_2$, while Chai works with those associated to a certain subgroup $G^* \subset G$. Here $G^*$ is the algebraic subgroup of $G$ whose $\Q$-points are the matrices in $\GL_2(F)$ with determinant lying in $\Q$. This means that we are studying a different space of modular forms than in \cite{chai}. More importantly, this allows us to incorporate the full group of diamond operators $G_\fn^+$, rather than the kernel of the canonical map $G_\fn^+ \rightarrow G_1^+$ as in \cite{chai}. (3) As $R$ is assumed to be a $\Z_{(p)}$-algebra, we are able to include all of the $p$-unramified cusps  $\mr{Cusp}_p(\fn) \subset \mr{Cusp}(\fn)$ rather than the smaller set of unramified cusps $\mr{Cusp}_{\infty}(\fn) \subset \mr{Cusp}(\fn)$ considered in \cite{chai}. This is an important generalization because the diamond operators do not act freely on $\mr{Cusp}_p(\fn)$, making certain computations more involved than the analogous ones in \cite{chai}.   All of these generalizations are essential in the applications in \cite{dk} and \cite{dk2}.  

In \S\ref{subsec:lifting}, we prove a lifting result under certain restrictive hypotheses. Suppose that $R$ is a \emph{good} ring (see Definition \ref{defn:good-ring}). In particular, $2N_{F/\Q}(\mf{n})$ is invertible in $R$. Fixing an ideal $I \subset R$, there is a reduction map 
\[ M_k(\mf{n}, R) \to M_k(\mf{n}, R/I) \] compatible with the reduction of constant terms $C_k(R) \to C_k(R/I).$ 

\begin{ithm}\label{t:lift}
For $k$ sufficiently large, the cokernel of $M_k(\mf{n}, R) \to M_k(\mf{n}, R/I)$ equals the cokernel of \[ C_k(R) \to C_k(R/I).\] Moreover, the restriction to the $\psi$-isotypic components
\[ M_k(\mf{n}, R)^{\psi} \to M_{k}(\mf{n}, R/I)^{\bar{\psi}} \] has cokernel 
equal to the cokernel of $C_k(R)^{\psi} \to C_k(R/I)^{\bar{\psi}}$. In particular, the reduction of cusp forms with fixed nebentype, $S_k(\mf{n}, R)^{\psi} \to S_{k}(\mf{n}, R/I)^{\bar{\psi}}$, is surjective.
\end{ithm}

The requirement that $N_{F/\Q}(\mf{n})$ is invertible in $R$ is quite restrictive, but we expect this requirement can be weakened. The method of proof, involving the computation of the action of $G^+_{\mf{n}}$ on $q$-expansions at all cusps, may be of independent interest.

In \S\ref{s:vk} we prove the following result, which is a mild generalization of a result of Hida whose statement and proof appear in Lemma 1.4.2 of the paper \cite{wilesrep} of Wiles.  Let $p$ be a prime. We denote by $ M_k((1), 1, \Z_{(p)})$ the space of modular forms over $\Z_{(p)}$ of weight $k$, level $1$, and trivial nebentype.

\begin{ithm}\label{t:vk}
Fix $n > 0$. For $k$ sufficiently large and divisible by $(p-1)$, there exists a form $V \in M_k((1), 1, \Z_{(p)})$ whose $q$-expansions (at all cusps) are congruent to $1$ modulo $p^n$. 
\end{ithm}

This form $V$ is needed for the applications of \cite{dk}, \cite{dk2}. The Hida--Wiles version of this result does not require $V$ to have level 1 and trivial nebentype; their  construction  uses the theta series associated to a certain cyclotomic extension of $F$.  One could deduce our strengthening from their result by averaging, but this would require an understanding of the $q$-expansion of the theta series at {\em all} cusps.  It is more convenient to prove the existence of a modular form with $q$-expansion congruent to $1$ mod $p$ at all cusps via a geometric argument by lifting the Hasse invariant, as in \cite{ag}.  This is the strategy that we employ.  

\subsection{Exposition}
In some sense, much of this paper is expository. However, many of the results in the literature are not in the form necessary to be applied to \cite{dk}, \cite{dk2}. This is because many results are proven only for Hilbert modular forms for the group $G^*$ instead of $G$, but also because careful descriptions relating classical cusps and $q$-expansions to the algebraic theory defined by Rapoport \cite{rapoport} are not available.

We have used the language of analytic and algebraic stacks in this paper (see Appendix \ref{sec:stack} for details). Much of the literature avoids this via the use of auxiliary level structures. However, in this paper, we cannot simply prove results for modular forms of level $\Gamma_1(\mf{n}) \cap \Gamma(M)$, for some $M > 3$ coprime to $\mf{n}$, and then formally descend these results to level $\Gamma_1(\mf{n})$. Our results are sensitive to $p$-integral structures, so we would need to choose $M$ such that $\SL_2(\cO/M\cO)/\{ \pm I \}$ has order coprime to $p$, and this is not possible for small $p$.

There are many foundational works on algebraic Hilbert modular forms for $G^*$ (e.g. \cite{rapoport}, \cite{katz}, \cite{chai}, \cite{deligne-pappas}), as well as more recent works which study algebraic Hilbert modular forms for $G$ (e.g. \cite{Kisin-Lai}, \cite{Dimitrov-Tilouine}). Unlike these works, we study algebraic Hilbert modular forms for $G$ by directly working with the Shimura variety associated to $G$, instead of the PEL type Shimura variety associated with $G^*$. This requires the use of certain stacks which are not Deligne--Mumford, as their inertia groups are infinite discrete groups (see \S\ref{subsec:shimura-var}, \S\ref{subsec:G-moduli}). However, this excursion into non-algebraic stacks is merely an intermediate step in the definition of algebraic stacks, whose algebraicity is verified, by descent, from the algebraic stacks constructed by Rapoport \cite{rapoport}.  Despite this complication, the use of these stacks is not difficult, makes many constructions more canonical, and eliminates the need to keep track of polarizations. For example, the use of $G$-Shimura varieties and the language of stacks helps in:
\begin{enumerate}
\item comparing Shimura normalizations for $q$-expansions (\cite{shim}, used in Dasgupta--Kakde \cite{dk}, \cite{dk2}) with that used in the algebraic literature (Rapoport \cite{rapoport}, Chai \cite{chai});
\item discussing the simple moduli-theoretic intepretation of $G$-Shimura varieties in terms of \emph{unpolarized} complex tori with real multiplication;
\item describing the $G^+_{\mf{n}}$-action on toroidal compactifications and $q$-expansions.
\end{enumerate}

\medskip

\bigskip
We thank Samit Dasgupta and Mahesh Kakde for carefully explaining their work, and for many other conversations regarding this paper. We thank Brian Conrad for helpful conversations during the writing of this paper. We thank Matthieu Romagny and Ching-Li Chai for their comments.

\

\noindent\textbf{Notation}
\

Throughout this paper, $F$ denotes a totally real number field, $\mf{n} \subset \cO_F$ an ideal, $U^+$ the totally positive units, and $U_{1, \mf{n}}$ the units congruent to $1$ modulo $\mf{n}$. We define $D_{\mf{n}} = U^+/(U_{1, \mf{n}})^2$.

In \S 3, we introduce the Shimura varieties $Sh_{G^*, K_{\lambda}}$, and $Sh_{G, K}$, while in \S 6 we introduce their algebraic counterparts, the algebraic stacks $\wt{\M}_{\mc{P}}$ and $\wt{\M}$. In both the analytic and algebraic settings, $\omega$ denotes the Hodge line bundle. The coarse moduli space of $\wt{\M}$ is denoted by $\wt{M}$, the partial toroidal compactification by $\wt{\M}^{\rm{partial}}$, and the minimal compactification by $\wt{M}^{\rm{min}}$.

The connected components of $\wt{\M}$ are indexed by ``component labels" (\S \ref{subsec:comp-label}), typically denoted by $\alpha$, while the cusps of $\wt{\M}$ are indexed by ``cusp labels" (\S \ref{subsec:cusp-labels}), typically denoted by $C$. The modules of $q$-expansions and constant terms at a cusp label $C$ are denoted $Q_{k,C}$ and $C_{k,C}$ respectively.

\tableofcontents

% LIST OF SYMBOLS
% Section 2
% \mf{n}
% G^+_{\mf{n}}, G_{\mf{n}}, I_{\mf{n}}, F_{1,\mf{n}}^+, U^+_{1,\mf{n}}, U_{1, \mf{n}}?
% H_{\mf{n}}, \chi, \chi_f, \chi_{\infty}, N_{F/\Q}
% K_f, K_v, \mc{H}^n, X_{\mf{n}}, Y_{\mf{n}}, slash operator
% \psi, M_k(\mf{n}, \C), M_k(\mf{n}, \si, \C)
% \Gamma_0{\lambda)_{\mf{n}}, \Gamma_1{\lambda)_{\mf{n}}, \mf{t}_{\lambda}, \mf{d}, f_{\lambda}
% Section 3
% G, G^*, Sh'_{G, K}, Sh_{G^*, K_{\lambda}}, Sh_{G,K}
% U^+ (should be defined earlier?)
% D_{\mf{n}}
% \mc{T}, \Lie(\mc{T})
% \omega_{\chi_{\infty}}
% Section 4
% Positivities (N, N^+_{\R})
% Component labels \alpha, Standard component label \alpha_{\lambda}
% Cusp label C = (\alpha, L), \mr{Cusp}(\mf{n})
% P_C, M_C^{\dual}, M_C, U_C, \epsilon_C 
% \psi_C, \sgn_C
% \mr{Cusp}(\mf{n})^*, \mr{Cusp}_{\infty}(\mf{n})^*, \mr{Cusp}_{p}(\mf{n})^*
% Q_{C,k}, Q_k, \mr{qexp}_C, \mr{qexp}, C_{[C],k}, C_k, \mr{const}_{[C]}, \mr{const}
% Section 5
% (A, \iota), \mc{P}, \mc{P}_{\alpha}
% \widetilde{\M}_{\mc{P}}, \widetilde{\M}', \widetilde{\M}
% \M
% Section 6
% \O
% Section 7
% R^0_C, S^0(C), S(\beta), S(\beta)^{\infty}, S_C, \wh{S}_C, S^{\infty}_C, \wh{Z}_[C], Z^{\infty}_[C]
% Too much notation here
% \wt{\M}^{\partial}, \wt{\M}^{\mr{tor}}, \wt{M}^{partial}, \wt{f}
% \omega_{\mr{min}}
% \mc{L}_k
% Section 8
% Some notation for group ring forms...

\section{Analytic Hilbert modular forms for $G$}\label{sec:hmf}

\subsection{Class groups}

We consider the groups $G^+_{\mf{n}} = \Cl^+(F, \mf{n})$ and $G_{\mf{n}} = \Cl(F, \mf{n})$, the narrow  and wide ray class groups of conductor $\mf{n}$, respectively.

Let $I_{\mf{n}}$ denote the group of prime-to-$\mf{n}$ fractional ideals $N \subset F$. Let $(F^*_{1, \mf{n}})^+\subset F^*$ denote the subgroup of totally positive elements, coprime to $\mf{n}$, and $\equiv 1 \pmod{\mf{n}}$. Let $U_{1, \mf{n}}^+$ denote the group of totally positive units of $\mathcal{O}_F$ congruent to $1 \pmod{\mf{n}}$.  There is an exact sequence
\[ 0 \to U_{1, \mf{n}}^+ \to (F^*_{1, \mf{n}})^+ \to I_{\mf{n}} \to G^+_{\mf{n}} \to 0. \]

Let $I_{p\mf{n}}$ denote the group of prime-to-$p\mf{n}$ fractional ideals $N \subset F$. Let $(F^*)' \subset F^*$ denote the subgroup of totally positive elements, coprime to $p\mf{n}$, and $\equiv 1 \mod \mf{n}$. There is an exact sequence
\[ 0 \to U_{1, \mf{n}}^+ \to (F^*)' \to I_{p\mf{n}} \to G^+_{\mf{n}} \to 0. \]

\subsection{Algebraic Hecke characters}\label{subsec:hecke-char}

Let $S_{\infty}$ denote the set of infinite places. For any place $v \not\in S_\infty$, let $\cO_{v, \fn, 1}^* \subset \cO_v^*$ denote the subgroup of elements congruent to $1 \pmod{\fn \cO_v}$.  Of course, this is just $\cO_v^*$ if $v \nmid \fn$.
Consider the complex Lie group \[ H_{\mf{n}} := F^* \backslash (\A^f_F)^* \times (F \tensor \C)^* / \prod_{v \notin S_{\infty}} \cO_{v, \fn,1}^* = F_{1, \mf{n}}^* \backslash (I_{\mf{n}} \times (F \tensor \C)^*). \] 

There is an exact sequence
\[ 0 \to (\C \tensor F)^*/U_{1,\mf{n}} \to H_{\mf{n}} \to G_{\mf{n}} \to 0. \]

Consider an (algebraic) Hecke character of $F$, i.e. a holomorphic character $\chi \from H_{\mf{n}} \to \C^*.$
The restriction \[ \chi_{\infty} := \chi|_{(\C \tensor F)^*} \from (\C \tensor F)^* \to \C^* \] must be $U_{1,\mf{n}}$-invariant. In particular, $\chi_{\infty} = N_{F/\Q}^{k}$ for some $k \in \Z$, for the norm character $N_{F/\Q} \from (\C \tensor F)^* \to \C^*$.
Define $\chi_f \from G^+_{\mf{n}} \to \C^*$ by \begin{equation} \chi_f(\mf{a}) := \chi|_{I_{\mf{n}}}(\mf{a}) \cdot |N_{F/\Q}(\mf{a})|^{k}. \end{equation}

\begin{lemma}
The Hecke characters $\chi \from H_{\mf{n}} \to \C^*$ are in bijection with the pairs \[ (k \in \Z, \chi_f \from G^+_{\mf{n}} \to \C^*) \] such that $\chi_f((\alpha)) =  \sgn(N_{F/\Q}(\alpha))^k $ for all $\alpha \in F^*_{1, \mf{n}}$.
\end{lemma}
Note that this condition implies that $N_{F/\Q}(U_{1,\mf{n}})^k = 1$,  since $\chi_f(U_{1, \mf{n}}) = 1$.
\begin{proof}

As the restriction of $\chi$ to $(\C \tensor F)^*/U_{1,\mf{n}}$ is a power of the norm, $\chi$ factors through the group $H_{\mf{n}}'$, defined to be the push-out of $H_{\mf{n}}$ along $N_{F/\Q}$:
\[\begin{tikzcd}
0 \arrow{r} & (\C \tensor F)^*/U_{1,\mf{n}} \arrow{r} \arrow{d}{N_{F/\Q}} & H_{\mf{n}} \arrow{r} \arrow{d} & G_{\mf{n}} \arrow{r} \arrow{d} & 0 \\
0 \arrow{r} & (\C \tensor F)^* \arrow{r} & H'_{\mf{n}} \arrow{r} & G_{\mf{n}} \arrow{r} & 0.
\end{tikzcd}
\]

There is an isomorphism \[ H'_{\mf{n}} =  (I_{\mf{n}} \times \C^*)/F_{1, \mf{n}}^* \isom (G^+_{\mf{n}} \times \C^*)/F_{1, \mf{n}}^*, \]
defined by $(\mf{a}, z) \mapsto (\mf{a}, |N_{F/\Q}(\mf{a})|^{-1} z)$. Note that, for $\alpha \in F^*_{1, \mf{n}}$, the isomorphism sends $(\alpha, N_{F/\Q}(\alpha))$  to $(\alpha, \sgn(N_{F/\Q}(\alpha)))$. Therefore the algebraic characters of $H'_{\mf{n}}$ are in bijection with the algebraic characters $G^+_{\mf{n}} \times \C^* \to \C^*,\ (\mf{a}, z) \mapsto \chi_f(\mf{a}) z^k$, subject only to the requirement that \[ \chi_f((\alpha)) \sgn(N_{F/\Q}(\alpha))^k = 1 \] for all $\alpha \in F_{1, \mf{n}}^*$. 

\end{proof}

We say that such a character $\chi_f$ is \emph{compatible} with the weight $k$. Given any character $\chi_f \from G \to \C^*$, we say $\sgn(\chi_f) = (-1)^k$ if $\psi$ is compatible with the weight $k$.

\begin{definition}
A weight $k$ is \textbf{admissible} if $N_{F/\Q}(U_{1,\mf{n}})^k = 1$. Equivalently, $k$ is admissible if there exists a character $\chi_f$ compatible with $k$.
\end{definition}

\subsection{Adelic Hilbert modular forms}\label{subsec:adelic-hmf}

Define $K_f = \prod_{v \notin S_{\infty}} K_v \subset \GL_2(\wh{\cO}_F)$ to be the pre-image, under reduction mod $\mf{n}$, of $\left\lbrace \mat{*}{*}{0}{1} \right\rbrace \subset \GL_2(\wh{\cO}_F/\mf{n})$.

Consider the space \begin{equation}X_{\mf{n}} := \GL_2^+(F) \backslash \GL_2(\A_F^f) \times \GL_2^+(F \tensor \R)/K_f.\end{equation}  

Let  $\underline{\mr{Isom}}_{F \tensor \R}((F \tensor \R)^2, F \tensor \C)$ denote the set of $(F \tensor \R)$-isomorphisms $(F \tensor \R)^2 \isom F \tensor \C$. This set has a natural complex-analytic structure. Fixing an isomorphism $(F \tensor \R)^2 \isom F \tensor \C$ by sending $e_1 \to 1 \tensor i$ and $e_2 \to 1 \tensor 1$, one obtains an isomorphism $\GL_2(F \tensor \R) \isom \underline{\mr{Isom}}_{F \tensor \R}((F \tensor \R)^2, F \tensor \C)$. In this way, $\GL_2^+(F \tensor \R)$ is given a complex-analytic structure. Moreover, the group $(F \tensor \C)^*$ acts holomorphically on $\GL^+_2(F \tensor \R)$, via scaling on $F \tensor \C$.  

Viewing $(\A^f_F)^*$ as scalar matrices in $\GL_2(\A_F^f)$, the group $H_{\mf{n}}$ acts holomorphically on $X_{\mf{n}}$ via right-multiplication.

There is an identification  \[ \mc{H}^n = \{ z \in (\C \tensor F)^*: \mr{Im}(z) \in (\R \tensor F)^+ \}, \] where $n = [F : \Q]$, using a fixed ordering of the infinite places of $F$. We will never reference this ordering, and always work with the more canonical space on the right of this equation. The group $\GL_2^+(F \tensor \R)$ acts on $\mc{H}^n$ via fractional linear transformations.

We also consider the space \begin{equation} Y_{\mf{n}} := \GL_2(\A_F^f) \times \mc{H}^n/K_f. \end{equation} For a function $f \from Y_{\mf{n}} \to \C$, an element $g = \mat{a}{b}{c}{d} \in \GL^+_2(F \tensor \R)$, and weight $k \in \N$, define the \textbf{weight $k$ slash operator} $f|_{k, g}(g_f, z) = N_{F/\Q}(c z + d)^{-k} f(g_f, \frac{a z + b}{c z + d})$. There is a bijection between \[ \{ f \from Y_{\mf{n}} \to \C : f|_{k, g} = f\ \forall \ g \in \GL_2^+(F) \} \] and  \[ \{ \varphi \from X_{\mf{n}} \to \C :  \varphi(g x^{-1}) = N_{F/\Q}^{-k}(x) \varphi(g)\ \forall \ g \in \GL_2^+(\A_F),\ x \in (F \tensor \C)^* \}, \] given by sending $f$ to the function $\varphi(g_f, g_{\infty}) := f|_{k, g_{\infty}}(g_f, (i, \ldots, i))$.

Fix a character $\chi \from H_{\mf{n}} \to \C^*$, corresponding to the pair $(k \in \N, \chi_f \from G_{\mf{n}}^+ \to \C^*)$.

\begin{definition}
A \textbf{holomorphic Hilbert modular form} of level $\Gamma_1(\mf{n})$, parallel \textbf{weight} $k$, and \textbf{central character} $\chi|_{\A_F^*}$ is a function $\varphi \from X_{\mf{n}} \to \C$ such that $\varphi(g \cdot h^{-1}) = \chi(h) \varphi(g)$ for all $g \in \GL_2^+(\A_F)$, $h \in H_{\mf{n}}$. In particular, there exists a function $f \from Y_{\mf{n}} \to \C$ corresponding to $\varphi$ via the weight $k$ slash operation. We further require that for fixed $g_f \in \GL_2(\A_F^f)$ the function $z \mapsto f(g_f, z)$ is \emph{holomorphic}, and  of \emph{moderate growth} as $N_{F/\Q}(\mr{Im}(z)) \to \infty$.
\end{definition}

\begin{remark}
We will not discuss the growth condition until \S\ref{subsec:q-cusplabel}, where it will be precisely the condition required to ensure that $q$-expansions have only non-negative terms. In fact, when $F \neq \Q$, the Koecher principle implies that the growth condition is automatic.
\end{remark}

These functions are said to have \textbf{nebentype} $\psi = \chi_f^{-1}$. We denote the space of such functions by $M_k(\mf{n}, \psi, \C)$. Similarly, define $M_k(\mf{n}, \C)$ as the as the set of functions $\varphi$ such that $\varphi(g \cdot h^{-1}) = \chi_{\infty}(h) \varphi(g)$ for all $h \in (F \tensor \C)^*$, with the same additional conditions of holomorphy and growth. Since $M_k(\mf{n}, \C)$ is finite-dimensional, any such function decomposes into eigenfunctions for the larger group $H_{\mf{n}}$, and hence
\[ M_k(\mf{n}, \C) = \oplus_{\psi} M_k(\mf{n}, \psi, \C). \]

\begin{remark}\label{rmk:holom-unit}
For $k > 0$, the central characters $\chi|_{\A^*_F}$ we are considering are not unitary, as $\chi_{\infty}|_{(F \tensor \R)^*} = N_{F/\Q}^{k}$. The space of holomorphic modular forms with central character $(x_f, x_{\infty}) \mapsto \chi(x_f, x_{\infty})$ can be related to a similarly defined space of functions with unitary central character $(x_f, x_{\infty}) \mapsto \chi(x_f, x_{\infty}) N_{F/\Q}(x_{\infty})^k$, by sending $\varphi$ to the function \[ (g_f, g_{\infty}) \mapsto N_{F/\Q}(\det(g_{\infty}))^{-k/2} \varphi(g_f, g_{\infty}). \] In terms of functions on $Y_{\mf{n}}$, this unitary normalization corresponds to using a different slash operation $f||_{k, g} := N_{F/\Q}(\det(g))^{k/2} f|_{k,g}$.

We denote the respective spaces of functions by $M^{\mr{hol}}_k(\mf{n}, \C)$ and $M^{\mr{uni}}_k(\mf{n}, \C)$, and the  twist defined above yields an isomorphism $M^{\mr{hol}}_k(\mf{n}, \C) \isom M^{\mr{uni}}_k(\mf{n}, \C)$. 
\end{remark}

\begin{remark}As we have mentioned, $X_{\mf{n}}$ is naturally a complex-analytic space, and holomorphic Hilbert modular forms are (certain) holomorphic functions on $X_{\mf{n}}$. One advantage of studying the complex-analytic space $X_{\mf{n}}$, as opposed to the Shimura variety $Sh_{G,K}$ discussed in \S\ref{subsec:shimura-var} below, is that Hilbert modular forms of \emph{non-parallel} weight are naturally sections of certain holomorphic line bundles on $X_{\mf{n}}$, while only those HMFs of parallel weight are sections of holomorphic line bundles on $Sh_{G,K}$. We will not consider non-parallel weights in this paper.
\end{remark}

\subsection{Classical Hilbert modular forms}\label{subsec:classical-hmf}

For each class $\lambda$ in the narrow class group $\Cl^+(F)$, we choose a representative fractional ideal $\ft_\lambda$.  Let $\fn \subset \cO_F$ be an ideal, and assume that the representative ideals $\ft_\lambda$ have been chosen to be relatively prime to $\fn$.  Define the groups
\begin{align*}
 \Gamma_{0, \lambda}(\fn) &= \left\{ \mat{a}{b}{c}{d} \in \GL_2^+(F): a, d \in \cO_F, c \in \ft_\lambda \fd \fn, b \in (\ft_\lambda \fd)^{-1}, ad-bc \in \cO_F^* \right\}, \\ 
  \Gamma_{1, \lambda}(\fn) &= \left\{ \mat{a}{b}{c}{d} \in \Gamma_{0, \lambda}(\mf{n}): d \equiv 1 \pmod{\fn} \right\}.
 \end{align*} 
Here $\fd$ denotes the different ideal of $F$. 

By the strong approximation theorem, \[ X_{\mf{n}} =\coprod_{\lambda \in \Cl^+(F)} \Gamma_{1, \lambda}(\mf{n}) \backslash \GL_2^+(F \tensor \R). \] More precisely, there exist elements $g_{\lambda} := \mat{x_{\lambda}}{0}{0}{1} \in \GL_2(\A_F^f)$ such that:
\begin{itemize}
\item The subgroup of $\GL_2^+(F)$ preserving $\prod_v \{ g_{\lambda} \} K_{v} \subset \GL_2(\A_F^f)$ equals $\Gamma_{1,\lambda}(\mf{n})$.
\item The inclusion $\coprod_{\lambda} \Gamma_{1, \lambda}(\mf{n}) \backslash (\{ g_{\lambda} \} K_{v} \times \GL_2^+(F \tensor \R)) \into X_{\mf{n}}$ is a bijection.
\end{itemize}

These elements $x_{\lambda} = (x_{\lambda, v})$ are defined by requiring, for all finite places $v$, that \[ (\mf{t}_{\lambda} \mf{d})^{-1} \cO_{F,v} = x_{\lambda, v}\cO_{F,v} \] as fractional ideals of $\cO_{F_v}$.

This construction implies that $M_{k}(\mf{n}, \C)$ is in bijection with tuples of functions \[ (f_{\lambda} \from \mc{H}^n \to \C)_{\lambda \in \Cl^+(F)} \] such that:
\begin{enumerate}
\item $f_{\lambda}$ is holomorphic,
\item $f_{\lambda}|_{k, \alpha} = f_{\lambda}$ for all $\alpha \in \Gamma_{1, \lambda}(\mf{n})$,
\item $f_{\lambda}|_{k, \alpha}$ is of moderate growth for all $\alpha \in \GL_2^+(F)$. 
\end{enumerate} 

\begin{remark}
We further elaborate on the relationship between the ``holomorphic" and ``unitary" normalizations of the central character mentioned in Remark \ref{rmk:holom-unit}.

Shimura \cite{shim} defines Hilbert modular forms of level $\Gamma_1(\mf{n})$ via a similar definition instead using the modified slash operator $||_{k,g}$, and so the HMFs appearing there agree with what we refer to as the unitary normalization, $M_k^{\mr{uni}}(\mf{n}, \C)$. The two different slash operators $f|_{k,g}$ and $f||_{k,g}$ agree for $g \in \Gamma_{1, \lambda}(\mf{n})$, and so the resulting spaces of classical modular forms, $M_k^{\mr{hol}}(\mf{n}, \C) = M_k(\mf{n}, \C)$ and $M_k^{\mr{uni}}(\mf{n}, \C)$, are equal. However, we will never use this equality, as it is not the isomorphism described in Remark \ref{rmk:holom-unit}. The correct isomorphism \[ M_k^{\mr{hol}}(\mf{n}, \C) \isom M_k^{\mr{uni}}(\mf{n}, \C) \] is given in classical terms by \[ (f_{\lambda}) \mapsto (|N_{F/\Q}(\mf{t}_{\lambda} \mf{d})|^{k/2} f_{\lambda}). \] This isomorphism identifies our $M_k^{\mr{hol}}(\mf{n}, \psi, \C)$ with Shimura's $M_k^{\mr{uni}}(\mf{n}, \psi, \C)$.
\end{remark}

\section{Analytic Hilbert modular varieties}\label{sec:shimura-var}
 \subsection{Hilbert modular varieties}\label{subsec:shimura-var}
Consider the reductive algebraic group $G = \Res_{F/\Q} \GL_2$.
The determinant defines a map \[ \det \from G \to \Res_{F/\Q} \G_m. \] By taking the preimage of $\G_m \subset \Res_{F/\Q} \G_m$, define
 \[G^* := {\det}^{-1}(\G_m) \subset G. \] We now introduce certain Shimura varieties associated to $G$ and $G^*$. The $G^*$ Shimura varieties are not needed for the analytic theory, and aside from defining them shortly, will not arise for the remainder of \S\ref{sec:shimura-var}-\S\ref{sec:comp-and-cusp}.

Recall the subgroups $K_v \subset \GL_2(F_v)$ defined in \S \ref{subsec:adelic-hmf}, as well as $K_f = \prod K_v.$ Given the choice of fractional ideal $\mf{t}_{\lambda}$ representing $\lambda \in \Cl^+(F)$, there are associated subgroups \[ K_{f, \lambda} := g_{\lambda} K_f g_{\lambda}^{-1} \cap \mr{det}^{-1}(\A_Q^*) \subset \GL_2(\A_F) \cap \mr{det}^{-1}(\A_Q^*) = G^*(\A_\Q) \] where $g_{\lambda} \in \GL_2(\A_F)$ are the elements defined in \S \ref{subsec:classical-hmf}.

These level structures $K_f$ and $K_{f, \lambda}$ define Shimura varieties
\begin{align*}
Sh'_{G, K} &= \GL_2(F)^+ \backslash (\GL_2(\A^f_F) \times \cH^n)/K_f\\
&\isom \coprod_{[\lambda] \in \Cl^+(F)} \Gamma_{1, \lambda}(\mf{n}) \backslash \cH^n. \\
Sh_{G^*, K_{\lambda}} &= G^*(\Q)^+ \backslash (G^*(\A^f_\Q) \times \cH^n)/K_{f, \lambda} \\
 &\isom (\Gamma_{1, \lambda}(\mf{n}) \cap \mr{det}^{-1}(\Q)) \backslash \cH^n.
\end{align*}

We consider these not as algebraic varieties, but as stacks over the category of complex-analytic spaces (\S\ref{subsec:stacks}), defined as the stack quotient of $(\GL_2(\A^f_F) \times \cH^n)/K_f$ by $\GL_2(F)^+$ (resp. $ (G^*(\A^f_\Q) \times \cH^n)/K_{f, \lambda}$ by $G^*(\Q)^+$). In particular, the inertia group at a point $x$ equals the stabilizer of $x$ in $\Gamma_{1, \lambda}(\mf{n})$ (resp. $\Gamma_{1, \lambda}(\mf{n}) \cap \mr{det}^{-1}(\Q)$). The stack $Sh'_{G,K}$ is not \emph{rigid}, as every point has inertia group containing $Z := U_{1, \mf{n}} = $ center of $\Gamma_{1, \lambda}(\mf{n})$. For this reason, it will also be useful to consider the rigidified stack 
\begin{equation}
Sh_{G, K} := \coprod_{[\lambda] \in \Cl^+(F)} P\Gamma_{1, \lambda}(\mf{n}) \backslash \cH^n,
\end{equation}
where $P\Gamma_{1, \lambda}(\mf{n}) := \Gamma_{1, \lambda}(\mf{n})/Z$.

There is a surjective covering map \[ \coprod_{[\lambda] \in \Cl^+(F)} Sh_{G^*, K_{\lambda}} \to Sh_{G,K}. \] Since $U^+ = \det(\Gamma_{1, \lambda}(\mf{n}))$, it is easy to see that the group $U^+$ acts on $Sh_{G^*, K_{\lambda}}$, so that
\begin{equation}
Sh'_{G,K} = \coprod_{[\lambda] \in \Cl^+(F)} Sh_{G^*, K_{\lambda}}/U^+.
\end{equation}

Passing to the rigidification, we find that 
\begin{equation} \label{e:shgk}
Sh_{G, K} = \coprod_{\lambda} Sh_{G^*, K_{\lambda}}/D_{\mf{n}},
\end{equation}
where $D_{\mf{n}} := U^+/(U_{1, \mf{n}})^2$. 

\begin{remark}
The stack $Sh_{G^*, K_{\lambda}}$ depends on the fractional ideal $\mf{t}_{\lambda}$: if $[\lambda_1] = [\lambda_2] \in \Cl^+(F)$, there is a \emph{non-canonical} isomorphism $Sh_{G^*, K_{\lambda_1}} \isom Sh_{G^*, K_{\lambda_2}}$. However, two such isomorphisms differ by the action of $U^+$, and so $Sh_{G^*, K_{\lambda}} / U^+$ and $Sh_{G^*, K_{\lambda}} / D_{\mf{n}}  $ depend only on the class $[\lambda] \in \Cl^+(F)$ up to canonical isomorphism.
\end{remark}

\subsection{Moduli interpretation}\label{subsec:moduli}

Let $H$ be a locally free $\cO_F$-module of rank $2$, and let \[ \underline{\mr{Isom}}_{\R \tensor_{\Z} \cO_F}(\C \tensor_{\Z} F, H \tensor \R) \] denote the space parametrizing $\R \tensor \cO_F$-module isomorphisms $h_{\infty} \from \C \tensor F \isom H \tensor \R$. This space has a natural complex-analytic structure.

\begin{theorem}\label{thm:moduli-lattice}
The groupoid $Sh_{G,K}'(\C)$ is equivalent to the groupoid whose objects are triples $(H, \gamma, \bar{h}_{\infty})$, where:
\begin{itemize}
\item $H$ a locally free, rank 2 $\cO_F$-module;
\item $\gamma \from \cO_F/\mf{n}\cO_F \into H/\mf{n}H$ is an injective $\cO_F$-module homomorphism;
\item $\bar{h}_{\infty}$ is an  equivalence class modulo $(\C \tensor F)^*$ of  $\R \tensor_{\Z} \cO_F$-module isomorphisms \[ h_{\infty} \from \C \tensor_{\Z} F \isom H \tensor \R,\]  
\end{itemize} 
and whose morphisms are $\cO_F$-module isomorphisms $H_1 \isom H_2$ compatible with the extra structures.

More generally, for simply-connected complex analytic spaces $S$, the groupoid $Sh_{G,K}'(S)$ is equivalent to the groupoid whose objects are triples \[ (H, \gamma, h_{\infty, S} \from S \to \underline{\mr{Isom}}_{\R \tensor_{\Z} \cO_F}(\C \tensor_{\Z} F, H \tensor \R)/(\C \tensor F)^*). \]
\end{theorem}

\begin{proof}
As this type of result is standard (for example, \cite{rapoport} 1.2.5), we will limit our discussion to the case $S = \C$, in order to define terminology and conventions for later. 

Note that the groupoid of triples $(H, \gamma, \bar{h}_{\infty})$ is equivalent to the groupoid of triples $(H, \gamma, \bar{h}_{\infty})$ such that $H$ is required to be a lattice in $F \oplus F$, i.e. $H \tensor_{\cO_F} F = F \oplus F$. We will prove that this latter groupoid is \emph{equal} to the groupoid quotient (see \S\ref{subsec:quotient-stacks}) of the set $(\GL_2(\A^f_F) \times \cH^n)/K_f$ by the left-action of $\GL_2(F)^+$.

Consider a triple $(H, \gamma, h_{\infty})$ such that $H \tensor_{\cO_F} F = F \oplus F$. For each finite place $v$, choose $h_v \from (F_v)^2 \to (F_v)^2$ so that: 
\begin{enumerate}
\item $h_v(\cO_{F,v}^2) = H \tensor \cO_{F,v}$;
\item $\gamma \tensor \cO_{F,v}$ equals the composition 
\begin{tikzcd}
0 \oplus \cO_{F,v}/\mf{n} \arrow[r] & \cO_{F,v}^2 \arrow[r,"{h_v \!\!\mod \mf{n}}"] & H/\mf{n} 
\tensor \cO_{F,v}.
\end{tikzcd}
\end{enumerate}

Together with the element $h_{\infty}$, this defines an element of $\GL_2(\A_F^f) \times \GL_2(F \tensor \R)$. Note that we view $h_{\infty}$ as a matrix via the identification $i(\R \tensor F) \oplus \R \tensor F  = \C \tensor_{\Z} F$, and that this identification also gives an inclusion $(\C \tensor F)^* \subset \GL_2(F \tensor \R)$.

This element $(h_v, h_{\infty})$ is well-defined modulo the right-action of $\prod_{v \not S_{\infty}} K'_v$, where $K_v'$ is the subgroup of $\GL(H \tensor \cO_{F,v})$ fixing $\gamma \tensor \cO_{F,v}$. This defines a bijection \[ \{ (H, \gamma, h_{\infty}): H \tensor F = F \oplus F \} \isom \GL_2(\A_F)/\prod_{v \notin S_{\infty}} K_v'. \] However, this is \emph{not} the correct bijection to use - we must replace the matrices $h_v, h_{\infty}$ with their \emph{inverse-transpose}, which we will denote $g_v, g_{\infty}$. We obtain a bijection 
\[ \{ (H, \gamma, h_{\infty}): H \tensor F = F \oplus F \} \isom \GL_2(\A_F)/K_f, \] where $K_f = \prod_{v \notin S_{\infty}} K_v$,
as the inverse-transpose of $K_v' \subset \GL_2(F_v)$ equals the subgroup $K_v$ defined \S \ref{subsec:adelic-hmf}.

There is a map $\GL_2(\A_F)/K_f \to \GL_2(\A_F^f) \times \mc{H}_{\pm}^n /K_f$, where $g \in \GL_2(F \tensor \R)$ is sent to $g \cdot (i,\ldots, i) \in \mc{H}_{\pm}^n$. We obtain a map of sets
\[ \{ (H, \gamma, \bar{h}_{\infty}): H \tensor F = F \oplus F \} \longrightarrow \GL_2(\A^f_F) \times \mc{H}_{\pm}^n/K_f. \] 

If $h_{\infty} \from \C \tensor F \isom (F \oplus F) \tensor \R$ has $h_{\infty}^{-1}(e_1) = z_1 \in (\C \tensor F)^*$ and $h_{\infty}^{-1}(e_2) = z_2 \in (\C \tensor F)^*$, that $g_{\infty} \cdot (i,\ldots, i) = z_1/z_2 \in (\mc{H}_{\pm})^n$. 

Isomorphisms of triples $(H, \gamma, \bar{h}_{\infty})$ correspond to the action of $\GL_2(F)$ on the set $\GL_2(\A^f_F) \times \mc{H}_{\pm}^n/K_f$. Hence, considering the groupoid $\mc{C}$ whose objects are \[ \{ (H, \gamma, \bar{h}_{\infty}): H \tensor F = F \oplus F \},\] we obtain an isomorphism of groupoids
\[ \mc{C} \isom \GL_2(F) \backslash (\GL_2(\A^f_F) \times \mc{H}_{\pm}^n/K_f). \]
\end{proof}

In fact, the  proof above gives:
\begin{theorem}\label{thm:lie-moduli}
The set $X_{\mf{n}}(\C)$ is equivalent to the groupoid whose objects are triples \[ (H, \gamma, h_{\infty} \from \C \tensor_{\Z} F \isom H \tensor \R), \] and whose morphisms are $\cO_F$-module isomorphisms $H_1 \isom H_2$ compatible with the extra structures.

More generally, for simply-connected complex analytic spaces $S$, the set $X_{\mf{n}}(S)$ is equivalent to the groupoid whose objects are triples $(H, \gamma, h_{\infty, S} \from S \to \underline{\mr{Isom}}_{\R \tensor_{\Z} \cO_F}(\C \tensor_{\Z} F, H \tensor \R))$.
\end{theorem}

We reinterpret this lattice-theoretic moduli in terms of complex tori.

\begin{definition}
An \textbf{RM torus} is a compact complex torus $T$ of complex dimension $n$, equipped with a map $\iota \from \cO_F \into \End(T)$.
\end{definition}

\begin{remark}
Any RM torus is of the form $(\C \tensor_{\Z} \cO_F)/H$, for $H$ a locally free, rank 2 $\cO_F$-module such that $H \tensor_{\Z} \R \isom \C \tensor_{\Z} \cO_F$ as $\R \tensor_{\Z} \cO_F$-modules.
\end{remark}

Define a complex-analytic stack $\mc{N}'$ as follows. For each complex-analytic space $S$, $\mc{N}'(S)$ equals the groupoid whose objects are pairs $(T, \gamma)$, where $T \to S$ is a family of RM tori, and $\gamma \from (\cO_F/\mf{n})_S \into T[\mf{n}]$ is an injection of $\cO_F$-modules over $S$. 

There is a universal complex torus $\mc{T} \to \mc{N}'$, with Lie algebra $\Lie(\mc{T}) \to \mc{N}'$. The sheaf of sections of this vector bundle sends pairs $(T \to S, \gamma)$ to the locally free $\cO_S$-module $\Lie_S(T)$. There is an open substack $\Lie(\mc{T})^* \subset \Lie(\mc{T})$, defined by triples $(T \to S, \gamma, v)$ such that $(\cO_S \tensor F)v = \Lie_S(T)$.

\begin{prop}
The stack $Sh'_{G, K}$ is equivalent to $\mc{N}'$, and the complex-analytic space $X_{\mf{n}}$ is equivalent to $\Lie(\mc{T})^*$.
\end{prop}
\begin{proof}
We define an equivalence $X_{\mf{n}}(\C) \isom \Lie(\mc{T})^*(\C)$. The same argument would work for any simply-connected complex-analytic space $S$ in place of $\C$, giving the result.

We first apply Theorem \ref{thm:lie-moduli}, so that instead of $X_{\mf{n}}(\C)$, we may consider the groupoid of triples $(H, \gamma, h_{\infty})$. Given a triple $(H, \gamma, h_{\infty})$, define a triple $(T = H \tensor \R/H, \gamma, v) \in \Lie(\mc{T})^*(\C)$ as follows. The isomorphism $h_{\infty} \from\C \tensor_{\Z} F \isom H \tensor \R$ makes $H \tensor \R$ into a $\C \tensor F$-module. In other words, it endows $H \tensor \R$ with a complex structure such that the action of $F$ on $H \tensor \R$ is holomorphic. Hence $H \tensor \R/H$ is an RM torus $T$, with $\Lie(T) = H \tensor \R$. Morover, $v := h_{\infty}(1 \tensor 1)$ defines an element $v \in \Lie(T)^*$.
\end{proof}

\subsection{The action of $H_{\mf{n}}$}\label{subsec:lie-compare}

We define a (non-strict) action of $I_{\mf{n}}$ on $\mc{N}'$ by $N \cdot (T \to S, \gamma) = (T \tensor N \to S, \gamma_N)$.
\begin{itemize}
\item If $S = \star$, then $T \tensor_{\cO_F} N := (\Lie(T) \tensor_{\cO_F} N)/(H \tensor_{\cO_F} N)$. This construction extends in an obvious manner to families.
\item The map $\gamma_N$ is defined as $\cO_F/\mf{n} \isom N/\mf{n} \nmto{\gamma \tensor N} H/\mf{n} \tensor N = (H \tensor N)/\mf{n}$, where the first map comes from the fact that $N \tensor \cO_{F, \mf{n}} = \cO_{F, \mf{n}}$ as submodules of $F_{\mf{n}}$.
\end{itemize}

We extend this definition to a (non-strict) action of $I_{\mf{n}} \times (F \tensor \C)^*$ on the stack $\Lie(\mc{T})$. For a complex analytic space $S$, consider triples $(T \to S, \gamma, v)$, with $v \in \Lie_S(T)$. The group $(F \tensor \cO_S)^*$ acts on $v \in \Lie_S(T)$, keeping $(T, \gamma)$ fixed. For any $N \in I_{\mf{n}}$, define $N \cdot (T, \gamma, v) := (T \tensor N, \gamma_N,  v_N)$. The element $v_N \in \Lie(T \tensor N)$ corresponds to $v$ under $\Lie(T \tensor N) \isom \Lie(T) \tensor N \isom \Lie(T)$, where the second isomorphism is induced by $N \tensor \Q = F \tensor \Q$. 

\begin{remark}
This non-strictness arises because the stack $\Lie(\mc{T})^*$ is not \emph{equal} to the complex analytic space $X_{\mf{n}}$ on which the group $F_{1,\mf{n}}^* \backslash (I_{\mf{n}} \times (F \tensor \C)^*)$ acts, merely equivalent. In fact, $\Lie(\mc{T})^*$ is a non-strict $(I_{\mf{n}} \times (F \tensor \C)^*, F^*_{1,\mf{n}})$-stack in the sense of \S\ref{subsec:grp-act}. To define a trivialization of the action of $F^*_{1, \mf{n}}$, we must specify, for $N = (\alpha)$, an isomorphism $(T \tensor N, \gamma_N, \alpha \cdot v_N) \isom (T, \gamma, v)$. This is the isomorphism given by ``division by $\alpha$".
\end{remark}

\begin{lemma}
Under the equivalence $\Lie(\mc{T})^* \isom X_{\mf{n}}$, the action of $(N , \alpha_{\infty}) \in I_{\mf{n}} \times (F \tensor \C)^*$ is identified with the action, via right-multiplication, of the the matrix \[ \left(\mat{\alpha_{v}^{-1}}{0}{0}{\alpha_{v}^{-1}}, \alpha_{\infty}^{-1}\right) \in \GL_2(\A_F^f) \times \GL_2(F \tensor \R), \] where $\alpha_{v} \cO_{F,v} = N \tensor \cO_{F,v}$ for all finite places $v$.
\end{lemma}
\begin{proof}
This follows easily from the definition of the equivalence $\Lie(\mc{T})^* \isom X_{\mf{n}}$ (see the proof of Theorem \ref{thm:moduli-lattice}). The appearance of an inverse is related to the inverse-transpose appearing in that proof.
\end{proof}

From now on, we will use the moduli-theoretic action of $I_{\mf{n}} \times (F \tensor \C)^*$ on $\Lie(\mc{T})^*$ defined above, which is the inverse of the action on $X_{\mf{n}}$ considered in \S \ref{subsec:adelic-hmf}. In particular, for this moduli-theoretic action, modular forms with central character $\chi$ (corresponding to $(k, \chi_f)$) are equal to
\[ M_k(\mf{n}, \chi_f^{-1}) = H^0(X_{\mf{n}}, \cO_{X_{\mf{n}}})_{\mr{mod}}^{\chi^{-1}}. \]

\subsection{Hodge line bundles and diamond operators}

We can view Hilbert modular forms of fixed weight as global sections of a line bundle in the following way. Given a character 
$\chi_{\infty} \from (F \tensor \C)^* \to \C^*$, 
we define a holomorphic line bundle $\omega_{\chi_{\infty}}$ on $Sh'_{G,K}$ (or more precisely, on the equivalent stack $\mc{N}'$), by defining 
 $\omega_{\chi_{\infty}}(T \to S, \gamma)$ to be the set of  functions
  \[f \from \Lie_S(T)^* \to \cO_S \text{ such that } f(w v) = \chi_{\infty}(w) f(v) \text{ for all } w \in (\cO_S \tensor F)^* . \]
  Here, for a complex analytic space $S$, we use the character $\chi_{\infty} \from (\cO_S \tensor F)^* \to \cO_S^*$ extending $\chi_{\infty} \from (\C \tensor F)^* \to \C^*$. This line bundle is $I_{\mf{n}}$-equivariant, using the action of $I_{\mf{n}}$ on $\Lie(\mc{T})$. 
  
  Global sections of this line bundle are given by \begin{equation} H^0(Sh'_{G,K}, \omega_{N_{F/\Q}^k}) = H^0(\Lie(\mc{T})^*, \cO_{\Lie(\mc{T})^*})^{N_{F/\Q}^{-k}}. \end{equation}
If we restrict to sections with moderate growth, \begin{equation}  H^0(Sh'_{G, K}, \omega_{N_{F/\Q}^k})_{\mr{mod}} = M_k(\mf{n}, \C). \end{equation}

Note that $\omega_{N_{F/\Q}^k} \isom \omega^{\tensor k}$, for $\omega$ the line bundle on $Sh_{G,K}'$ sending $(T \to S, \gamma)$ to the $\cO_S$-module $\det(\Lie_S(T))^{\dual}$. For an admissible weight $k$, the line bundle $\omega^{\tensor k}$ (on $Sh_{G, K}'$) descends to a line bundle on $Sh_{G,K}$. Hence, for any admissible weight $k$,
\begin{equation} H^0(Sh_{G, K}, \omega^{\tensor k})_{\mr{mod}} = M_k(\mf{n}, \C). \end{equation}

The action of $I_{\mf{n}}$ on $Sh'_{G,K}$ induces an action of $G^+_{\mf{n}}$ on $Sh_{G,K}$. The line bundle $\omega^{\tensor k}$ on $Sh_{G,K}$ can be made $G^+_{\mf{n}}$-equivariant by twisting the action of $I_{\mf{n}}$: \begin{equation} N \cdot (T \to S, \gamma, f) = ((T \tensor N) \to S, \gamma_N, f(v_N)/|N_{F/\Q}(N)|^{\tensor k} ). \end{equation}

The above discussion implies:
\begin{prop}\

\begin{enumerate}
\item For any admissible weight $k \in \N$, \[ M_k(\mf{n}, \C) = H^0(Sh_{G, K}, \omega^{\tensor k})_{\mr{mod}}. \]
\item For any weight $k \in \N$ and compatible $\chi_f \from G_{\mf{n}}^+ \to \C^*$, \[ M_k(\mf{n}, \chi_f^{-1}, \C) = H^0(Sh_{G, K}, \omega^{\tensor k})^{\chi_f^{-1}}_{\mr{mod}}. \]
\end{enumerate}
\end{prop}

In particular, the group $G^+_{\mf{n}}$ acts on $M_k(\mf{n}, \C)$ via the classical action of \textbf{diamond operators}, with a fractional ideal $N \in I_{\mf{n}}$ acting via Shimura's operator $S_{\mf{n}}(N)$ (\cite{shim} pg. 648).

\section{Components, cusps, and $q$-expansions}\label{sec:comp-and-cusp}
In this section, we introduce module-theoretic terminology for studying the connected components and cusps of $Sh_{G,K}$, similar to that used by Rapoport \cite{rapoport} for the cusps of $Sh_{G^*,K_{\lambda}}$.

\subsection{$\cO_F$-modules with positivity}

Consider pairs $(A, \bar{h}_A)$, where \begin{enumerate}
\item $A$ is an free, rank 1 $F \tensor \R$-module.
\item  $\bar{h}_A$ is an equivalence class of  isomorphisms $h_A \from F \tensor \R \isom A$, where equivalence is up to the 
  the action of $(F \tensor \R)^+ \isom \prod \R_{> 0}$.
\end{enumerate}

The equivalence class $\bar{h}_A$ is determined by the subset $A^+ \subset A$ induced by $h_A$. We say that $A$ is an \emph{$F \tensor \R$-module with positivity}.

\begin{definition}[\cite{rapoport}]
An \textbf{invertible $\cO_F$-module with positivity} is a pair $(N, N_{\R}^+)$, where $N$ is an invertible $\cO_F$-module and $N_{\R}^+ \subset N_{\R}$ is a positivity on $N_{\R}$.
\end{definition}

\begin{remark}
Given an invertible $\cO_F$-module $N$, $N^{\tensor 2}$ is canonically equipped with a positivity $(N^+_{\R})^{\tensor 2}$, given by choosing any notion of positivity for $N$.
\end{remark}

\begin{remark}
Given an invertible $\cO_F$-module $N$ (with positivity), the group $\Hom_{\Z}(N, \Z)$ is canonically isomorphic, as an $\cO_F$-module (with positivity), to $N^{-1} \tensor \mf{d}^{-1}$.
\end{remark}

For an invertible $\cO_F$-module $N$, define $N_{F/\Q}(N) := \det\nolimits_{\Z}(N)$. If $N$ has a positivity, then $N_{F/\Q}(N)$ has a canonical generator, $|N_{F/\Q}(N)| \in N_{F/\Q}(N)$. Of course, if $N \subset F$ is a fractional ideal,  then $|N_{F/\Q}(N)| \in N_{F/\Q}(N) \subset \Q$ is the norm of $N$.

\subsection{Component labels}\label{subsec:comp-label}

There is an orientation of $\C$ as an $\R$-vector space given by $\R_{> 0}(i \wedge 1) \subset \det\nolimits_{\R}(\C)$. Using this, one may define a positivity $\det\nolimits_{\R \tensor F}(\C \tensor F)^+ \subset \det\nolimits_{\R \tensor F}(\C \tensor F)$ on $ \det\nolimits_{\R \tensor F}(\C \tensor F)$.

A triple $(H, \gamma, \bar{h})$ as in \S \ref{subsec:moduli} determines a triple \[ (H, \gamma, {\det}_{\cO_F}(H)^+_{\R}).\]  Here $(\det\nolimits_{\cO_F}(H), \det\nolimits_{\cO_F}(H)^+_{\R})$ is an invertible $\cO_F$-module with positivity,  defined by transporting the positivity from $\det\nolimits_{\R \tensor F}(\C \tensor F)$ along $h_{\infty} \from \C \tensor F \isom H \tensor \R$.

\begin{definition}
A \textbf{component label} is a triple $\alpha = (H, \gamma, \det\nolimits_{\cO_F}(H)^+_{\R})$,  where
\begin{itemize}
\item $H$ is a locally free rank 2 $\cO_F$-module.
\item $\gamma \from \cO_F/\mf{n}\cO_F \into H/\mf{n}H$ is an injective $\cO_F$-module homomorphism,
\item $\det\nolimits_{\cO_F}(H)^+_{\R} \subset \det\nolimits_{\cO_F}(H)_{\R}$ is such that $(\det\nolimits_{\cO_F}(H), \det\nolimits_{\cO_F}(H)^+_{\R})$ is an invertible $\cO_F$-module with positivity.
\end{itemize} 
\end{definition}

There is a notion of isomorphism of component labels, given by isomorphisms $H_1 \isom H_2$ compatible with the additional structures.

\begin{lemma}\label{lem:adelic-cmpt}\

\begin{enumerate}
\item Component labels such that $H \tensor F = F \oplus F$ are in bijection with $\GL_2(\A_F^f)/K_f$.
\item Isomorphism classes of component labels are in bijection with $\GL^+_2(F) \backslash \GL_2(\A_F^f)/K_f$.
\item The automorphism group of a component label equals its stabilizer in $\GL_2^+(F)$.
\item The connected components of $Sh_{G,K}$ are in bijection with the isomorphism classes of component labels.
\end{enumerate}
\end{lemma}
\begin{proof}
The same proof as in Theorem \ref{thm:moduli-lattice} works. Note the inverse transpose that occurs in the construction, as it will be relevant to Lemma \ref{lem:normalize}.
\end{proof}

\begin{definition}
The \textbf{standard component label} associated to $\lambda \in \Cl^+(F)$ equals \[ \alpha_{\lambda} = (H_{\lambda}, \gamma_{\lambda}, \det\nolimits_{\cO_F}(H_{\lambda})^+), \] where:
\begin{enumerate}
\item $H_{\lambda} = \mf{t}_{\lambda} \mf{d} \oplus \cO_F$.
\item $\gamma_{\lambda} \from \cO_F/\mf{n} \into (\mf{t}_{\lambda} \mf{d})/\mf{n} \oplus \cO_F/\mf{n}$ the inclusion of the second factor.
\item $\det\nolimits_{\cO_F}(H_{\lambda})^+ = (\mf{t}_{\lambda} \mf{d})^+$.
\end{enumerate}
\end{definition} 

By Lemma \ref{lem:adelic-cmpt}, we have an inclusion $\Aut(\alpha_{\lambda}) \subset \GL_2^+(F)$. The standard component label is defined so that $\Aut(\alpha) = \Gamma_{1, \lambda}(\mf{n})$ as subgroups of $\GL_2^+(F)$.

\begin{lemma}\label{lem:std-cmpt}
Any component label is isomorphic to the standard component label for a unique $\lambda \in \Cl^+(F)$.
\end{lemma}

\subsection{Evaluation of HMFs at component labels}\label{subsec:eval-at-cmpt}

Fix a component label $\alpha = (H, \gamma, \det\nolimits_{\cO_F}(H)^+)$ with $H \tensor_{\cO_F} F = F \oplus F$. 

We define an RM torus over $\mc{H}^n$, depending on $\alpha$. Recall that we have identified \[ \mc{H}^n = \{ z \in (\C \tensor F)^*: \mr{Im}(z) \in (\R \tensor F)^+ \} \] using a fixed ordering of the infinite places of $F$. For $z \in \mc{H}^n$, define $h_{\infty}^{-1} \from H_{\R} \to \C \tensor F$ sending $e_1 \mapsto z \tensor 1,\ e_2 \mapsto 1 \tensor 1$. This defines a lattice $H \subset \C \tensor F$. The pair $(H \subset \C \tensor F, \gamma)$ defines an RM torus with $\Gamma_1(\mf{n})$-structure. In this way, we obtain an RM torus $T_{\alpha} \to \mc{H}^n$, as well as a map $p_{\alpha} \from \mc{H}^n \to Sh'_{G,K}$. The Lie algebra $\Lie(T_{\alpha})$ has a canonical trivialization, $\Lie(T_{\alpha}) = (\C \tensor F) \times \mc{H}^n$. In particular, there is an isomorphism $t_{\alpha} \from \omega^{\tensor k}_{T_{\alpha}} \isom \cO_{\mc{H}^n}$ of sheaves on $\mc{H}^n$.

Given $f \in M_k(\mf{n},\C) = H^0(Sh'_{G,K}, \omega^{\tensor k})$, we obtain a holomorphic function \begin{equation}
f_{\alpha} := t_{\alpha}(p_{\alpha}^*(f)) \from \mc{H}^n \longrightarrow \C. \end{equation} 

This construction gives an interpretation of classical modular forms $f = (f_{\lambda} \from \mc{H}^n \to \C)$ as follows.

\begin{prop}\label{prop:slash-op}
Let $\alpha$ be a component label such that $H \tensor F = F \oplus F$. 
\begin{enumerate}
\item The isomorphism $H^0(Sh_{G,K}, \omega^{\tensor k})_{\mr{mod}} \isom M_k(\mf{n}, \C)$ sends $f$ to $(f_{\alpha_{\lambda}})_{\lambda \in \Cl^+(F)}$, the evaluations at the standard component labels.
\item For $g \in \GL^+_2(F)$, $f_{g^{-1} \cdot \alpha} = f_{\alpha}|_{k, g}$. In particular, $f|_{k, g} = f$ for $g \in \Aut(\alpha) \subset \GL^+_2(F)$. 
\end{enumerate}
\end{prop}
\begin{proof}
(1): This follows from the adelic description of component labels given in Lemma \ref{lem:adelic-cmpt}. In particular, considering $f$ as a function on pairs $(\alpha, h_{\infty})$, we have \[ f_{\alpha}(z) = f(\alpha, h_{\infty, z}), \] where $h_{\infty,z}$ is the isomorphism $\C \tensor F = z F_{\R} + F_{\R} \isom F_{\R}^2$.

(2):
Using (1), we compute:
\begin{align*}
 f_{g^{-1} \cdot \alpha}(z) &= f(g^{-1} \cdot \alpha, h_{\infty,z}) \\
 & = f(\alpha, g \cdot h_{\infty,z}) \\ &= f(\alpha,  \C \tensor F \isom (az + b ) F_{\R} + (c z + d) F_{\R}) \\
 &= N_{F/\Q}(cz + d)^{-k} f(\alpha, \C \tensor F \isom \frac{az + b}{cz + d} F_{\R} + F_{\R}) \\
 &= f_{\alpha}|_{k,g}(z). \end{align*}
  The second  equality holds because $(g^{-1} \cdot \alpha, h_{\infty,z}) \isom (\alpha, g \cdot h_{\infty,z})$. The factor of $N_{F/\Q}(cz + d)^{-k} $ arises from the transformation property of $f$ under the action of $(\C \tensor F)^*$, since dividing the chosen basis vectors by $(cz + d)$ is equivalent to multiplying $\C \tensor F$ by $(cz + d)$.
\end{proof}

Let $L_0 := F \oplus 0 \subset F \oplus F$, and $L_{\infty} := 0 \oplus F \subset H \tensor F = F \oplus F$. Consider the fractional ideal $H \cap L_{\infty} \subset F$. Define
\begin{equation}
f_{\alpha, norm}(z) := |N_{F/\Q}(H \cap L_{\infty})|^{-k}  f_{\alpha} \from \mc{H}^n \to \C.
\end{equation} This satisfies a modified transformation property:
\begin{equation}
f_{g^{-1} \cdot \alpha, norm}(z) = C_{\alpha, g} \cdot f_{\alpha, norm}|_{k, g}(z) \end{equation}
for a constant $C_{\alpha, g} \in \Q$ which we now determine in certain cases.

\begin{lemma}\label{lem:normalize}
Consider a component label $\alpha$ such that \[ H \tensor F = F \oplus F, \qquad H = (H \cap L_0) + (H \cap L_{\infty}), \] and write  $\mf{a} = H \cap L_{\infty}$, $\mf{b} = H \cap L_0$.  Given a matrix $g = \mat{a}{b}{c}{d} \in \GL_2^+(F)$, define a fractional ideal $J_{\alpha, g} = \det(g)^{-1}  (a + c \mf{b}^{-1} \mf{a} )$. Then
\[ f_{g^{-1} \cdot \alpha, norm}(z) = |N_{F/\Q}(J_{\alpha,g})|^{-k}  f_{\alpha, norm}|_{k, g}(z).  \]
\end{lemma}

\begin{proof}

The component label $g^{-1} \cdot \alpha$ has lattice $g^t(H) \subset F \oplus F$, hence \[ f_{g^{-1} \cdot \alpha, norm}(z) = |N_{F/\Q}(g^t(H) \cap L_{\infty})|^k |N_{F/\Q}(H \cap L_{\infty})|^{-k}  f_{\alpha, norm.}|_{k, g}(z). \] We need to compute the fractional ideal $I_{\alpha, g} := g^t(H) \cap L_{\infty} \subset F$.

In fact we will show
\begin{equation}
I_{\alpha, g} =  \det(g) \mf{b} \mf{a} (a\mf{b} + c \mf{a})^{-1}. \end{equation}
This will complete the proof since $J_{\alpha, g} = I_{\alpha, 1}/I_{\alpha,g}$. 

Consider the exact sequence \[ 0 \to (g^t(H) \cap L_{\infty}) \to g^t(H) \to g^t(H)/ (g^t(H) \cap L_{\infty}) \to 0. \] The lattice $g^t(H)$ equals $(\mf{b} (a, b) + \mf{a} (c, d))$. Projecting mod $L_{\infty}$ gives the ideal $(a\mf{b} + c \mf{a})$. The ideal $\det\nolimits_{\cO_F}(g^t(H)) \subset F$ equals $\det(g) \mf{a} \mf{b}$. This implies that the ideal $H \cap L_{\infty} \subset F$ equals $I_{\alpha, g} = \det(g) \mf{b} \mf{a} (a\mf{b} + c \mf{a})^{-1}$. 

\end{proof}

\begin{remark}\label{rmk:normalizing-ideal}
For the standard component label $\alpha_{\lambda}$, \[ J_{\alpha_{\lambda}, g} =  \det(g) (a+ c \mf{t}_{\lambda}^{-1} \mf{d}^{-1})^{-1}. \]
\end{remark}

\subsection{Cusp labels}\label{subsec:cusp-labels}

\begin{definition}
A \textbf{cusp label} is a pair $C = (\alpha,  L)$, where
\begin{itemize}
\item $\alpha = (H, \det\nolimits_{\cO_F}(H)^+, \gamma)$ is a component label.
\item $L \subset H \tensor F$ is a 1-dimensional $F$-subspace, i.e. an element of $\bP(H \tensor F)$.
\end{itemize}
\end{definition}

There is a notion of isomorphism of cusp labels,  given by isomorphisms of $\cO_F$-modules $H_1 \isom H_2$ compatible with the additional structures. In particular, $\Aut(C) \subset \GL(H)$. The \textbf{cusps} are the isomorphism classes of cusp labels, denoted \begin{equation}
\mr{Cusp}(\mf{n}) := \{ (\alpha, L) \}/\text{isom.}
\end{equation}
We will often write $\mr{Cusp} = \mr{Cusp}(\mf{n})$ when $\mf{n}$ can be understood from the context.

For a fixed $\alpha_0$, we have bijections 
 \begin{equation} \{ (\alpha_0, L) \} = \bP(H \tensor F),
 \end{equation}
\begin{equation} \{ (\alpha_0, L) \}/\mr{isom} = \Aut(\alpha_0) \backslash \bP(H \tensor F).
 \end{equation}

Lemma \ref{lem:std-cmpt} says that any component label is isomorphic to a unique standard component label---hence any cusp label is isomorphic to a cusp label of the form $(\alpha_\lambda, L)$ for a unique $\lambda \in \Cl^+(F)$. 

\begin{corollary}\

\begin{enumerate}
\item For $C = (\alpha_\lambda, L)$, $\Aut(C) = \Stab_L(\Gamma_{1, \lambda}(\mf{n}))$.
\item $\mr{Cusp}(\mf{n}) = \coprod_{\lambda \in \Cl^+(F)} \Gamma_{1, \lambda}(\mf{n}) \backslash \bP^1(F)$.
\end{enumerate}
\end{corollary}

We introduce some notation for certain $\cO_F$-modules and subgroups of the group $\Aut(C)$ associated to a cusp label $C$. Associated to any cusp label $C$, we have invertible $\cO_F$-modules $\mf{a}$, $\mf{b}$ defined by $\mf{a} = H \cap L$, $\mf{b} = H/(H \cap L)$, fitting into an exact sequence \[ 0 \to \mf{a} \to H \to \mf{b} \to 0. \]
Fixing a splitting $H \isom \mf{a} \oplus \mf{b} \subset F \oplus F$, we obtain an inclusion $\Aut(C) \subset \GL(\mf{a} \oplus \mf{b}) \into \GL_2(F)$, whose image we denote by $P_C \subset \GL_2(F)$. The group $P_C$ is a subgroup of the upper-triangular matrices in $\GL(\mf{a} \oplus \mf{b})$. We now define:
\begin{itemize} 
\item $M^{\dual}_C$ is the subgroup of matrices in $P_C \subset \GL_2(F)$ of the form $\begin{pmatrix}
1 & * \\ 0 & 1 \\
\end{pmatrix}$. As the strictly upper triangular matrices of $\GL(\mf{a} \oplus \mf{b})$ are identified with $\Hom_{\cO_F}(\mf{b}, \mf{a}) = \mf{b}^{-1} \mf{a}$, we can view $M^{\dual}_C$ as a fractional ideal of $F$ such that $M^{\dual}_C \subset \mf{b}^{-1} \mf{a}$.
\item $M_C := \Hom(M_C^{\dual}, \Z)$ is the $\Z$-linear dual of $M_C^{\dual}$. Via the trace pairing $\tr_{F/\Q} \from F \tensor F \to \Q$, we can view $M_C$ as a fractional ideal such that $M_C  \supset \mf{b} \mf{a}^{-1} \mf{d}^{-1}$.
\item $U_C$ is the image of the conjugation map $P_C \to \Aut_{\cO_F}^+(M^{\dual}_C) = U^+$, i.e. the map sending $\begin{pmatrix}
u_1 & n \\ 0 & u_2 \\
\end{pmatrix} \in P_C$ to the automorphism \[ \begin{pmatrix}
1 & m \\ 0 & 1 \\
\end{pmatrix} \mapsto \begin{pmatrix}
u_1 & n \\ 0 & u_2 \\
\end{pmatrix} \begin{pmatrix}
1 & m \\ 0 & 1 \\
\end{pmatrix} \begin{pmatrix}
u_1 & n \\ 0 & u_2 \\
\end{pmatrix}^{-1} = \begin{pmatrix}
1 & u_1 u_2^{-1} m \\ 0 & 1 \\
\end{pmatrix}.\] The kernel of the conjugation map equals $\{ \begin{pmatrix} u & m \\ 0 & u \\ \end{pmatrix} : u \in U_{1, \mf{n}}, m \in M^{\dual}_C \} \subset P_C$.
\item If $U_{1,\mf{n}} \subset \ker(N_{F/\Q})$, so that odd weights $k$ are admissible, there is a character $\epsilon_C \from U_C \to \{ \pm 1 \}$ defined by $\begin{pmatrix}
u & n \\
0 & v \\
\end{pmatrix} \mod U_{1, \mf{n}} \mapsto N_{F/\Q}(v) = N_{F/\Q}(u)^{-1}$.
\end{itemize}

\begin{remark}\label{rmk:parabolic-confusion}
For a component label $\alpha$ such that $H \tensor F = F \oplus F$, we introduced an inclusion $\Aut(\alpha) \subset \GL_2^+(F)$. For any cusp label $(\alpha, L)$, this gives an inclusion $\Aut(C) \subset \Aut(\alpha) \subset \GL_2^+(F)$. This is not the same as the inclusion $\Aut(C) \subset \GL(\mf{a} \oplus \mf{b}) \subset \GL_2(F)$ discussed above---for example, the latter inclusion is \emph{always} upper-triangular, while the former is only 
upper-triangular for $L = L_{\infty}$.  

Even when $(H \cap L_0) \oplus (H \cap L_{\infty}) = \mf{b} \oplus \mf{a}\subset H$ (as was our convention in \S\ref{subsec:eval-at-cmpt}), so that these two inclusions both land in the upper-triangular matrices, they are still not equal, essentially because $\mf{a}$ and $\mf{b}$ are in the reversed order. For example, the elements of $\Aut(\alpha)$ that fix the line $\mf{a} \tensor F$, and act on it by the scalar $d^{-1}$, are identified with $\left\lbrace \begin{pmatrix}d^{-1} & * \\ 0 & * \end{pmatrix} \right\rbrace$ as a subgroup of $\GL(\mf{a} \oplus \mf{b})$, but with $\left\lbrace \begin{pmatrix}* & * \\ 0 & d \end{pmatrix} \right\rbrace$ as a subgroup of $\GL_2^+(F)$.
\end{remark}

\subsection{Action of $G^+_{\mf{n}}$ on the cusps}\label{subsec:act-on-cusp}

We define an action of prime-to-$\mf{n}$ fractional ideals $N \subset F$ on cusp labels. For such an $N$, the cusp label $C \tensor N$ is defined by
\begin{itemize}
\item the $\cO_F$-module $H \tensor N$,
\item the line $L \tensor N \subset (H \tensor N) \tensor F$,
\item the positivity $\det\nolimits_{\cO_F}(H)^+ \tensor (N^{\tensor 2})^+ \subset \det\nolimits_{\cO_F}(H \tensor N) \tensor \R$,
\item the homomorphism $\gamma_N \from \cO_F/\mf{n}\cO_F \isom N/\mf{n}N \nmto{\gamma \tensor N} H/\mf{n}H \isom (H \tensor N)/\mf{n}(H \tensor N)$.
\end{itemize}
This defines an action of $G^+_{\mf{n}}$ on $\mr{Cusp}(\mf{n})$.
There are isomorphisms \begin{equation}
M_C^{\dual} \nmisom{\tensor N} M^{\dual}_{C \tensor N},\ M_C \nmisom{\tensor N} M_{C \tensor N},
\end{equation}
compatible with positivity. For example, if $\mf{n} = 1$, the first isomorphism is given by the natural map $\Hom_{\cO_F}(\mf{b}, \mf{a}) \isom \Hom_{\cO_F}(\mf{b} \tensor N, \mf{a} \tensor N)$.
 
We now define two characters of the stabilizer $Stab_{[C]} := Stab_{[C]}(G^+_{\mf{n}})$:
\begin{align}
\psi_C & \from Stab_{[C]} \to U^+/U_C, \\
\sgn_C & \from Stab_{[C]} \to \{ \pm 1 \}.
\end{align}

Suppose that we have a specified isomorphism $\varphi \from C \tensor N \isom C$, verifying that $[N] \in G^+_{\mf{n}}$ represents an element of $Stab_{[C]}$. The composition $M_{C} \nmisom{\tensor N} M_{C \tensor N} \nmisom{\varphi} M_C$ gives rise to an element of $\Aut_{\cO_F}^+(M_C) = U^+$. As this depends on the choice of $\varphi$, this gives rise to a well-defined element $\psi_C([N]) \in U^+/U_C$.

On the other hand, $\mf{a} \tensor N \isom \mf{a}$ implies that $[N] = 1$ as an element of $\Cl(\cO_F)$. Thus there is a map $Stab_{[C]} \to \ker(\Cl^+(\cO_F) \to \Cl(\cO_F)) \subset (\Z/2\Z)^n$. Composing with the character $\sgn \from (\Z/2\Z)^n \to \{\pm 1\}$, we obtain $\sgn_C \from Stab_{[C]} \to \{ \pm 1 \}$.

\subsection{Types of cusps}\label{subsec:type-of-cusp}

\begin{definition}
A cusp label $C$ is \textbf{unramified} if  $\gamma(\cO_F/\mf{n}) \subset H/\mf{n}H$ is equal to $(H\cap L)/\mf{n}(H\cap L) \subset H/\mf{n}H$.
\end{definition}

\begin{definition}
For a given prime $p$, a cusp label $C$ is \textbf{$p$-unramified} if, for $\mf{P} \subset \cO_F$ the $p$-part of $\mf{n}$, $\gamma(\cO_F/\mf{n}) \mod \mf{P} \subset H/\mf{P}H$ is equal to $(H\cap L) /\mf{P}(H\cap L) \subset H/\mf{P}H$. Equivalently, a cusp is $p$-unramified if and only if its image under $\mr{Cusp}(\mf{n}) \to \mr{Cusp}(\mf{P})$ is unramified.
\end{definition}

\begin{definition}
A cusp label $C$ is \textbf{admissible} (for the weight $k$) if $\epsilon^k_C = 1$ (see \S\ref{subsec:cusp-labels} for the definition of $\epsilon_C \from U_C \to \{ \pm 1 \}$.)
\end{definition}

These notions only depend on the cusp, not the representative cusp label. We denote the set of unramified cusps by $\mr{Cusp}_{\infty}(\mf{n})$, the set of $p$-unramified cusps by $\mr{Cusp}_p(\mf{n})$, and the set of admissible cusps (resp. admissible unramified, admissible $p$-unramified) by $\mr{Cusp}(\mf{n})^*$ (resp. $\mr{Cusp}_{\infty}(\mf{n})^*$, $\mr{Cusp}_p(\mf{n})^*$), where the relevant weight $k$ should be clear from the context.

\begin{remark}
Note that unramified cusps are admissible (for the weight $k$) if and only if $k$ is an admissible weight, i.e. $N_{F/\Q}(U_{1,\mf{n}})^k = 1$. If $k$ is not an admissible weight, then no cusps are admissible for the weight $k$. 
\end{remark}

We verify that these notions agree with those defined in \cite{dka} (where unramified cusps are denoted $C_{\infty}(\mf{n})$, $p$-unramified cusps are denoted $C_{\infty}(\mf{P}, \mf{n})$, and admissible cusps are denoted $\mr{cusp}^*(\mf{n})$).

\begin{lemma} \label{lem:unram-cusp}
\

\begin{enumerate}
\item This definition of admissible cusp agrees with the definition in \cite{dka}.
\item The set of unramified cusps equals $G_{\mf{n}} \cdot \{ [(\alpha_{\lambda}, L_{\infty})] : \lambda \in \Cl^+(F) \} \subset \mr{Cusp}(\mf{n})$.
\item The set of $p$-unramified cusps equals the pre-image of $\mr{Cusp}_{\infty}(\mf{P})$ along the map \[ \mr{Cusp}(\mf{n}) \to \mr{Cusp}(\mf{P}). \]
\end{enumerate}
\end{lemma}
\begin{proof}
(1) This agrees with the definition of admissible cusps in \cite{dka} by \cite{dka}*{Theorem 3.3}.

(2) This follows from that fact that, for $\mf{n} = 1$, $G_{1} \cdot \{ [(\alpha_{\lambda}, L_{\infty})] : \lambda \in \Cl^+(F) \} = \mr{Cusp}(1)$ gives all cusps (\cite{dka}*{Corollary 3.12}). The group $\ker(G_{\mf{n}} \to G_{1}) = (\cO_F/\mf{n})^*/\cO_F^*$ acts freely and transitively on the fibers of $\mr{Cusp}_{\infty}(\mf{n}) \to \mr{Cusp}(1)$, hence $G_{\mf{n}}\cdot\{ [(\alpha_{\lambda}, L_{\infty})] : \lambda \in \Cl^+(F) \} = \mr{Cusp}_{\infty}(\mf{n})$. 

(3) This is clear, as the map $ \mr{Cusp}(\mf{n}) \to \mr{Cusp}(\mf{P})$ sends the cusp label $(H, L, \gamma \from \cO/\mf{n} \to H \mod \mf{n} H)$ to $(H, L, \gamma \mod \mf{P}  \from \cO/\mf{P} \to H/\mf{P} H)$.
\end{proof}

\subsection{$q$-expansions and constant terms}\label{subsec:q-cusplabel}
\subsubsection{Formal $q$-expansions}
Define \begin{equation}
Q_{C, k} := (N_{F/\Q}(\mf{a})^{\tensor -k} \tensor (\prod_{\alpha \in M_C^+ \cup \{ 0 \}} \Z q^{\alpha}))^{\Aut(C)},\end{equation} the module of \textbf{integral $q$-expansions}. Note that the action of $\Aut(C)$ factors through $\Aut(C) \to U_C$. More generally, for any abelian group $B$, define \begin{equation}
Q_{C, k}(B) := (N_{F/\Q}(\mf{a})^{\tensor -k} \tensor (\prod_{\alpha \in M_C^+ \cup \{ 0 \}} B q^{\alpha}))^{\Aut(C)},\end{equation} 
There is an injective homomorphism $Q_{C,k}(\Z) \tensor B \into Q_{C,k}(B)$.

The module $Q_{C, k}$ is, in an appropriate sense, independent of the representative cusp label $C$ for a cusp $[C] \in \mr{Cusp}(\mf{n})$. If $[C_1] = [C_2] \in \mr{Cusp}(\mf{n})$, choose an isomorphism $\varphi \from C_1 \isom C_2$, giving $\mf{a}_1 \isom \mf{a}_2, M_{C_1} \isom M_{C_2}$. We obtain an isomorphism \[ N_{F/\Q}(\mf{a}_1)^{\tensor - k}\tensor \prod_{b \in  M_{C_1}^+ \cup \{ 0 \}} \C q^b \isom N_{F/\Q}(\mf{a}_2)^{\tensor - k}\tensor \prod_{b \in  M_{C_2}^+  \cup \{ 0 \}} \C q^{b}. \] By taking $U_{C_1}$-invariants on the left and $U_{C_2}$-invariants on the right, this becomes independent of the choice of $\varphi$.

For this reason, we may define \begin{equation} Q_{[C], k} := \lim_{C' \in [C]} Q_{C',k}, \end{equation} the module of compatible $q$-expansions for every cusp label $C'$ in the isomorphism class $[C]$, and \begin{equation} Q_{k} := \oplus_{[C] \in \mr{Cusp}(\mf{n})} Q_{[C], k}. \end{equation} Similarly, let $Q_{\infty, k}$ (resp. $Q_{p,k}$) denote the module of $q$-expansions at the unramified (resp. $p$-unramified) cusps.

\subsubsection{The $q$-expansion of HMFs}
We now define the \textbf{analytic $q$-expansion at a cusp $[C]$}, \[ \mr{qexp}_{[C]} \from M_k(\mf{n}, \C) \to Q_{[C], k}(\C). \] 

It suffices to define $\mr{qexp}_{C}(f) \in Q_{C,k}(\C)$ for cusp labels $C = (\alpha, L_{\infty})$ such that $H \tensor F = F \oplus F$ and $L_{\infty} = 0 \oplus F \subset F \oplus F$. For such a cusp label, we have an associated function $f_{\alpha} \from \mc{H}^n \to \C$ (see \S\ref{subsec:eval-at-cmpt}). 

For $g = \begin{pmatrix}
a & b \\
0 & d \\
\end{pmatrix} \in P_C \subset \GL_2^+(F)$, $f_{\alpha}|_{g, k}(z) = N_{F/\Q}(d)^{-k} f(g \cdot z)$. In particular, $f_{\alpha}$ is invariant by translations $z \mapsto z + M_C^{\dual}$, hence its Fourier expansion is contained in $\prod_{b \in M_C^+ \cup \{ 0 \}} \C q^{b}$, where \[ q^b = e^{2 \pi i \tr_{F/\Q}(b z)} = e^{2 \pi i \sum b_i z_i}. \] Moreover, the growth condition implies that its Fourier expansion is contained in \[ \prod_{b \in M_C^+ \cup \{ 0 \}} \C q^{b}. \] 

Viewing $g$ as an element of $\Aut(C)$, the isomorphism $g \from N_{F/\Q}(\mf{a})^{\tensor -k} \isom N_{F/\Q}(\mf{a})^{\tensor- k}$ is given by multiplication by $N_{F/\Q}(d)^{k}$ (see Remark \ref{rmk:parabolic-confusion}). This implies that the element \[ \mr{qexp}_{C}(f) := 1 \tensor f_{\alpha} \in N_{F/\Q}(\mf{a})^{\tensor -k} \tensor \prod_{b \in M_C^+ \cup \{ 0 \}} \C q^{b} \] is in fact $\Aut(C)$-invariant (noting that $\mf{a}$ is canonically a fractional ideal for such a cusp label, and so $1 \in N_{F/\Q}(\mf{a})^{-k} \tensor \Q = \Q$), hence is contained in $Q_{C, k}(\C)$.

\subsubsection{The action of diamond operators}\label{subsubsec:action-q}
\

We describe the action of diamond operators on $q$-expansions. Given a fractional ideal $N$ and a cusp label $C$, there is an isomorphism \[ Q_{C \tensor N^{-1}, k}(\C) \isom Q_{C, k}(\C), \] given by multiplication by $|N_{F/\Q}(N)|^k$, as well as the isomorphism $M_C \isom M_{C \tensor N^{-1}}$ as invertible $\cO_F$-modules with positivity. This defines an action of $G^+_{\mf{n}}$ on \[ Q_{k}(\C) = \oplus_{[C] \in \mr{Cusp}(\mf{n})} Q_{[C], k}(\C), \] where $[N] \in G^+_{\mf{n}}$ acts via $[C] \mapsto [C \tensor N^{-1}]$ on the summands.

\begin{theorem}\label{thm:qexp-equivar}
The total $q$-expansion map 
\[ \mr{qexp} \from M_{k}(\mf{n}, \C) \longhookrightarrow Q_{k}(\C) = \oplus_{[C] \in \mr{Cusp}(\mf{n})} Q_{[C], k}(\C) \]
is $G^+_{\mf{n}}$-equivariant.
\end{theorem}

We will not prove this theorem now, as we will prove it more generally for the algebraic $q$-expansions in \S\ref{subsec:toroidal-cpt}. However, it would not be difficult to prove this analytically.

\subsubsection{Constant terms}
Define \begin{equation} C_{C, k} := (N_{F/\Q}(\mf{a})^{\tensor- k})^{\Aut(C)}, \end{equation} the module of \textbf{integral constant terms} for a cusp label $C$. For any abelian group $B$, define \begin{equation}C_{C,k}(B) := (N_{F/\Q}(\mf{a})^{\tensor- k} \tensor B)^{\Aut(C)}. \end{equation} Note that $C_{C,k} = N_{F/\Q}(\mf{a})^{\tensor -k}$ if $C$ is admissible for the weight $k$, and $0$ otherwise. However, $C_{C,k}(\Z/2\Z)$ is non-zero for all $k$. We define 
\[ C_{[C], k} := \lim_{C' \in [C]} C_{C,k}, \qquad C_{k} := \oplus_{[C] \in \mr{Cusp}(\mf{n})} C_{[C], k}.  \]

For each cusp $[C]$, we have a \textbf{constant term map}:
\[ \mr{const}_{[C]} \from M_k(\mf{n}, \C) \to C_{[C], k}(\C), \]
given by the $q^0$-coefficient of the $q$-expansion $\mr{qexp}_{[C]}(f)$.

The action of $G^+_{\mf{n}}$ on $Q_{k}(\C)$ induces an action on $C_k(\C)$. By Theorem \ref{thm:qexp-equivar}, we obtain:
\begin{corollary}
The total constant term map \[ \mr{const} \from M_k(\mf{n}, \C) \to C_{k}(\C) = \oplus_{[C] \in \mr{Cusp}(\mf{n})} C_{[C], k}(\C) \] is $G^+_{\mf{n}}$-equivariant, where $[N] \in G^+_{\mf{n}}$ acts via $[C] \mapsto [C \tensor N^{-1}]$ on $\mr{Cusp}(\mf{n})$, and $Stab_{[C]}$ acts on the $[C]$-summand via the character $\sgn^k_{[C]}$.
\end{corollary}

\subsubsection{Comparison with \cite{dka}}
We explain how these definitions of constant terms and $q$-expansions compare to those of \cite{dka}. This comparison is the only reason we needed to introduce classical HMFs $(f_{\lambda} \from \mc{H}^n \to \C)$, the unitary normalization, and the evaluation at the standard component labels.  Recall that $M_k(\mf{n}, \C) = M^{\mr{hol}}_k(\mf{n}, \C)$, and let $f^{\mr{hol}}$ denote an element of $M^{\mr{hol}}_k(\mf{n}, \C)$, corresponding to $f^{\mr{uni}}$ under the isomorphism $M^{\mr{hol}}_k(\mf{n}, \C) \isom M^{\mr{uni}}_k(\mf{n}, \C).$

\begin{prop}\

\begin{enumerate}
\item The normalized Fourier expansion of $|N_{F/\Q}(\mf{d})|^{-k/2} f^{\mr{uni}}$ at $\mc{A} = (\lambda, I)$ (in the sense of \cite{dka} \S 2.3) equals the $q$-expansion $\mr{qexp}_{C_{\lambda}}(f^{\mr{hol}})$, where $C_{\lambda} := (\alpha_{\lambda}, L_{\infty})$.
\item The normalized constant term of $|N_{F/\Q}(\mf{d})|^{-k/2} f^{\mr{uni}}$ at the pair $\mc{A} = (\lambda, A)$ (in the sense of \cite{dka} \S 2.2) equals $|N_{F/\Q}(\mf{a})|^{k} \mr{const}_{C_{(A, \lambda)}}(f^{\mr{hol}}) \in \C$, where $C_{(A, \lambda)} := (A^{-1} \alpha_{\lambda}, L_{\infty})$.
\end{enumerate}
\end{prop}
\begin{proof}
(1) is clear - the normalization on the unitary side arises from the factors in the isomorphism $M^{\mr{uni}}_k(\mf{n}, \C) \isom M^{\mr{hol}}_k(\mf{n}, \C)$.

(2) Taking into account the bijection $M_k^{\mr{uni}}(\mf{n}, \C) \isom M_k^{\mr{hol}}(\mf{n}, \C)$ and the difference between the slash operators $||_{k,A}$ and $|_{k,A}$, we see that the normalizing factor on the unitary side matches with the normalizing factor on the holomorphic side computed in Lemma \ref{lem:normalize}.
\end{proof}

\section{Algebraic Hilbert modular varieties}\label{sec:alg-hmv}

Having reviewed the analytic theory of Hilbert modular forms for the group $G$, we now describe the algebraic theory, using the algebraic theory for $G^*$ due to Rapoport \cite{rapoport} and Deligne--Pappas \cite{deligne-pappas}. We will make use of many standard facts about abelian schemes; see \cite{mumford} for reference.

\subsection{Real multiplication abelian schemes}\label{subsec:rm-ab}

\begin{definition}
An \textbf{RM abelian scheme} (over $S$) is a pair $(A, \iota)$, for $A \to S$ an abelian scheme of relative dimension $[F : \Q]$, and $\iota \from \cO_F \to \End_S(A)$ a group homomorphism.
\end{definition}

Given an invertible $\cO_F$-module $N$, and an RM abelian scheme $A$, there is another RM abelian scheme $A \tensor_{\cO_F} N$, whose construction (\cite{deligne-pappas} 2.6) we now recall. Given an $S$-scheme $T$, define $(A \tensor_{\cO_F} N)^{pre}(T) := A(T) \tensor_{\cO_F} N$. Then $A \tensor_{\cO_F} N$ is the sheafification of $(A \tensor_{\cO_F} N)^{pre}$, considered as a pre-sheaf on the big \'etale site of $S$.

 \begin{remark}
Over $\C$, $A(\C) = \Lie(A)/H$ and $(A \tensor_{\cO_F} N)(\C) = (\Lie(A) \tensor_{\cO_F} N)/(H \tensor_{\cO_F} N)$.
\end{remark} 

We consider $\Lie_S(A)$ as an $\cO_F$-module, via $\iota$. This allows us to view $\Lie_S(A)$ as an $(\cO_F \tensor \cO_S)$-module. There are natural isomorphisms \[ \Lie_S(A \tensor N) \isom \Lie_S(A) \tensor_{\cO_F \tensor \cO_S} N, \qquad (A \tensor N)[\mf{a}] \isom A[\mf{a}] \tensor N,\] for an ideal $\mf{a} \subset \cO_F$. Recall that \[ A[\mf{a}] := \cap_{\alpha \in \mf{a}} \ker(\iota(\alpha) \from A \to A). \]

Let $(\mc{P}, \mc{P}^+)$ be an invertible $\cO_F$-module with positivity.  
\begin{definition}[\cite{deligne-pappas}]
An RM abelian scheme $A$ is \textbf{$\mc{P}$-polarized} if it is equipped with an isomorphism \[ \varphi \from (\mc{P}, \mc{P^+}) \isom (\Hom_{\cO_F}(A, A^{\dual})_{\mr{sym. isog.} }, \cO_F\text{-polarizations}) \] of $\cO_F$-modules with positivity, such that the induced map $A \tensor \mc{P} \to A^{\dual}$ is an isomorphism.
\end{definition}

\begin{definition}
An RM abelian scheme is \textbf{RM-polarizable} if it admits a $\mc{P}$-polarization for some $(\mc{P}, \mc{P}^+)$.
\end{definition}

We recall some fundamental properties:

\begin{theorem}[\cite{rapoport} 1.25]
Any RM complex torus is the analytification of an RM abelian variety.
\end{theorem}

\begin{theorem}[\cite{rapoport} 1.25]
Any RM abelian scheme $A/\Spec(\C)$ is RM-polarizable.
\end{theorem}

\begin{prop}[\cite{rapoport} 1.17]\label{prop:polarization-sheaf}
Given an RM abelian scheme $A/S$, the functor 
\[ (Sch/S)^{op} \to Set,  \qquad T \mapsto \{ \text{RM polarizations of } A_T \} \] is a sheaf on the big \'etale site of $S$. On each connected component of $S$, it is either empty or an \'etale $U^+$-torsor. 
\end{prop}

As all \'etale $U^+$-torsors over a normal noetherian ring are trivial:
\begin{corollary}\label{cor:auto-polar}
Any RM abelian scheme over a normal noetherian ring $R \subset \C$ is RM-polarizable.
\end{corollary}

Given an invertible $\cO_F$-module $N$, and an $\mc{P}$-polarized RM abelian scheme $A$, the RM abelian scheme $A \tensor N$ is canonically $\mc{P} \tensor N^{\tensor -2}$-polarized, using the canonical positivity structure on $N^{\tensor - 2}$.

\subsection{The moduli problem for $G^*$}\label{subsec:G*-moduli}
 
Let $(\mc{P}, \mc{P}^+)$ be an $\cO_F$-module with positivity. There is a stack $\wt{\M}_{\mc{P}}$ such that the objects of the groupoid $\wt{\M}_{\mc{P}}(S)$ are triples $(A, \varphi, x)$, where
\begin{enumerate}
\item $A \to S$ is an RM abelian scheme;
\item $\varphi$ is a $\mc{P}$-polarization of $A$;
\item $x \from (\underline{\cO_F/\mf{n}\cO_F})(1) \into A[\mf{n}]$ is an embedding of $\cO_F$-module schemes, where $\underline{\cO_F/\mf{n}\cO_F}$ is a constant group scheme over $S$.
\end{enumerate}
Here $G(1) := G \tensor \wh{\Z}(1)$, where $\wh{\Z}(1) = \lim \mu_n$. 

\begin{remark}
This stack may not be rigid: if $1 \equiv -1 \pmod{\mf{n}}$, the automorphism $-1 \in \Aut_S(A)$ acts as an automorphism of all triples $(A, \varphi, x)$.
\end{remark}

For any component label $\alpha$, define \[ \mc{P}_{\alpha} := \Hom_{\Z}({\det}_{\cO_F}(H), \Z) = {\det}_{\cO_F}(H)^{\tensor -1} \tensor_{\cO_F} \mf{d}^{-1}. \] This is an $\cO_F$-module with positivity $(\mc{P}_{\alpha}, \mc{P}_{\alpha}^{+})$. For the standard component label $\alpha_{\lambda}$, we have $\mc{P}_{\alpha_{\lambda}} = \mf{t}_{\lambda}^{-1} \mf{d}^{-2}$. 

\begin{theorem}[\cite{rapoport}, \cite{deligne-pappas}]\label{thm:G*-moduli}
\
\begin{enumerate}
\item $\wt{\M}_{\mc{P}}$ is a geometrically normal Deligne--Mumford stack over $\Spec(\Z)$. It is smooth over $\Spec(\Z[1/\mf{d}])$. 
\item $\wt{\M}^{\mr{an}}_{\mc{P}_{\alpha_{\lambda}}, \C} \isom Sh_{G^*, K_{\lambda}}$ (depending on a choice of isomorphism $\cO_F/\mf{n} \isom \cO_F/\mf{n}(1)$). 
\item $\wt{\M}_{\mc{P}}$ is irreducible.
\item For all primes $p$, $\wt{\M}_{\mc{P}, \F_p}$ is geometrically irreducible.
\end{enumerate}
\end{theorem}

Let $\M_{\mc{P}}$ denote the stack in the case $\mf{n} = 1$. This theorem was proven by \cite{rapoport} and \cite{deligne-pappas} for $\M_{\mc{P}}$, and the results for $\wt{\M}_{\mc{P}}$ can immediately be deduced via the \'etale morphism $\wt{\M}_{\mc{P}} \to \M_{\mc{P}}$.

\subsection{The moduli problem for $G$}\label{subsec:G-moduli}
\
The group $U^+ = \Aut(\mc{P}, \mc{P}^+)$ acts on $\wt{\M}_{\mc{P}}$ by precomposition with the polarization $\varphi$. This action of $U^+$ on $\wt{\M}_{\mc{P}, \C}^{\mr{an}} \isom Sh_{G^*, K_{\lambda}}$ agrees with the action introduced in \S\ref{subsec:shimura-var}. Motivated by the identification $Sh_{G,K} = \coprod_{\lambda \in \Cl^+(F) } Sh_{G^*, K_{\lambda}}/D_{\mf{n}}$ given in (\ref{e:shgk}), we might want to define an algebraic stack $\wt{\M}$ via
\[ \wt{\M} \ ``=" \coprod_{[\mc{P}] \in \Cl^+(F)} \wt{\M}_{\mc{P}}/D_{\mf{n}}. \] 

However, it is somewhat inconvenient to directly define an action of $D_{\mf{n}}$ on $\wt{\M}_{\mc{P}}$, as $(U_{1,\mf{n}})^2 \subset U^+$ does not actually act trivially on the groupoids $\wt{\M}_{\mc{P}}(S)$. For this reason, we will define $\wt{\M}$ via a moduli problem. We first define a stack $\wt{\M}'$. Consider, for any scheme $S$, the groupoid $\wt{\M}'(S)$ of pairs $(A, x)$ such that:
\begin{enumerate}
\item $A \to S$ is an RM abelian scheme.
\item $A$ is \'etale locally RM-polarizable.
\item $x \from (\underline{\cO_F/\mf{n}\cO_F})(1) \into A[\mf{n}]$ is an embedding of $\cO_F$-module schemes.
\end{enumerate}

This data is not rigid---every such pair $(A, x)$ has automorphisms by $\iota(U_{1,\mf{n}}) \subset \Aut_S(A)$.  We define $\wt{\M}$ to be the rigidification of $\wt{\M}'$ (see \S \ref{subsec:stacks}). 

\begin{lemma}
The categories fibered in groupoids $\widetilde{\M}'$ and $\widetilde{\M}$ are both stacks.
\end{lemma}
\begin{proof}
The main thing which must be checked is \'etale descent for the objects $(A,x) \in \widetilde{\M}'(S)$. If ``scheme" in condition (1) were replaced by ``algebraic space", this would be automatic. But it is a deep result of of Raynaud that any ``abelian algebraic space" over a scheme $S$ is in fact an abelian scheme (\cite{faltings-chai} I.1.6). We leave the remaining details to the reader.
\end{proof}

Note that $\wt{\M}'$ is not an \emph{algebraic} stack, merely a stack in the category of algebraic spaces, as its inertia groups are infinite discrete groups.

\begin{prop}\

\begin{enumerate}
\item $\wt{\M}'$ equals the stack quotient $\coprod \wt{\M}_{\mc{P}}/U^+$.
\item $\wt{\M}$ is a Deligne--Mumford stack, and the natural map $\coprod \wt{\M}_{\mc{P}} \to \wt{\M}$ is an \'etale $D_{\mf{n}}$-torsor.
\end{enumerate}
\end{prop}
\begin{remark}
To be precise, the fibers of $\coprod \wt{\M}_{\mc{P}} \to \wt{\M}$ are, \'etale locally on the base, equal to the groupoid quotient (see \S\ref{subsec:quotient-stacks}) of the set $U^+$ by the group $(U_{1,\mf{n}})^2$. This groupoid is \emph{equivalent} to the set $D_{\mf{n}} = U^+/(U_{1,\mf{n}})^2$.
\end{remark}

\begin{proof}
This essentially follows from Proposition \ref{prop:polarization-sheaf}, and we omit the details.
\end{proof}

One can also consider the ``universal RM abelian scheme" $\mc{A} \to \mc{\M}'$ defined by \[ \mc{A}(S) = \{ (A, x, y) \mid (A,x) \in \mc{\M}'(S),\ y \in A(S) \}. \] Of course, $\mc{A}$ is not actually a scheme, but this abuse of terminology is standard. 

Despite $\wt{\M'}$ and $\mc{A}$ not being \emph{algebraic} stacks, there is no difficulty in defining the analytifications $(\mc{\M}'_{\C})^{\mr{an}}$ and $(\mc{A}_{\C})^{\mr{an}}$ in the usual way, and we have:
\begin{prop}\

\begin{enumerate}
\item $(\wt{\M}'_{\C})^{\mr{an}} \isom Sh'_{G, K}$, and $\mc{A}^{\mr{an}}_{\C} \isom \mc{T}$.
\item $(\wt{\M}_{\C})^{\mr{an}} \isom Sh_{G, K}$. 
\end{enumerate}
\end{prop}
\begin{proof}
The isomorphism $(\wt{\M}_{\mc{P}_{\alpha_{\lambda}}, \C})^{\mr{an}} \isom Sh_{G^*, K_{\lambda}}$ studied in \cite{rapoport} is $U^+$-equivariant. Moreover, this isomorphism identifies the analytification of the universal $\mc{P}$-polarized RM abelian scheme with the universal $\mc{P}$-polarized RM torus. Just as $\wt{\M}' = (\coprod \wt{\M}_{\mc{P}})/U^+$, the analytic stack $Sh'_{G, K}$ is the quotient of $\coprod Sh_{G^*, K_{\lambda}}$ by $U^+$. The result immediately follows.
\end{proof}

\subsection{Elliptic points}

We discuss the closed points of the stack $\wt{\M}$ with non-trivial inertia group, the \textbf{elliptic points}. We give bounds on these inertia groups, as well as conditions on $\mf{n}$ for which they are trivial. These can be deduced from well-known results for the stacks $\wt{\M}_{\mc{P}}$, but we find it easier to prove them directly, as opposed to studying the action of $D_{\mf{n}} = U^+/(U_{1, \mf{n}})^2$ on $\wt{\M}_{\mc{P}}$.

\begin{lemma}\label{lem:cm-points}
For all sufficiently divisible $\mf{n}$, $\wt{\M}$ is an algebraic space.
\end{lemma}

\begin{proof}
By the theory of lifting of CM abelian varieties, any closed point of $\M$ with non-trivial inertia group is defined over a finite field $\F_q$, and for some number ring $\cO_E$ and maximal ideal $\mf{m}$, equals the reduction modulo $\mf{m}$ of a point of $\M(\cO_E)$ that also has non-trivial inertia group. Hence it suffices to show that, for $\mf{n}$ sufficiently divisible, the automorphisms of a point $(A \mod U^+) \in \M(\C)$ do not lift to automorphisms of any point $(A, \gamma \from  \cO_F/\mf{n}(1) \into A[\mf{n}]) \mod U_{1,\mf{n}} \in \wt{\M}(\C)$ lying over $(A \mod U^+)$.

Fixing a $\Gamma(1)$ component label $(H, \det\nolimits_{\cO_F}(H)^+)$, the automorphisms of objects of $\M(\C)$ are in bijection with elements $\bar{x} \in \GL^+(H)/\cO_F^*$, such that choosing a representative $x \in \GL^+(H)$, the ring $\cO_F[x]$ is an order in a CM field $K$ and $x \in \cO_F[x]^* \subset \cO_K^*$. We will show that there exists an ideal $\mf{n} \subset \cO_F$ (independent of $x$) such that the action of $x$ on the set $B := \{ \gamma \from \cO_F/\mf{n} \into H/\mf{n}H \}$ has no fixed points.

 Dirichlet's unit theorem implies that $\cO_K^* = \mu(K) \times U'$, for $U'$ containing $U^+$ as a finite-index subgroup. We consider two cases:
\begin{enumerate}
\item $x = \zeta u$ for some $u \in U'$, $\pm 1 \neq \zeta$ in $\mu(K)$
\item $x = \pm u$ for some $u \in U'$.
\end{enumerate} 

\textbf{Case 1:}
By replacing $x$ with $x^k$, we may assume that $u \in U^+$.

As there are only finitely many possible CM extensions $K/F$ containing additional roots of unity, we may choose a prime $\mf{p} \subset \cO_F$ that is inert in all of them. The element \[ x \in \GL(H \tensor \cO_{F, \mf{p}}) \isom \GL_2(\cO_{F, \mf{p}}) \subset \GL_2(F_{\mf{p}}) \] has eigenvalues $(\zeta u, \zeta^{-1} u)$. By the choice of $\mf{p}$, and the fact that $\cO_K = \cO_F[\zeta]$, the element $\zeta \in \cO_K/\mf{p}\cO_K$ is not contained in $\cO_F/\mf{p}\cO_F$. Hence the eigenvalues of $x \pmod {\mf{p}} \in \GL_2(\cO_F/\mf{p})$ are $\neq 1$, and so $x$ has no fixed points in the set $B$.

\textbf{Case 2:}
By taking the minimal $k$ so that $x^k \in U^+$, we may assume that $x^k$ is at worst a square in $U^+$, that is, $x^k$  is non-zero in $U^+ \tensor \Z/\ell\Z$ for all primes $\ell \neq 2$.

It is possible to choose an ideal $\mf{n}$, and a prime $\ell$, so that $U^+ \tensor \Z/\ell\Z \to (\cO_F/\mf{n} \cO_F)^* \tensor \Z/\ell \Z$ is injective (note that this does not depend on $x$). In particular, the image of $x^k$ in $(\cO_F/\mf{n} \cO_F)^*$ is non-zero. As $x^k$ acts on $H/\mf{n}H$ by a non-zero element of $\cO_F/\mf{n} \cO_F$, we find that $x^k$ has no fixed points in the set $B$, and so neither does $x$.

\end{proof}

The  proof above shows that, for any ideal $\mf{n}$ satisfying
\begin{enumerate}
\item $\mf{n}$ is divisible by a prime $\mf{p}$, such that any CM extension $K_i/F$ containing an additional root of unity is inert at $\mf{p}$,
\item there exists a prime $\ell \neq 2$ such that $U^+ \tensor \Z/\ell \Z \to (\cO_F/\mf{n}\cO_F)^* \tensor \Z/\ell \Z$ is injective,
\end{enumerate} the level $\Gamma_1(\mf{n})$ moduli stack $\wt{\M}_{\Z_{(p)}}$ is an algebraic space.

\begin{lemma}\label{lem:full-level}
Fix an ideal $\mf{n} \subset \cO_F$ such that:
\begin{enumerate}
\item $\mf{n} \cap \Z$ is divisible by a prime $> 2[F : \Q]$,
\item there exists a prime $\ell \neq 2$ such that $U^+ \tensor \Z/\ell \Z \to (\cO_F/\mf{n}\cO_F)^* \tensor \Z/\ell \Z$ is injective.
\end{enumerate}

The moduli stack $\M(\mf{n})$ with $\Gamma(\mf{n})$ level structure, parametrizing pairs $(A, (\cO_F/\mf{n} \cO_F)^2 \isom A[\mf{n}])$, is an algebraic space.
\end{lemma}
\begin{proof}
The proof of this is similar to the proof as the previous lemma, where instead of requiring that $x$ has no fixed points in the set $B$, we require that $x \not\equiv 1 \mod \mf{n}$. Thus we can weaken the requirement that a prime $\mf{p} \mid \mf{n}$ where all possible roots of unity are inert at $\mf{p}$, to merely require that no possible roots of unity are congruent to 1 modulo $\mf{p}$.
\end{proof}

If we want the inertia groups of $\M(\mf{n})$ to be $2$-groups, instead of being trivial, there is a much simpler condition:
\begin{lemma}\label{lem:full-level-2}
For any ideal $\mf{n}$, such that $\mf{n} \nmid (2)$, the inertia groups of $\M(\mf{n})$ are subgroups of $U_{1, \mf{n}}^+/(U_{1, \mf{n}})^2$.
\end{lemma}
\begin{proof}
As in Lemma \ref{lem:cm-points}, we reduce to the complex-analytic question.

If $x \in \ker(\GL(H) \to \GL(H/\mf{n}H))$ such that $\cO_F[x]$ is an order in a CM field, we find that $N_{K/F}(x) \in U_{1, \mf{n}}^+$. If $N_{K/F}(x) = u^2 \in (U_{1, \mf{n}})^2$, then $N_{K/F}(x/u) = 1$.

Note that $(\cO_K^*)^{N_{K/F} = 1} = \mu(K)$, and that $\SL(H) \to \SL(H/\mf{n}H)$ is injective on finite-order elements for $\mf{n} \nmid (2)$ (from the well-known fact that $\SL_m(\Z) \to \SL_m(\Z/n\Z)$ is injective on finite-order elements for $n \geq 3$), to conclude. 
\end{proof}

If $\mf{n}$ is coprime to $p$, the map $\M(\mf{n})_{\Z_p^{ur}} \to \M_{\Z_p^{ur}}$ is finite \'etale with Galois group \[ G = \GL_2(\cO_F/\mf{n}\cO_F)/\cO_F^*, \] so that $\M_{\Z_p^{ur}} \isom \M(\mf{n})_{\Z_p^{ur}}/G$. This implies that:

\begin{corollary}\label{cor:inertia-bound}
The level $\Gamma(1)$ moduli stack $\M_{\Z_{(p)}}$ has inertia groups of order dividing \[ |U_{1, \mf{n}}^+/(U_{1, \mf{n}})^2||\GL_2(\cO_F/\mf{n}\cO_F)/\cO_F^*| \] for any ideal $\mf{n}$ coprime to $p$, such that $\mf{n} \nmid 2$.
\end{corollary}

\section{Algebraic Hilbert modular forms for $G$}
\subsection{Hodge line bundles and diamond operators}\label{subsec:hodge-bundle}
\begin{definition}
The \textbf{Hodge bundle} $\omega$ is the line bundle on $\wt{\M}'$ defined by \[ (A \to S, x) \mapsto \omega_A := \det\nolimits_{\cO_S}(\Lie_S(A)^{\dual}). \] 
\end{definition}

To verify that this defines a line bundle on the stack $\wt{\M}'$, note that $\omega_A$ is compatible with arbitrary base-change: for $g \from S' \to S$, $\omega_{A_S'} = g^*(\omega_A)$. 

\begin{lemma}
Suppose that $A/S$ is (\'etale locally) RM-polarizable. For $\alpha \in \cO_F$, the map $\iota(\alpha)^* \from \omega_A \to \omega_A$ equals multiplication by $N_{F/\Q}(\alpha) \in \cO_S$.
\end{lemma}
\begin{proof}
If $S = \Spec(\C)$, the analytic uniformization of RM abelian varieties implies that $\Lie(A) = \cO_F \tensor_{\Z} \C$ as an $\cO_F \tensor \C$-module, from which the result easily follows. In general, it follows from \cite{rapoport}, \cite{deligne-pappas}, that the (strictly Henselian) local rings $\cO_x$ of the stack $\M_{\mc{P}}$ at all closed points $x$ can be embedded in $\C$. 
\end{proof}

Hence, if a weight $k \in \N$ is admissible, i.e. if $U_{1,\mf{n}} \subset \ker(N_{F/\Q}^k)$, the line bundle $\omega^{\tensor k}$ descends to a line bundle on the rigidification $\wt{\M}$ of $\wt{\M}'$.

\begin{theorem}\label{thm:alg-hodge-bundle}
For any admissible weight $k$, there is an action of $G^+_{\mf{n}}$ on $(\wt{\M}, \omega^{\tensor k})$ such that the analytification map
\[ H^0(\wt{\M}_{\C}, \omega^{\tensor k}) \to M_k(\mf{n}, \C) \]
is an isomorphism of $G^+_{\mf{n}}$-modules.
\end{theorem}

The remainder of \S\ref{subsec:hodge-bundle} will be devoted to: defining the action of $G^+_{\mf{n}}$, defining the analytification map, and proving Theorem \ref{thm:alg-hodge-bundle}.

 Fix a prime $p$. We now define an action of $I_{p\mf{n}}$ on $\wt{\M}'_{\Z_{(p)}}$, and an action of \[ I_{p\mf{n}} \times \Res_{\cO_{F, (p)}/\Z_{(p)}} \G_{m, \Z_{(p)}} \] on the vector bundle $\Lie(\mc{A})_{\Z_{(p)}} \to \wt{\M}'_{\Z_{(p)}}$. 

For any $\Z_{(p)}$-scheme $S$,  there is an action of $I_{p\mf{n}} \times (\Res_{\cO_F/\Z} \G_m)(S)$ on  $\Lie(\mc{A})_{\Z_{(p)}}$. For $N \in I_{p\mf{n}}$, $\alpha \in (\cO_S \tensor_{\Z} \cO_F)^*$,  and $(A, x, v) \in \Lie(\mc{A})_{\Z_{(p)}}(S)$, this action is defined by \begin{equation}
(N, \alpha) \cdot (A, x, v) := (A \tensor N, x_N, \alpha \cdot v_N),
\end{equation}
where:
\begin{itemize}
\item the torsion point $x_N$ is defined by the composition 
\[ \begin{tikzcd}
x_N \from \underline{\cO_F/\mf{n}}(1) \isom \underline{\cO_F/\mf{n}}(1) \tensor N \ar[r,"{x \tensor N}"] & A[\mf{n}] \tensor N \isom (A \tensor N)[\mf{n}] \end{tikzcd};
\] 
\item the element $v_N  \in \Lie_S(A \tensor N)$ is the image of $v$ along the isomorphism \[ \Lie_S(A) \isom \Lie_S(A) \tensor N  \isom \Lie_S(A \tensor N), \] where the first isomorphism exists since $N \tensor \Z_{(p)} = \cO_F \tensor \Z_{(p)}$. 
\end{itemize}

As in the analytic setting, there is a sheaf $\omega_{\chi_{\infty}}$ on $\wt{\M}'$ defined by \[ (A \to S, \gamma) \mapsto \{ \text{functions } f \from \Lie_S(T) \to \cO_S \text{ s.t. } f(w v) = \chi_{\infty}(w) f(v) \text{ for all } w \in (\cO_S \tensor F)^* \}, \]
such that $H^0(\Lie(\mc{A}), \cO_{\Lie(\mc{A})})^{\chi_{\infty}} = H^0(\wt{\M}', \omega_{\chi_{\infty}})$. Moreover for $\chi_{\infty} = N_{F/\Q}^k$ where $k$ is an admissible weight, $\omega_{\chi_{\infty}}$ descends to $\wt{\M}$. After base-change to $\Z_{(p)}$, $\omega_{\chi_{\infty}}$ is a $(I_{p\mf{n}}, (F_{1, \mf{n}}^*)^+/U_{1, \mf{n}}^+)$-equivariant line bundle on the stack $\wt{\M}_{\Z_{(p)}}$ (see \S\ref{subsec:grp-act}), which we abbreviate to calling a $G^+_{\mf{n}}$-equivariant line bundle.

\begin{remark}\label{rmk:alg-defect}
We do not define the analog of $\Lie(\mc{T})^* \subset \Lie(\mc{T})$ in this setting, as $\Lie_S(A)$ may not be a locally free $(\cO_S \tensor_{\Z} \cO_F)$-module of rank 1 when the discriminant is not invertible in $\cO_S$. For the same reason, it need not be the case that the sheaf $\omega_{\chi_{\infty}}$ is a line bundle!
\end{remark}

\begin{prop}\label{prop:alg-hmf}
For any algebraic Hecke character $\chi \from I_{\mf{n}} \times \Res_{F \tensor \C/\C} \G_{m, \C} \to \G_{m, \C}$, the analytification map
\[ H^0(\Lie(\mc{A})_{\C}, \cO_{\Lie(\mc{A})})^{\chi_{\infty}^{-1}} \to H^0(\Lie(\mc{T}), \cO_{\Lie(\mc{T})})^{(\chi_{\infty}^{\mr{an}})^{-1}}_{\mr{mod}} \] is an isomorphism, compatible with the action of $I_{\mf{n}}$. In particular,
\[ H^0(\Lie(\mc{A})_{\C}, \cO_{\Lie(\mc{A})})^{\chi^{-1}} \to H^0(\Lie(\mc{T}), \cO_{\Lie(\mc{T})})^{(\chi^{\mr{an}})^{-1}}_{\mr{mod}}, \] 
is an isomorphism.
\end{prop}

\begin{proof}
The main point is that the isomorphism $\Lie(\mc{A})^{an} \isom \Lie(\mc{T})$ is compatible with the action of $I_{\mf{n}} \times \Res_{F \tensor \C/\C} \G_{m, \C}$, so that the analytification map in question is well-defined. 

It remains to show that, for a fixed weight $\chi_{\infty}$, the map
\[  H^0(\Lie(\mc{A})_{\C}, \cO_{\Lie(\mc{A})})^{\chi_{\infty}^{-1}}_{\mr{mod}} \to H^0(\Lie(\mc{T}), \cO_{\Lie(\mc{T})})^{(\chi_{\infty}^{\mr{an}})^{-1}}_{\mr{mod}} \]
is an isomorphism. First, note that \begin{itemize}
\item $H^0(\wt{\M}'_{\C}, \omega_{\chi_{\infty}}) = \oplus H^0(\wt{\M}_{\mc{P}, \C}, \omega_{\chi_{\infty}})^{U^+}$,
\item $H^0(Sh'_{G,K}, \omega_{\chi_{\infty}})_{\mr{mod}} = \oplus H^0(Sh_{G, K_{\lambda}}, \omega_{\chi_{\infty}})^{U^+}_{\mr{mod}}$.
\end{itemize}
The analytification map $H^0(\wt{\M}_{\mc{P}, \C}, \omega_{\chi_{\infty}}) \to H^0(Sh_{G, K_{\lambda}}, \omega_{\chi_{\infty}})_{\mr{mod}}$ is known to be an isomorphism (the proof follows from the Koecher principle as well as GAGA applied to certain sheaves on the minimal compactification of $\wt{\M}_{\mc{P}, \C}$). 
\end{proof}

The following is an immediate consequence of Proposition \ref{prop:alg-hmf}:
\begin{corollary}
For an admissible weight $k$, the sheaf $\omega_{N_{F/\Q}^k}$ on $\wt{\M}$ is $G^+_{\mf{n}}$-equivariant, such that the analytification $H^0(\wt{\M}_{\C}, \omega_{N_{F/\Q}^k}) \to M_k(\mf{n}, \C)$ is an isomorphism of $G^+_{\mf{n}}$-modules.
\end{corollary}

 It is easy to see that $\omega^{\tensor k} \subset \omega_{N_{F/\Q}^k} \tensor \C$ is a $G^+_{\mf{n}}$-equivariant subsheaf, and that $\omega^{\tensor k} \tensor \C = \omega_{N_{F/\Q}^k} \tensor \C$. With the above corollary, this completes the proof of Theorem \ref{thm:alg-hodge-bundle}.

\subsection{Modular forms with coefficients}\label{subsec:hmf-coeff}

For any abelian group $B$, define \textbf{Hilbert modular forms with coefficients in $B$} by:
\[ M_k(\mf{n}, B) := H^0(\wt{\M}, \omega^{\tensor k} \tensor_{\Z} B). \] The main case of interest is for rings $R$, where $M_k(\mf{n}, R) = H^0(\wt{\M}_{R}, \omega^{\tensor k})$, but this additional generality is convenient. 

Given a ring $R$, a character $\psi \from G^+_{\mf{n}} \to R^*$, and an $R$-module $B$, define
\textbf{Hilbert modular forms with coefficients in $B$ and generalized nebentype $\psi$} by:
\[ M_k(\mf{n}, \psi, B) := M_k(\mf{n}, B)^{\psi} =  (M_k(\mf{n}, B) \tensor_{R} R(\psi^{-1}))^{G^+_{\mf{n}}}. \]

Note that \[ M_{k}(\mf{n}, B) = \cup_{B' \subset B \text{ f. gen.}} M_k(\mf{n}, B'), \qquad   M_k(\mf{n}, B_1 \oplus B_2) = M_k(\mf{n}, B_1) \oplus M_k(\mf{n}, B_2). \] In this way, we can often reduce proofs to the cases of $R = \Z$ or $R = \Z/p^k\Z$. For example, this reduction immediately implies that if $B$ is torsion-free, then $M_k(\mf{n}, B) \isom M_k(\mf{n}, \Z) \tensor B$. In particular, \[ M_k(\mf{n}, \C) = M_k(\mf{n}, \Z) \tensor \C. \]

\begin{definition}
A cusp $[C]$ is \textbf{defined over} a ring $R \subset \C$ if the image of $M_k(\mf{n}, R) \to M_k(\mf{n}, \C)  \to Q_{[C], k}(\C)$ is contained in $Q_{[C], k}(R)$ for all weights $k$.
\end{definition}

Let $\mu_m \subset \C$ be sufficiently many roots of unity to define a $\cO_F$-linear isomorphism $\cO_F/\mf{n} \isom \cO_F/\mf{n}(1)$ of $\Gal(\overline{\Q(\mu_m)}/\Q(\mu_m))$-modules. For a prime $p$, let $\mu_{m_p}$ be subgroup of prime-to-$p$ roots of unity in $\mu_m$. Define $\cO = \Z_{(p)}[\mu_{m_p}]$.

\begin{theorem}[\cite{rapoport}, \cite{chai}]\label{thm:cusp-defn}
\

\begin{enumerate}
\item The unramified cusps are defined over $\Z$.
\item The $p$-unramified cusps are defined over $\cO$.
\end{enumerate}
\end{theorem}

(Note that (1) follows from (2), as unramified cusps are $p$-unramified for all $p$.)

\begin{theorem}[$q$-expansion principle, \cite{rapoport} Thm. 6.7]\label{thm:q-principle}
There exists, for any abelian group $B$ and any unramified cusp $[C]$, a $q$-expansion map \[ \mr{qexp}_{[C]} \from M_k(\mf{n}, B) \longrightarrow Q_{[C], k}(B). \] These maps are functorial in $B$. If $\{ [C_i] \} $ is a collection of unramified cusps meeting every component, then $ M_k(\mf{n}, B) \to \oplus_i Q_{[C_i], k}(B)$ is injective.
\end{theorem}

In \S\ref{subsec:toroidal-cpt}, we will deduce these theorems from the analogous results for $\wt{\M}_{\mc{P}}$, due to \cite{rapoport}, \cite{chai}.

\begin{corollary}
$M_k(\mf{n}, \Z)$ is the subset of $M_k(\mf{n}, \C)$ consisting of HMFs whose $q$-expansions at all unramified cusps have coefficients in $\Z$.
\end{corollary}

\begin{corollary}
For any inclusion $B_1 \subset B_2$ of abelian groups,  $M_k(\mf{n}, B_1)$ is the subset of $M_k(\mf{n}, B_2)$  whose $q$-expansions at all unramified cusps have coefficients in $B_1$.
\end{corollary}
\begin{proof}
If $f \in M_k(\mf{n}, B_2)$ has $q$-expansions contained in $B_1$, then its image in $M_k(\mf{n}, B_2/B_1)$ is zero, by Theorem \ref{thm:q-principle}.

By considering the long exact sequence associated to the exact sequence of sheaves \[ 0 \to \omega^{\tensor k} \tensor B_1 \to \omega^{\tensor k} \tensor B_2 \to \omega^{\tensor k} \tensor B_2/B_1 \to 0, \] we see that 
\[ M_k(\mf{n}, B_1) = \ker(M_k(\mf{n}, B_2) \to M_k(\mf{n}, B_2/B_1)).\]
\end{proof}

\begin{remark}
In the above corollaries, one may replace ``at all unramified cusps" with ``a collection of unramified cusps meeting every component". The set of cusps $\{ (\alpha_{\lambda}, L_{\infty}) \}_{\lambda \in \Cl^+(F)}$, the ``standard $\infty$-cusps", are one such collection, and it is with these cusps that the corollaries are applied in \cite{dk}.
\end{remark}

\begin{remark}
Similar statements about $p$-unramified cusps and $\cO$-modules $B$ are also true. For example, $M_k(\mf{n}, \cO)$ is the subset of $M_k(\mf{n}, \C)$ consisting of HMFs whose $q$-expansions at all $p$-unramified cusps have coefficients in $\cO$.
\end{remark}

\section{Toroidal and minimal compactifications}\label{sec:compact}
\subsection{Partial toroidal compactifications}\label{subsec:toroidal-cpt}

The purpose of this section is to explain the theory of algebraic $q$-expansions of HMFs at $p$-unramified cusps. Most results will be easily deduced from analogous results for $\wt{\M}_{\mc{P}}$ due to Rapoport \cite{rapoport}. We also prove that the action of $G^+_{\mf{n}}$ extends to the partial toroidal compactification $\wt{\M}^{\mr{partial}}$, and use this to give an algebraic proof of Theorem \ref{thm:qexp-equivar}.

\subsubsection{Toric charts and formal $q$-expansions}\label{subsubsec:formal-model}
We briefly recall the construction of the ``toric charts" associated to a cusp $[C] \in \mr{Cusp}(\mf{n})$ following \cite{rapoport}, although our notation is not quite the same. 

Let $R^0_C = \Z[M_C] = \oplus_{b \in M_C} \Z q^b$, and $S^0(C)= \Spec(R^0_C)$. Write $\beta = (C, \sigma)$, for $\sigma \subset (M^{\dual}_C)^+ \cup \{ 0 \}$ a rational polyhedral cone. The scheme $S(\beta) = \Spec(\Z[\sigma^{\dual}]) = \Spec(R_{\beta})$ contains $S^0(C)$ as a dense open. These $S(\beta)$ are functorial in $\sigma$, in the sense that $\sigma_1 \subset \sigma_2$ corresponds to an open embedding $S(C,\sigma_1) \subset S(C, \sigma_2)$. Let $S^{\infty}(\beta) = S(\beta) - S^0(C)$, let $\wh{S}(\beta) = \Spf(\wh{R}_{\beta})$ denote the formal scheme obtained by completing $S(\beta)$ along $S^{\infty}(\beta)$, and let $\wh{S}(\beta)_{\mr{alg}} = \Spec(\wh{R}_{\beta})$.

It is possible to choose a smooth, rational polyhedral cone decomposition $\Sigma_C$ of $(M^{\dual}_C)^+ \cup \{ 0 \}$, functorial in $C$. This is obtained by choosing, for a fixed $C$ in each isomorphism class, a $U_C$-admissible, smooth, rational polyhedral cone decomposition. We in fact require $\Sigma_C$ to be admissible for the larger group $U^+ \subset \Aut(M^{\dual}_C)$.

Fixing a cone $C$, and varying over $\sigma \in \Sigma_C$, we glue together the schemes $S(\beta) = S((C,\sigma))$ to obtain a scheme $S_C$. Similarly we obtain a formal scheme $\wh{S}_C$ whose underlying reduced scheme we denote $S^{\infty}_C$.

The quotient $\wh{Z}_{C} := \wh{S}_C/U_{C}$ exists as a formal scheme. Let $Z_{C} := S^{\infty}_C/U_{C}$ denote the underlying reduced scheme of $\wh{Z}_{C}$. As $\wh{Z}_{C}$ depends only on $[C]$, up to canonical isomorphism, we can define $\wh{Z}_{[C]} = \lim_{C' \in [C]} \wh{Z}_{C}$.

\begin{theorem}[\cite{rapoport} \S 4]
The morphism $Z_{[C]} \to \Spec(\Z)$ is proper and geometrically connected.
\end{theorem}

Define \begin{equation}
\omega := N_{F/\Q}(\mf{a})^{\tensor -1} \tensor \cO_{\wh{S}_C}, \end{equation} a trivial line bundle on $\wh{S}_C$. The sheaf $\omega$ is $\Aut(C)$-equivariant (being defined canonically in terms of $C$), and the center $U_{1,\mf{n}} \subset \Aut(C)$ acts trivially on $\omega^{\tensor k}$ if $k$ is an admissible weight. Therefore, for $k$ admissible, $\omega^{\tensor k}$ descends to a (possibly non-trivial) line bundle on $\wh{Z}_C$.

\begin{theorem}[\cite{rapoport} \S 4]
For any abelian group $B$,
\begin{enumerate}
\item $H^0(\wh{Z}_{[C]}, \omega^{\tensor k} \tensor_{\Z} B) = Q_{[C], k}(B)$,
\item $H^0(Z_{[C]}, \omega^{\tensor k} \tensor_{\Z} B) = C_{[C], k}(B)$.
\end{enumerate}
\end{theorem}

\subsubsection{The action of $G^+_{\mf{n}}$ on formal $q$-expansions}\label{subsubsec:formal-action}

There is an action of the group $G^+_{\mf{n}}$ on the pair $(\wh{Z},  \omega^{\tensor k})$.

First, we define an action of $I_{\mf{n}}$ on the pairs $\beta = (C, \sigma)$, $N \cdot (C, \sigma) := (C \tensor N, \sigma_N)$. The cusp label $C \tensor N$ was defined in \S\ref{subsec:act-on-cusp}. The cone $\sigma_N \in \Sigma_{C \tensor N}$ is defined to be the image of $\sigma$ under the isomorphism $M^{\dual}_{C} \isom M^{\dual}_{C \tensor N}$. 

By picking a $U^+$-admissible cone decomposition for representatives of $\mr{Cusp}_{\infty}^p/G^+_{\mf{n}}$, and defining the remaining cone decompositions $\Sigma_C$ via the action of $I_{\mf{n}}$, this defines an action of $I_{\mf{n}}$ on $\wh{S}_{\mr{alg}} = \coprod_{[C] \in \mr{Cusp}} \wh{S}_{C, \mr{alg}}$ and an action of $G^+_{\mf{n}}$ on $\wh{Z} = \coprod_{[C] \in \mr{Cusp}} \wh{Z}_{[C]}$. Note that an element $G^+_{\mf{n}}$ acts on the irreducible components of $\wh{Z}_{[C]}$ through the homomorphism  $\psi_C \from G^+_{\mf{n}} \to U^+/U_C$. This is why we required $U^+$-admissibility --- to ensure that $U^+/U_C$ acted on the irreducible components of $\wh{Z}_{[C]}$.

The line bundle $\omega^{\tensor k}$ on $\coprod_{[C] \in \mr{Cusp}} \wh{S}_{C, \mr{alg}}$ can be made $I_{\mf{n}}$-equivariant by, for $N$ a prime-to-$\mf{n}$ fractional ideal, using the isomorphism
\begin{align*}
[N]^*(\omega_{\wh{S}_{(C, \sigma)}}^{\tensor k}) &\isom N_{F/\Q}(\mf{a})^{\tensor -k} \tensor  [N]^*\cO_{\wh{S}_{(C, \sigma)}} \\
&\isom N_{F/\Q}(\mf{a})^{\tensor -k}  \tensor \cO_{\wh{S}_{(C \tensor N, \sigma_N)}}  \\
&\isom N_{F/\Q}(\mf{a} N)^{\tensor -k}  \tensor \cO_{\wh{S}_{(C \tensor N, \sigma_N)}} \\
& \isom \omega_{\wh{S}_{(C \tensor N, \sigma_N)}}^{\tensor k},
\end{align*}
where the third isomorphism is given by multiplication by $|N_{F/\Q}(N)|^{-k}$. This descends to a $G^+_{\mf{n}}$-equivariant structure on $(\wh{Z}, \omega^{\tensor k})$.

This defines an action of $G^+_{\mf{n}}$ on $H^0(\wh{Z}, \omega^{\tensor k}) = Q_k$, which can be easily seen to agree with the action defined in \S \ref{subsubsec:action-q}.

Now fix a weight $k$, and define $\wh{Z}^* = \coprod_{[C] \in \mr{Cusp}^*} \wh{Z}_{[C]} \subset \wh{Z}$, the formal scheme corresponding to the $k$-admissible cusps. The group $G^+_{\mf{n}}$ acts on $\wh{Z}^*$, and hence $\cO_{\wh{Z}^*}$ is a $G^+_{\mf{n}}$-equivariant line bundle. Define an alternate $G^+_{\mf{n}}$-equivariant structure on this line bundle, denoted $\cO_{\wh{Z}^*}(\sgn)$, where we twist by the sign character $\sgn_{[C]}$ of each stabilizer $Stab_{[C]}$.

It is easy to see that:
\begin{equation}\label{eqn:equivar-triv}
\omega^{\tensor k}_{\wh{Z}^*} \isom \cO_{\wh{Z}^*}(\sgn)
\end{equation} as $G^+_{\mf{n}}$-equivariant sheaves, where this isomorphism depends on a choice of representatives for the quotient $\mr{Cusp}^*/G^+_{\mf{n}}$.

\subsubsection{Mumford's construction}
Mumford's construction of semi-abelian schemes gives rise to the following data (\cite{rapoport}, \cite{mumford}). Functorially associated to pairs $\beta = (C = (\alpha, L), \sigma)$, where $C$ is a $p$-unramified cusp label and $\sigma \in \Sigma_C$ is a cone in the cone decomposition chosen in \S\ref{subsubsec:formal-model}, there exists:

\begin{enumerate}
\item a semi-abelian scheme \[ A(\beta) \to \wh{S}(\beta)_{\mr{alg}} \] with real multiplication $\iota \from \cO_F \into \End_{\wh{S}(\alpha)_{\mr{alg}}}(A(\beta))$, whose restriction to $\wh{S}^0(\beta)_{\mr{alg}}$, denoted $A^0(\beta)$, is abelian, and whose restriction to $S^{\infty}(\beta)$ equals the constant torus $\G_m \tensor_{\Z} \mf{a}$;
\item a torsion point $x_{\beta} \from (\cO_F/\mf{n}(1))_{\cO} \into A^0(\beta)[\mf{n}]_{\cO}$;
\item a map $\wh{S}^0(\beta)_{\mr{alg}, \cO} \to \wt{\M}_{\mc{P}_{\alpha}, \cO}$ such that $(A^0(\beta), x_{\beta})$ is pulled back from the universal abelian scheme (with its torsion point) on $\wt{\M}_{\mc{P}_{\alpha}, \cO}$;
\item an isomorphism $\omega_{A(\beta)} \isom N_{F/\Q}(\mf{a})^{\tensor -1} \tensor \cO_{\wh{S}(\beta)_{\mr{alg}}}$.
\end{enumerate}

\begin{eg} 
In the analytic setting, it is easy to define an RM torus over an open subdomain of $S^0(C)^{\mr{an}} = M_C^{\dual} \tensor \C^*$, analogous to $A^0(\beta) \to \wh{S}^0(\beta)_{\mr{alg}}$.

Consider $\tau \in M_C^{\dual} \tensor \C = \Hom_{\cO_F}(\mf{b}, \C \tensor \mf{a})$ such that $\mathrm{Im}(\tau)$ is contained in $(M_C^{\dual})^+$. Taking the exponential $q = e^{2 \pi i \tau}$, we obtain a homomorphism $q \from \mf{b} \to \C^* \tensor \mf{a}$ which in fact has discrete image. For such $q$, we have an associated RM torus
\[ A(\C)_q = (\C^* \tensor \mf{a})/q(\mf{b}). \]
These define a family of RM tori over an open subdomain of $S^0(\C)^{\mr{an}}$. 

\end{eg}

\subsubsection{Partial toroidal compactification}

Fix a prime $p$, let $\cO = \Z_{(p)}[\mu_{m_p}]$ be as in \S\ref{subsec:hmf-coeff}, and fix an embedding $\cO \into \C$. Define $\wh{Z}_p = \coprod_{[C] \in \mr{Cusps}_p} \wh{Z}_{[C]}$, the disjoint union of the formal models of the $p$-unramified cusps.

Using Artin's Approximation Theorem, it is possible to glue the formal models $\wh{Z}_p$ onto $\wt{\M}$, where the gluing occurs along algebraic models of the charts $\wh{S}^0(\beta)_{\mr{alg}, \cO}$ arising in Mumford's construction. In this way, one obtains a partial compactification $\wt{\M}^{\mr{partial}}$ of $\wt{\M}$. More precisely:

\begin{theorem}\label{thm:toroidal-cpt}
There exists an algebraic stack $\wt{\M}^{\mr{partial}}$ over $\Z_{(p)}$ and a line bundle $\omega_{\mr{can}}$ on $\wt{\M}^{\mr{partial}}$ such that
\begin{enumerate}
\item $\wt{\M}_{\Z_{(p)}}$ is a dense open substack of $\wt{\M}^{\mr{partial}}$;
\item $\wt{\M}^{\mr{partial}}_{\cO} - \wt{\M}_{\cO} \isom Z_{p, \cO}$;
\item the formal scheme given by completing $\wt{\M}^{\mr{partial}}_{\cO}$ along $Z_{p, \cO}$ is isomorphic to $\wh{Z}_{p, \cO}$;
\item the restriction of $\omega_{\mr{can}}$ to $\wh{Z}_{p, \cO}$ is isomorphic to $\omega$.
\end{enumerate}
\end{theorem}

\begin{proof}
We begin by sketching a proof of the analogous theorem for $\wt{\M}_{\mc{P}} \subset \wt{\M}_{\mc{P}}^{\mr{partial}}$, originally due to Rapoport (\cite{rapoport} Thm 6.18), although we instead use the technique of Faltings--Chai \cite{faltings-chai}. 

Recall that $\wh{S}(\beta)_{\mr{alg}} = \Spec(\wh{R}_{\beta})$, where $\wh{R}_{\beta}$ is the algebraic completion of $R_{\beta}$ along the closed subscheme $S^{\infty}_{\beta}$.  As an application of the Artin approximation theorem, there exists a ring $R_{\beta}^{approx}$, such that: \begin{itemize}
\item $R_{\beta}\subset R_{\beta}^{approx} \subset R_{\beta}^{alg}$;
\item $\Spec(R_{\beta}^{approx}) \to \Spec(R_{\beta})$ is an \'etale morphism whose image equals an open neighborhood of $S^{\infty}_{\beta}$;
\item Mumford's construction $A(\beta) \to \wh{S}(\beta)_{\mr{alg}}$ (along with all auxiliary data) descends to a semiabelian scheme $A(\beta)^{approx}$ over $\Spec(R_{\beta}^{approx})$, in the sense that \[ A(\beta) = A(\beta)^{approx} \times_{\Spec(R_{\beta}^{approx})} \Spec(R_{\beta}^{alg}). \]
\end{itemize}
These algebraic models can be glued onto $\wt{\M}_{\mc{P}}$ to define the algebraic stack $\wt{\M}_{\mc{P}}^{\mr{partial}}$. See \cite{faltings-chai} pg. 109 for details of this construction in the case of Siegel moduli space (which does not include $\wt{\M}_{\mc{P}}$), or \cite{lan-compact} \S 6 for the general setting of PEL Shimura varieties (which includes $\wt{\M}_{\mc{P}}$).

The group $U^+$ acts on $\wt{\M}_{\mc{P}}^{\mr{partial}}$ (in fact on the semi-abelian variety over $\wt{\M}_{\mc{P}}^{\mr{partial}}$), extending the action on $\wt{\M}_{\mc{P}}$ discussed in \S \ref{subsec:G-moduli}. We then define $\wt{\M}^{\mr{partial}}$ to be the rigidification of $\coprod \wt{\M}^{\mr{partial}}_{\mc{P}}/U^+$.

 To see that the boundary of $\wt{\M}^{\mr{partial}}$ is as claimed, we consider the action of the group $D_{\mf{n}}$ on $\wt{\M}_{\mc{P}}^{\mr{partial}} - \wt{\M}_{\mc{P}}  = \coprod_{C} Z'_{C}$. Here $C$ runs over $\coprod_{\lambda} (\Gamma_{1,\lambda}(\mf{n}) \cap \det^{-1}(\Q)) \backslash \bP^1(F)$. We must verify that $D_{\mf{n}}$ acts freely on $\coprod_{C} Z'_{C}$. A cusp $C$ for  $\wt{\M}_{\mc{P}}$ has a formal scheme $\wh{Z}'_C = \wh{S}(C)/U'_C$, for $U'_C := \{ \mat{u}{0}{0}{u^{-1}} \in \Aut(C) \}/Z \subset U_C$. The stabilizer of the cusp $C$ for $\wt{\M}_{\mc{P}}$, in $D_{\mf{n}}$, is equal to \[ P_C/(Z \cdot (P_C \cap \mr{det}^{-1}(\Q)) = \mr{im}(P_C \to \Aut(M_C))/\mr{im}(P_C \cap \mr{det}^{-1}(\Q) \to \Aut(M_C)) = U_C/U'_C. \] Then, by assumption of $U_C$-admissibility, this group acts freely on the set of irreducible components of $Z'_C$ (such that $g \cdot W \cap W = \emptyset$ for any irreducible component $W$), with quotient equal to $Z_C$, and $\wh{Z}'_C/Stab_{[C]}(D_{\mf{n}}) = \wh{Z}_C$.

\end{proof}

\begin{remark}
If one restricts to the unramified cusps, the analogous result can be stated over $\Z$, with no need for base-change.
\end{remark}

\begin{remark}
When $p \nmid \mf{n}$, $\wt{\M}^{\mr{partial}}$ is proper.
\end{remark}

Rapoport \cite{rapoport} proved the following version of the Koecher principle:
\begin{theorem}
For any abelian group $B$,
  \[ H^0(\wt{\M}^{\mr{partial}}, \omega^{\tensor k}_{\mr{can}} \tensor B) \isom H^0(\wt{\M}, \omega^{\tensor k} \tensor B). \]
\end{theorem}
\begin{proof}
The general case reduces to the cases $B = \Z$ or $B = \Z/p^n\Z$. These cases were proven by Rapoport (\cite{rapoport}, Prop. 4.9) for $\wt{\M}^{\mr{partial}}_{\mc{P}}$, but the same proof works for $\wt{\M}^{\mr{partial}}$.
\end{proof}

Associated to $[C] \in \mr{Cusp}_p$ is a boundary divisor $i_{[C]} \from Z_{[C], \cO} \into \wt{\M}^{\mr{partial}}_{\cO}$, with formal neighborhood $\wh{Z}_{[C], \cO}$. For any $\cO$-module $B$, the composition \[ H^0(\wt{\M}_{\cO}, \omega^{\tensor k} \tensor B) \isom H^0(\wt{\M}^{\mr{partial}}_{\cO}, \omega^{\tensor k}_{\mr{can}} \tensor B) \nmto{\wh{i}_{[C]}^* } H^0(\wh{Z}_{[C], \cO}, \omega^{\tensor k}_{\mr{can}} \tensor_{\cO} B) \into Q_{C, k}(B) \]
defines an \textbf{algebraic $q$-expansion}
\[ \mr{qexp}_{[C]} \from M_k(\mf{n}, B) \to Q_{[C], k}(B). \]

When $B = \C$:
\begin{theorem}[\cite{rapoport}, \cite{lan-compare}]\label{thm:compare-q-expns}
The algebraic and complex-analytic $q$-expansions (\S \ref{subsec:q-cusplabel}) at the cusp $[C]$ are equal. 
\end{theorem}
\begin{proof}
This was originally proven in \cite{rapoport} \S 6.13, for the level $\Gamma(n)$ stack $\M(\Gamma(n))_{\mc{P}}$, and one can deduce the result for $\wt{\M}_{\mc{P}}$ from this. Alternately, one can use \cite{lan-compare} Thm. 5.3.5, which proves a general result for PEL Shimura varieties, of which the result for $\wt{\M}_{\mc{P}}$ is a special case. The equality of $q$-expansions for $\wt{\M}$ may be checked after pullback to $\coprod_{\mc{P}} \wt{\M}_{\mc{P}}$.
\end{proof}

This immediately proves Theorem \ref{thm:cusp-defn}, that the image of the classical $q$-expansion $M_k(\mf{n}, \cO) \to Q_{[C], k}(\C)$ is contained in $Q_{[C], k}(\cO)$.

\begin{prop}\label{prop:extend-action}
The action of $G^+_{\mf{n}}$ on $(\wt{\M}_{\cO}, \omega^{\tensor k})$ extends to $(\wt{\M}^{\mr{partial}}_{\cO}, \omega_{\mr{can}}^{\tensor k})$, compatible with the action on $(\wh{Z}_p, \omega^{\tensor k})$ described in \S\ref{subsubsec:formal-model}.
\end{prop}

\begin{proof}
We begin by introducing the charts which are used in the construction of $\wt{\M}^{\mr{partial}}_{\cO}$.

The construction of the toroidal compactification of $\wt{\M}_{\mc{P}}$, via Mumford's construction of semi-abelian schemes, gives rise to the following data. Functorially associated to pairs $\beta = (C = (\alpha, L), \sigma)$, where $[C] \in \mr{Cusp}_p$ and $\sigma \in \Sigma_C$, there exists:
\begin{enumerate}
\item a semi-abelian scheme \[ A(\beta) \to \wh{S}(\beta)_{\mr{alg}} \] with real multiplication $\iota \from \cO_F \into \End_{\wh{S}(\alpha)_{\mr{alg}}}(A(\beta))$, whose restriction to $\wh{S}^0(\beta)_{\mr{alg}}$, denoted $A^0(\beta)$, is abelian, and whose restriction to $S^{\infty}(\beta)$ equals the constant torus $\G_m \tensor_{\Z} \mf{a}$;
\item a torsion point $x_{\beta} \from (\cO_F/\mf{n}(1))_{\cO} \into A^0(\beta)[\mf{n}]_{\cO}$;
\item a map $\wh{S}^0(\beta)_{\mr{alg}, \cO} \to \wt{\M}_{\mc{P}_{\alpha}, \cO}$ such that $(A^0(\beta), x_{\beta})$ is pulled back from the universal abelian scheme (with its torsion point) on $\wt{\M}_{\mc{P}_{\alpha}, \cO}$;
\item an isomorphism $\omega_{A(\beta)} \isom N_{F/\Q}(\mf{a})^{\tensor -1} \tensor \cO_{\wh{S}(\beta)_{\mr{alg}}}$.
\end{enumerate}

 The map $\wh{S}^0(C, \sigma)_{\mr{alg}, \cO} \to \wt{\M}_{\mc{P}_{\alpha}, \cO} \to \wt{\M}_{\cO}$ extends to a map $\wh{S}(C, \sigma)_{\mr{alg}, \cO} \to \wt{\M}^{\mr{partial}}_{\cO}$, such that restriction to $S^{\infty}(C, \sigma)_{\cO}$ equals $S^{\infty}(C, \sigma)_{\cO} \to Z_{C, \cO} \subset \wt{\M}^{\mr{partial}}_{\cO}$.

Let $N \in I_{p\mf{n}}$ be a prime-to-$p\mf{n}$ fractional ideal. We claim that
\[ \begin{tikzcd}
 A(C, \sigma) \tensor N \arrow{r}{\isom} \arrow{d} & A(C \tensor N, \sigma_N) \arrow{d} \\
\wh{S}(C, \sigma)_{\mr{alg}} \arrow{r}{\isom} & \wh{S}(C \tensor N, \sigma_N)_{\mr{alg}} 
\end{tikzcd}, \]
and that $x_{(C, \sigma)} \tensor N$ is identified with $x_{(C \tensor N, \sigma_N)}$. The former follows from Mumford's construction. The identification of torsion points is a bit subtle - it follows from the fact that these torsion points are \emph{chosen} in order to make Theorem \ref{thm:compare-q-expns} hold. 

This implies that $\wh{S}^0(\beta)_{\mr{alg}, \cO} \to \wt{\M}_{\cO} \nmto{[N] \cdot} \wt{\M}_{\cO}$ equals $\wh{S}^0(N \cdot \beta)_{\mr{alg}, \cO} \to \wt{\M}_{\cO}$.

We have shown that there is an action of $I_{p\mf{n}}$ on the scheme $\coprod \wh{S}(\beta)_{\mr{alg}, \cO}$ such that the morphism $\coprod \wh{S}^0(\beta)_{\mr{alg}, \cO} \to \wt{\M}_{\cO}$ is $I_{p\mf{n}}$-equivariant. Although the construction of $\wt{\M}^{\mr{partial}}_{\cO}$ involves choosing ``approximations" of these schemes and morphisms, it is not difficult to see that the resulting stack $\wt{\M}^{\mr{partial}}_{\cO}$ is independent of such choices. Hence the group $I_{p\mf{n}}$ acts on $\wt{\M}^{\mr{partial}}_{\cO}$.

The fact that $\omega^{\tensor k}_{\mr{can}}|_{\wh{Z}_{p, \cO}}$ agrees with the line bundle $\omega^{\tensor k}$ we defined directly on $\wh{Z}$ follows from property 4 of $A_{(C, \sigma)}$. We omit the proof that the line bundle $\omega^{\tensor k}_{\mr{can}}$ is $I_{p\mf{n}}$-equivariant, and that $\omega_{\mr{can}}^{\tensor k}$ is an $(I_{p\mf{n}}, (F^*)')$-equivariant line bundle (in the sense of \S\ref{subsec:equiv-line}).

\end{proof}

This implies:
\begin{corollary}\label{cor:qexp-equivar}
For any $\cO$-module $B$, the $q$-expansion at the $p$-unramified cusps, 
\[ M_{k}(\mf{n}, B) \into Q_{p, k}(B) \]
is $G^+_{\mf{n}}$-equivariant, for the action of $G^+_{\mf{n}}$ described in \S\ref{subsubsec:formal-action}.
\end{corollary}

Choosing $p$ coprime to $\mf{n}$, so that $\mr{Cusp}_{p} = \mr{Cusp}$, and taking $B = \C$, viewed as an $\cO$-module via the fixed $\cO \subset \C$, this proves Theorem \ref{thm:qexp-equivar}.

\subsubsection{Proof of the $q$-expansion principle}
The following argument is identical to the one given by Rapoport (\cite{rapoport} Thm. 6.7) for the stack $\M(\Gamma(n))_{\mc{P}}$, and we include it for the sake of exposition.

For the moment, consider the partial toroidal compactification $\wt{\M}^{\mr{partial}}$ over $\Z$, whose boundary components correspond to the unramified cusps.

We obtain, for any abelian group $B$, and any unramified cusp $[C]$, a $q$-expansion map
\begin{align*} 
M_k(\mf{n}, B) &= H^0(\wt{\M}, \omega^{\tensor k} \tensor_{\Z} B) \\
&= H^0(\wt{\M}^{\mr{partial}}, \omega^{\tensor k}_{\mr{can}} \tensor_{\Z} B) \\
& \to H^0(\wh{Z}_{[C]}, \omega^{\tensor k} \tensor_{\Z} B) \\
&= Q_{[C],k}(B). \end{align*}
This is clearly functorial in $B$. To prove the $q$-expansion principle, it remains to show that,  if the collection of cusps $\{ [C_i] \}$ meets every component, $M_k(\mf{n}, B) \to \oplus_i Q_{[C_i],k}(B)$ is injective. As $B$ is a union of its finitely generated subgroups, we reduce to the case $B = \Z$ or $B = \Z/p^n\Z$.

Suppose $R = \Z/p^n \Z$ or $R = \Z$, and consider sections of the line bundle $\omega^{\tensor k}_{\mr{can}}$ over $\wt{\M}^{\mr{partial}}_R$. If a section of a line bundle is zero in any stalk of an irreducible noetherian algebraic stack, then it is zero. The result then follows from the fact that every irreducible component of $\wt{\M}^{\mr{tor}}_{\Spec(R)}$ contains some divisor $Z_{[C_i]}$, since, by Theorem \ref{thm:G*-moduli}, the connected components of $\wt{\M}^{\mr{tor}} = (\wt{\M}^{\mr{tor}}_{\mc{P}}/D_{\mf{n}})_{\Spec(R)}$ are the same as the irreducible components.

\subsection{The minimal compactification}\label{subsec:minimal-cpt}

In this section, we define the minimal compactification of $\wt{\M}$, a projective scheme $\wt{M}^{\mr{min}}$ compactifying the coarse moduli space $\wt{M} = [\wt{\M}]$. We closely follow Chai \cite{chai}, who constructed the minimal compactification for $\wt{\M}_{\mc{P}}$ as an application of the more general techniques developed in Faltings--Chai \cite{faltings-chai}. We then define coherent sheaves on $\wt{M}^{\mr{min}}$ whose sections are Hilbert modular forms. Finally, we check a compatibility between the constant term map and the evaluation of HMFs at cusp points in $\wt{M}^{\mr{min}}$.

\subsubsection{Positivity of the Hodge bundle}

\begin{theorem}[\cite{moret-bailly} Ch. IX, Thm. 2.1]\label{thm:hodge-semiample}
Consider a semi-abelian scheme $A \to S$ over a normal, excellent, noetherian base $S$, such that $A$ is abelian over a dense open subspace of $S$. Then $\omega_A$ is a semi-ample line bundle.
\end{theorem}
According to Faltings--Chai \cite{faltings-chai} (Ch. 5, Prop. 2.1), the same proof works more generally for $S$ an algebraic space covered by normal, excellent, noetherian schemes. We will need this stronger result. 
 
 \begin{theorem}[\cite{faltings-chai} Ch. 5, Prop. 2.2]\label{thm:hodge-nontrivial}
Suppose $S$ is a proper curve over an algebraically closed field, and $A \to S$ is a semi-abelian scheme, abelian over a dense open subspace of $S$. If $\deg(\omega_A) \leq 0$, then $A \to S$ is isotrivial and is abelian over $S$.
\end{theorem}

\subsubsection{Minimal compactification at prime-to-$p$ level}\label{subsubsec:good-level}

In this section, assume that $\mf{n}$ is relatively prime to a fixed prime $p$. Let $\cO$ be the ring defined in \S \ref{subsec:hmf-coeff}, a $\Z_{(p)}$-algebra containing enough roots of unity. All stacks/schemes will be $\Z_{(p)}$-schemes.

Since $\mf{n}$ is relatively prime to $p$, all cusps in $\mr{Cusp}(\mf{n})$ are $p$-unramified, and so: 
\begin{theorem}[\cite{rapoport}]
$\wt{\M}^{\mr{partial}}$ is proper over $\Spec(\Z_{(p)})$.
\end{theorem} Hence we write $\wt{\M}^{\mr{tor}} := \wt{\M}^{\mr{partial}}$. Define $\wt{M}^{\mr{min}} = \Proj(\oplus_{k \geq 0} H^0(\wt{\M}^{\mr{tor}}, (\omega_{\mr{can}}^{\tensor 2})^{\tensor k}))$. 

\begin{theorem}[\cite{chai}]\label{thm:min-cpt}
\

\begin{enumerate}
\item For some $k_0 > 0$, the line bundle $\omega_{\mr{can}}^{\tensor 2 k_0}$ on $\wt{\M}^{\mr{tor}}$ is globally generated.
\item The ring $\oplus_{k \geq 0} H^0(\wt{\M}^{\mr{tor}}, (\omega_{\mr{can}}^{\tensor 2})^{\tensor k})$ is a finitely generated $\Z_{(p)}$-algebra.
\item $\wt{M}^{\mr{min}}$ is a normal, projective scheme of finite type over $\Spec(\Z_{(p)})$.
\item There is a map $\wt{f} \from \wt{\M}^{\mr{tor}} \to \wt{M}^{\mr{min}}$, such that the restriction of $\wt{f}$ to $\wt{\M}$ is identified with the coarse moduli space map $\wt{\M} \to \wt{M} := [\wt{M}]$.
\item The map $\wt{f}$ contracts the toroidal boundary $\wt{\M}^{\mr{tor}} - \wt{\M}$ to points. More precisely, after base-change to $\cO$, \[ \wt{f}((\wt{\M}^{\mr{tor}} - \wt{\M})_{\cO}) = (\wt{M}^{\mr{min}} - \wt{M})_{\cO} = \coprod_{\mr{Cusp}(\Gamma_1(\mf{n}))} \Spec(\cO). \] 
\item The push-forward $\omega_{\mr{min}}^{\tensor 2 k_0} := \wt{f}_*(\omega_{\mr{can}}^{\tensor 2 k_0})$ is an ample line bundle on $\wt{M}^{\mr{min}}$.
\end{enumerate}
\end{theorem}
\begin{proof}
By Corollary \ref{cor:semiample-descent}, it suffices to test the semi-ampleness of $\omega^{\tensor 2}_{\mr{can}}$ after pullback along a surjective finite \'etale morphism, for example $\coprod_{\Cl^+(F)} \wt{\M}_{\mc{P}}^{\mr{tor}} \to \wt{\M}^{\mr{tor}}$. By adding auxiliary $\Gamma(m)$-level to the moduli stacks $\wt{\M}_{\mc{P}}$, for $m \geq 3$ coprime to $p$, we obtain an \emph{algebraic space} $\M_{\mc{P}}(m)^{\mr{tor}}$ and finite \'etale surjective morphism $\M_{\mc{P}}(m)^{\mr{tor}} \to \M_{\mc{P}}^{\mr{tor}}$. Then apply Theorem \ref{thm:hodge-semiample} to the semi-abelian space $A \to \wt{\M}_{\mc{P}}(m)^{\mr{tor}}$ to deduce that $\omega^{\tensor 2}_{\mr{can}}$ on $\wt{\M}_{\mc{P}}(m)^{\mr{tor}}$ is semi-ample.

We apply Proposition \ref{prop:semiample} to the semi-ample $\cL = \omega^{\tensor 2}_{\mr{can}}$ on $\wt{\M}^{\mr{tor}}$, and (2), (3), (6) follow immediately.

 To deduce that $\wt{f} := f_{\cL}$ (in the notation of \S\ref{sec:ample}) is the coarse moduli space map upon restriction to $\wt{\M}$, use Theorem \ref{thm:hodge-nontrivial}. To deduce that $\wt{f}$ contracts $\wt{\M}^{\mr{tor}} - \wt{\M}$ to points, use that \[ H^0((\wt{\M}^{\mr{tor}} - \wt{\M})_{\cO}, \omega_{\mr{can}}^{\tensor 2k}|_{\wt{\M}^{\mr{tor}} - \wt{\M}}) \isom \oplus_{[C] \in \mr{Cusp}(\Gamma_1(\mf{n}))} N_{F/\Q}(\mf{a})^{\tensor -k} \tensor \cO, \] hence all global sections of $\omega_{\mr{can}}^{\tensor 2k}$ are constant along each connected component of $\wt{\M}^{\mr{tor}} - \wt{\M}$. As the fibers of $\wt{f}$ are connected (by Prop. \ref{prop:semiample}), the scheme-theoretic image $\wt{f}((\wt{\M}^{\mr{tor}} - \wt{\M})_{\cO})$ equals $\coprod_{\pi_0(\wt{\M}^{\mr{tor}} - \wt{\M})} \Spec(\cO)$.
\end{proof}

In particular, for any $\mf{n}$, there is a full toroidal compactification $\wt{\M}^{\mr{tor}}_{\Q}$ over $\Q$, a projective minimal compactification $\wt{M}^{\mr{min}}_{\Q}$, and a morphism $\wt{f}_{\Q} \from \wt{\M}^{\mr{tor}}_{\Q} \to \wt{M}^{\mr{min}}_{\Q}$.

\subsubsection{Minimal compactification in general}

Now, suppose that $\gcd(\mf{n}, p) \neq 1$. We will define the minimal compactification $\wt{M}^{\mr{min}}$, a projective scheme over $\Z_{(p)}$.

Let $\M$ denote the level $\Gamma(1) = \Gamma_1(1)$ moduli stack over $\Spec(\Z_{(p)})$. By the theory of compactification at prime-to-$p$ level, there is a proper toroidal compactification $\M^{\mr{tor}}$, a proper minimal compactification $M^{\mr{min}}$, and a surjective morphism $f \from \M^{\mr{tor}} \to M^{\mr{min}}$.

Consider the $\Gamma_1(\mf{n})$ moduli stack $\wt{\M}$. Define the scheme $\wt{M}^{\mr{min}}$ as the normalization of $M^{\mr{min}}$ relative to the $G_{\mf{n}}^+$-equivariant morphism $\wt{\M}^{\mr{partial}} \to \M^{\mr{tor}} \to M^{\mr{min}}$. This gives $G_{\mf{n}}^+$-equivariant morphisms
\[ \wt{f} \from \wt{\M}^{\mr{partial}} \to \wt{M}^{\mr{min}},\  h \from \wt{M}^{\mr{min}} \to M^{\mr{min}}, \]
where $h$ is a finite morphism, and $\cO_{\wt{M}^{\mr{min}}}$ is integrally closed in $\wt{f}_*(\cO_{\wt{\M}^{\mr{partial}}})$. 
 Define $\wt{M}^{\mr{partial}} \subset \wt{M}^{\mr{min}}$ to be the scheme-theoretic image of $\wt{f}$.

As $\wt{M}^{\mr{min}}$ is a scheme, $\wt{f}|_{\wt{\M}}$ factors through the coarse moduli space $\wt{\M} \to [\wt{\M}] \to \wt{M}^{\mr{min}}$, where $[\wt{\M}] \to \wt{M}^{\mr{min}}$ is in fact an open embedding. In this way, $\wt{M}^{\mr{min}}$ is a compactification of $\wt{M} := [\wt{\M}]$.

\subsection{Integral models for Hodge sheaves}

For $k$ an admissible weight (for level $\Gamma_1(\mf{n})$), define \begin{equation}
\mc{L}_{k, \Q} := (\wt{f}_{\Q})_*((\omega_{\mr{can}}^{\tensor k})_{\Q}),
\end{equation} a coherent sheaf on $\wt{M}^{\mr{min}}_{\Q}$. There exists $k_0 > 0$ such that $\cL_{k_0, \Q}$ is an ample line bundle.

\begin{remark}
Although $(\omega_{\mr{can}}^{\tensor k_1})_{\Q} \tensor_{\cO_{\wt{\M}^{\mr{tor}}_{\Q}}} (\omega_{\mr{can}}^{\tensor k_2})_{\Q} \to (\omega_{\mr{can}}^{\tensor k_1 + k_2})_{\Q}$ is an isomorphism, the maps $\mc{L}_{k_1, \Q} \tensor_{\cO_{\wt{M}_{\Q}^{\mr{min}}}} \mc{L}_{k_2, \Q} \to \mc{L}_{k_1 + k_2, \Q}$ need not be isomorphisms.
\end{remark}

In the following lemma, we will define sheaves $\mc{L}_k$ on $\wt{M}^{\mr{min}}$ whose pull-back to the generic fiber $\wt{M}^{\mr{min}}_{\Q}$, which we will denote $\mc{L}_k \tensor \Q$, equals $\mc{L}_{k,\Q}$.

\begin{lemma}\label{lem:integral-sheaf}
For each admissible $k \geq 0$, there exists a $G^+_{\mf{n}}$-equivariant coherent sheaf $\mc{L}_k$ on $\wt{M}^{\mr{min}}$ such that:
\begin{enumerate}
\item $\mc{L}_k \tensor \Q = \mc{L}_{k, \Q}$;
\item $\mc{L}_k|_{\wt{M}^{\mr{partial}}} = \wt{f}_*(\omega_{\mr{can}}^{\tensor k})|_{\wt{M}^{\mr{partial}}}$;
\item $H^0(\wt{M}^{\mr{min}}, \mc{L}_k) \subset H^0(\wt{\M}, \omega^{\tensor k})\ (= M_k(\mf{n}, \Z_{(p)}))$;
\item if $k = i + j  k_0$, $\mc{L}_k \isom \mc{L}_i \tensor (\mc{L}_{k_0})^{\tensor j}$;
\item $\mc{L}_{k_0}$ is an ample line bundle.
\end{enumerate}
These are all maps of $G^+_{\mf{n}}$-modules/$G^+_{\mf{n}}$-equivariant sheaves.
\end{lemma}

\begin{proof}
Note that, if $k = i + j  k_0$, \begin{equation}\mc{L}_{k, \Q} \isom \mc{L}_{i,\Q} \tensor (\mc{L}_{k_0, \Q})^{\tensor j}. 
\end{equation} It will suffice to define $\mc{L}_k$ for $k = 1, \ldots, k_0$, as if $k = i + j  k_0$ for $j \geq 1$, we may define $\mc{L}_k := \mc{L}_i \tensor (\mc{L}_{k_0})^{\tensor j}$, so that property (4) holds by definition.

Choose an $m > 0$, coprime to $p \mf{n}$, such that $M(\Gamma(m))^{\mr{min}}$ has no elliptic points or non-admissible cusps (see Lemma \ref{lem:full-level}). Let $\Gamma' = \Gamma(m), \Gamma'' = \Gamma(m) \cap \Gamma_1(\mf{n})$. We have morphisms $g$, $h$:
\[ M(\Gamma')^{\mr{min}} \nmleft{g} M(\Gamma'')^{\mr{min}} \nmto{h} M(\Gamma_1(\mf{n}))^{\mr{min}}. \] We also have morphisms:
\[ f' \from \M(\Gamma')^{\mr{tor}} \to M(\Gamma')^{\mr{min}}, \]
 \[ f'' \from \M(\Gamma'')^{\mr{partial}} \to M(\Gamma'')^{\mr{min}}, \]
 \[ \wt{f} \from \M(\Gamma_1(\mf{n}))^{\mr{partial}} \to M(\Gamma_1(\mf{n}))^{\mr{min}}. \]
Define $\cL'_k := f'_*(\omega^{\tensor k}_{\mr{can}})$, $\cL''_k := g^*(\cL'_k)$, $\cL_k := h_*(\cL''_k)^{\GL_2(\cO/m\cO)}$. The condition on $m$ ensures that:
\begin{itemize}
\item $\cL'_k$ is a line bundle;
\item the restriction of $g$ to  $M(\Gamma'')^{\mr{partial}} \subset  M(\Gamma'')^{\mr{min}}$ is a finite morphism, ramified only at the cusps.
\end{itemize}
One can then verify that $\cL''_{k} \tensor \Q =  \cL''_{k, \Q}$ (where $\cL''_{k, \Q}$ is the push-forward of a sheaf on $\M(\Gamma'')^{\mr{tor}}_{\Q}$) and that $f''_*(\omega^{\tensor k}_{\mr{can}})$ equals the restriction of of $\cL''_k$ to $M(\Gamma'')^{\mr{partial}}$. Taking $\GL_2(\cO/m\cO)$-invariants gives properties (1) and (2) for the sheaves $\cL_k$, and (3) follows immediately from (2).

We give an equivalent definition of $\cL_{k_0}$. Write \[ r \from M(\Gamma_1(\mf{n}))^{\mr{min}} \to M(\Gamma(1))^{\mr{min}}, \] \[ f \from \M(\Gamma(1))^{\mr{tor}} \to  M(\Gamma(1))^{\mr{min}}. \] Then \begin{equation}\label{eq:hodge-compare}
\mc{L}_{k_0} = r^*(f_*(\omega^{\tensor k_0}_{\mr{can}})).
\end{equation} Since $r$ is a finite morphism, and $f_*(\omega^{\tensor k_0}_{\mr{can}})$ is an ample line bundle, the same is true for $\cL_{k_0}$. This proves (5).

\end{proof}

\subsection{Constant terms and cusp points}

Fix a weight $k \geq 0$. For any $p$-unramified $k$-admissible cusp $[C] \in \mr{Cusp}_{p}(\mf{n})^*$, corresponding to the cusp point $x_{[C]} \from \Spec(\cO) \to \wt{M}^{\mr{min}}_{\cO}$, we consider the restriction map $x_{[C]}^* \from \mc{L}_k \to x_{[C]}^*(\mc{L}_k)$. We want to relate this to analogous map for the toroidal compactification, $i_{[C]}^* \from \omega_{\mr{can}}^{\tensor k} \to i_{[C]}^*(\omega_{\mr{can}}^{\tensor k})$, and hence to constant terms of HMFs.

\begin{lemma}\label{lem:base-change-admiss}
The base-change map $x_C^* \wt{f}_*(\omega^{\tensor k}_{\mr{can}})  \to \wt{f}_* i_C^*(\omega^{\tensor k}_{\mr{can}})$ is an isomorphism of sheaves on $\wt{M}^{\mr{partial}}$.
\end{lemma}
\begin{proof}
After passing to the formal completion along $x_{C}$, this map is identified with $Q_{C,k}(\cO) \tensor Q_{C,0}(\cO)/J_C \to C_{C, k}(\cO)$, where $J_{C} \subset Q_{C, 0}(\cO_E)$ is ideal $J_{C} = \ker(Q_{C,0}(\cO) \to C_{C,0}(\cO))$. As $C$ is $k$-admissible, $Q_{C,k}(\cO) \isom N_{F/\Q}(\mf{a})^{\tensor -k}  \tensor Q_{C, 0}(\cO)$ and $C_{C,k}(\cO) \isom N_{F/\Q}(\mf{a})^{\tensor -k} \tensor C_{C,0}(\cO)$.
\end{proof}

It is clear that the following diagram commutes:
\[ \begin{tikzcd}
H^0(\wt{M}^{\mr{min}}_{\cO}, \mc{L}_k) \arrow{r}{x_C^*} \arrow{d} & H^0(\Spec(\cO), x_C^*\mc{L}_k) \arrow{d} \\
H^0(\wt{M}^{\mr{partial}}_{\cO}, \mc{L}_k) \arrow{r}{x_C^*} \arrow{d}{\wt{f}^*} & H^0(\Spec(\cO), x_C^*\mc{L}_k) \arrow{d}{\wt{f}^*} \\
H^0(\wt{\M}^{\mr{partial}}_{\cO}, \omega^{\tensor k}_{\mr{can}}) \arrow{r}{i_C^*} \arrow{d} & H^0(Z_{C}, i_C^*\omega_{\mr{can}}^{\tensor k}) \arrow{d} \\
M_k(\mf{n}, \cO) \arrow{r}{\mr{const}_C} & C_{C,k}(\cO) \\
\end{tikzcd}. \]
Moreover, all of the vertical arrows in the right column are isomorphisms - for $\wt{f}^*$, this follows from the above lemma.

\begin{corollary}\label{cor:const-term-min}
For any $k$-admissible $p$-unramified cusp label $C$, and any $\cO$-module $B$, there is an identification $H^0(\wt{M}^{\mr{min}}_{\cO}, x_C^*(\mc{L}_k \tensor B)) \isom C_{C,k}(B)$ such that $x_C^* \from H^0(\wt{M}^{\mr{min}}_{\cO}, \mc{L}_k \tensor B) \to C_{C,k}(B)$ factors as
\[ H^0(\wt{M}^{\mr{min}}_{\cO}, \mc{L}_k \tensor B) \subset M_k(\mf{n}, B) \nmto{\mr{const}_C} C_{C,k}(B). \]
\end{corollary}
\begin{proof}
When $B = \cO$, this follows from the above discussion. The general case uses $x_C^*(\mc{L}_k \tensor_{\cO} B) \isom x_C^*(\mc{L}_k) \tensor_{\cO} B$, which follows from the fact that $\mc{L}_k$ is locally free when localized to an admissible cusp.
\end{proof}

\begin{remark}
For cusp $[C]$ which is not admissible for the weight $k$, the evaluation map \[ x_{[C]}^* \from H^0(\wt{M}^{\mr{min}}_{\cO}, \mc{L}_k) \to x_{[C]}^*(\mc{L}_k) \] is not related to constant terms. 
\end{remark}

\section{Main results}\label{sec:main-results}

\subsection{Equivariant lifting results}\label{subsec:equiv-lift}

We now state two propositions, Prop. \ref{prop:equivariant-surj} and Prop. \ref{prop:equivariant-lift}, which are ``abstract" versions of two of our main theorems, Theorem \ref{thm:equivar-surj} and Theorem \ref{thm:equiv-lift} (i.e. not in the setting of Hilbert modular varieties). The former proposition is used in the proof of Theorem \ref{thm:equivar-surj}, and hence we will give a proof. For technical reasons, the latter proposition is not actually used in the proof of Theorem \ref{thm:equiv-lift}, and so we do not give the proof.

Let $X \to \Spec(R)$ be a projective scheme over a noetherian ring $R$, equipped with an ample line bundle $\cL$. Let $\{ x_j \}$ be a finite set of pairwise disjoint $R$-points, and let \[ i \from Z := \coprod x_j \into X\] denote the closed embedding. A classical result is:
\begin{theorem}\label{thm:ample-vanishing}
For any coherent $\cO_X$-module $\mc{F}$, there exists $k_0 > 0$ such that for all $k \geq k_0$, 
 the map \[ H^0(X, \cL^{\tensor k} \tensor \mc{F}) \to H^0(Z, i^*(\cL^{\tensor k} \tensor \mc{F})) \] is surjective.
\end{theorem}

Suppose that a finite group $G$ acts on $X$, such that $\cL$ is a $G$-equivariant ample line bundle and $Z$ is $G$-stable. Let $\mc{I}$ be the ideal sheaf defining $Z$, and let $\wh{\cO}_{Z} := \lim \cO_X/\mc{I}^n$ denote the completion at $\mc{I}$.

\begin{prop}\label{prop:equivariant-surj}
With notation as above, let $\mc{F}$ be a $G$-equivariant coherent sheaf on $X$. Suppose that:
 \begin{enumerate}
\item $H^0(X, \mc{F} \tensor \wh{\cO}_{Z}) \surj H^0(Z, i^*(\mc{F}))$ is $G$-equivariantly split, and
\item $H^0(X, \cL \tensor \wh{\cO}_{Z}) \surj H^0(Z, i^*(\cL))$ is $G$-equivariantly split.
 \end{enumerate}

Then there exists $k_0 > 0$ such that for all $k \geq k_0$  and all characters $\psi \from G \to R^*$,
 the map \[ H^0(X, \cL^{\tensor k} \tensor \mc{F})^\psi \to H^0(Z, i^*(\cL^{\tensor k} \tensor \mc{F}))^{\psi} \] is surjective.
\end{prop}

\begin{proof}
As $X$ is projective, the schematic quotient $Y := X/G$ exists, with finite quotient map $p \from X \to Y$. For any coherent sheaf $\mc{F}$ on $X$, define $\mc{F}^G := p_*(\mc{F})^G$, a coherent sheaf on $Y$. Also, for any character $\psi \from G \to R^*$, define $\mc{F}^{\psi} = (\mc{F} \tensor R(\psi^{-1}))^G$.

Replacing $\cL$ with $\cL^{\tensor |G|}$ and replacing $\mc{F}$ with $\mc{L} \tensor \mc{F}, \ldots, \mc{L}^{\tensor |G| - 1} \tensor \mc{F}$, we may assume by Lemma \ref{lem:proj-quot} that
\begin{enumerate}
\item $\cL = p^*(\cL^G)$, and
\item $\cL^G$ is an ample line bundle on $Y$. 
\end{enumerate} The fact that this new pair $\cL,\ \mc{F}$ still satisfies conditions 1 and 2 may be deduced from the affine-ness of the scheme $Z$.

We consider the exact sequence of sheaves
\begin{equation} 0 \to (\mc{L}^{\tensor k} \tensor \mc{F})_0 \to \mc{L}^{\tensor k} \tensor \mc{F} \nmto{i^*} i^*(\mc{L}^{\tensor k} \tensor \mc{F}) \to 0,  \label{e:lfseq} \end{equation} where the subscript $(\cdot)_0$ indicates the kernel of $i^*$.
Note that since $\mc{L}$ is locally trivial, $(\mc{L}^{\tensor k} \tensor \mc{F})_0 = \mc{L}^{\tensor k} \tensor \mc{F}_0$.
We apply $\psi$ to the  sequence (\ref{e:lfseq}), to obtain 
\[ 0 \to (\mc{L}^{\tensor k} \tensor \mc{F})^{\psi}_0 \to (\mc{L}^{\tensor k} \tensor \mc{F})^{\psi} \to i^*(\mc{L}^{\tensor k} \tensor \mc{F})^{\psi} \to 0.  \] The final map is in fact surjective, even though $(-)^\psi$ is not a right-exact functor. Note that $(-)^\psi$ commutes with arbitrary flat base change, in particular with localization and completion of coherent sheaves.
%https://www.math.ru.nl/~bmoonen/BookAV/Quotients.pdf Thm 4.16
 Since $i^*(\mc{L}^{\tensor k} \tensor \mc{F})$ is supported on the subscheme $Z$, we can check the surjectivity after tensoring with the localization $\cO_{X, Z}$ along $Z \subset X$. Since the completed semi-local ring $\wh{\cO}_Z$ is faithfully flat over $\cO_{X, Z}$, it will suffice to show that the following map is surjective:
% it will then suffice to show that \[ \left((\mc{L}^{\tensor k} \tensor \mc{F})^{\psi}\right) \tensor \wh{\cO}_Z \to \left(i^*(\mc{L}^{\tensor k} \tensor \mc{F})^{\psi}\right) \tensor \wh{\cO}_Z \] is surjective. Thus
 \begin{equation}\label{e:cmplt} H^0(X, (\mc{L}^{\tensor k} \tensor \mc{F}) \tensor \wh{\cO}_{Z})^{\psi} \to H^0(Z, i^*(\mc{L}^{\tensor k} \tensor \mc{F}))^{\psi}. \end{equation} Since $\cL$ is locally free, the surjectivity of (\ref{e:cmplt}) is equivalent to that of 
\begin{equation}\label{eqn:factor} \left(H^0(X, \cL \tensor \wh{\cO}_{Z})^{\tensor k} \tensor H^0(X, \mc{F} \tensor \wh{\cO}_{Z})\right)^{\psi} \to \left(H^0(Z, i^*(\cL))^{\tensor k} \tensor H^0(Z, i^*(\mc{F}))\right)^{\psi}. \end{equation} By assumptions (1) and (2), the surjective map \begin{equation}H^0(X, \mc{L} \tensor \wh{\cO}_{Z})^{\tensor k} \tensor H^0(X, \mc{F} \tensor \wh{\cO}_{Z}) \to H^0(Z, i^*(\cL))^{\tensor k} \tensor H^0(Z, i^*(\mc{F}))\end{equation} is $G$-equivariantly split, hence remains surjective upon passing to $\psi$-invariants, and so (\ref{eqn:factor}) is surjective. %This is surjective by the assumption (1) that $H^0(X, \mc{F} \tensor \wh{\cO}_{Z}) \surj H^0(Z, i^*(\mc{F}))$ is $G$-equivariantly split.
%on completed local rings.This follows from the assumption on the existence of $G$-equivariant splittings, by passing to local rings and then to completions, using the fact that completion is faithfully flat.

The projection formula implies that $(\mc{L}^{\tensor k} \tensor \mc{F}_0)^{\psi} =  (\mc{L}^G)^{\tensor k} \tensor \mc{F}^{\psi}_0$.
For $k$ sufficiently large, $H^1(Y,  (\mc{L}^G)^{\tensor k} \tensor \mc{F}^{\psi}_0) = 0$, and so the map \[ H^0(Y, (\mc{L}^{\tensor k} \tensor \mc{F})^{\psi}) \to H^0(Z/G, (i^*(\mc{L}^{\tensor k} \tensor \mc{F}))^{\psi}) \] is surjective.

\end{proof}

\begin{prop}\label{prop:equivariant-lift}
Let $X/\Spec(R),\ G,$ and $\cL$ be as above. Moreover, fix an ideal $I \subset R$ and a character $\psi \from G \to R^*$, with reduction $\bar{\psi} \from G \to (R/I)^*$. Let $H \subset G$ denote the subgroup of elements which act trivially on $X$. Suppose that $|H|$ is invertible in $R$, and suppose that any element of $G/H$ fixes only finitely many points of $X$.  

Fix a $G$-equivariant coherent sheaf $\mc{F}$ on $X$. For $k$ sufficiently large, the cokernel of
\[ H^0(X_R, \cL^{\tensor k} \tensor \mc{F})^{\psi} \to H^0(X_{R/I}, \cL^{\tensor k} \tensor \mc{F} \tensor R/I)^{\bar{\psi}} \]
equals the cokernel of
\[ \oplus_{x \in |X|} ((\cL^{\tensor k} \tensor \mc{F})_x)^{\psi|_{Stab_x(G)}} \to \oplus_{x \in |X|} ((\cL^{\tensor k} \tensor \mc{F})_x \tensor R/I)^{\bar{\psi}|_{Stab_x(G)}}. \]
This sum may be taken over the finitely many closed points $x \in |X|$ where $Stab_x(G/H) \neq 0$. 
\end{prop}

As noted at the beginning of this section, we omit the proof of this result, as it will not be used later.

\subsection{Equivariant surjectivity of the constant term}\label{subsec:equivar-const}

Fix a prime $p$, and let $\cO = \Z_{(p)}[\mu_{m_p}]$ as in \S \ref{subsec:hmf-coeff}.  The ideal $\mf{n} \subset \cO_F$ need not be relatively prime to $p$. Let $R$ be a noetherian $\cO$-algebra, and let $\psi \from G^+_{\mf{n}} \to R^*$ be any character. We do not require $R$ to be $\Z$-flat. Let $B$ be an arbitrary $R$-module.

Consider the total constant term at the the $p$-unramified cusps, for modular forms with coefficients in $B$, denoted
\[ \mr{const}_p \from M_k(\mf{n}, B) \to C_{p,k}(B). \] 
By Corollary \ref{cor:qexp-equivar}, $\mr{const}_p$ is $G^+_{\mf{n}}$-equivariant. Defining $C_{p,k}(\psi,B) := C_{p,k}(B)^{\psi}$, and taking $\psi$-isotypic components, we obtain a map $\mr{const}_p \from M_k(\mf{n}, \psi, B)  \to C_{p,k}(\psi, B)$.

\begin{theorem}\label{thm:equivar-surj} With hypotheses as above, and for all $k$ sufficiently large, the map
 \[ \mr{const}_p \from M_k(\mf{n}, \psi, B)  \to C_{p,k}(\psi, B) \]
 is surjective.
 \end{theorem}

\begin{proof}
Since $B$ is a union of its finitely-generated submodules $B'$, and we have $M_k(\mf{n}, \psi, B) = \bigcup_{B'} M_k(\mf{n}, \psi, B')$ as well as $C_{p,k}(\psi, B) = \bigcup_{B'} C_{p,k}(\psi, B')$, we reduce to the case that $B$ is a finitely-generated $R$-module.

Fix a parity $\epsilon = (-1)^k$, so that we may talk about the subset of admissible cusps $\mr{Cusp}_{p}(\mf{n})^* \subset \mr{Cusp}(\mf{n})$ as we vary $k$.

Recall the line bundles $\mc{L}_k$ on $\wt{M}^{\mr{min}}$ defined by Lemma \ref{lem:integral-sheaf}. We apply Proposition \ref{prop:equivariant-surj} to the space $\wt{M}^{\mr{min}}_R$, with $G^+_{\mf{n}}$-equivariant ample line bundle $\mc{L}_{k_0}$, coherent sheaf $\mc{F} = \mc{L}_i \tensor_{R} B$, and with the finite collection of points in $\wt{M}^{\mr{min}}_R$ corresponding to the cusps $\mr{Cusp}_{p}(\mf{n})^* \subset \mr{Cusp}(\mf{n})$. The condition of $G^+_{\mf{n}}$-equivariant splittings follows from  (\ref{eqn:equivar-triv}) and Theorem \ref{thm:toroidal-cpt}.

As $\mc{L}_k \tensor_R B \isom \mc{L}_{k_0}^{\tensor n} \tensor (\mc{L}_i \tensor_R B)$ for some $i = 0,\ldots, k_0 - 1$, we conclude that
\[ H^0(\wt{M}^{\mr{min}}_R, \mc{L}_k \tensor_R B)^{\psi} \to (\oplus_{[C] \in \mr{Cusp}_p(\mf{n})^*} H^0(\Spec(R), x_{[C]}^*\mc{L}_k \tensor_R B) )^{\psi} \isom C_{p,k}(\psi, B) \]
is surjective for all $k$ sufficiently large. This map factors through $\mr{const}_p \from M_k(\mf{n}, \psi, B) \to C_{p,k}(\psi, B)$ (by Corollary \ref{cor:const-term-min}), proving the surjectivity.
\end{proof}

\subsection{Equivariant lifting of modular forms}\label{subsec:lifting}

We first prove a result on the lifting of modular forms at prime-to-$p$ level:
\begin{theorem}\label{thm:classic-lift}
Fix a prime $p$, and fix an ideal $\mf{n} \subset \cO_F$ coprime to $p$, such that $\wt{\M} = \M_1(\mf{n})$ is an algebraic space. Let $R$ be a noetherian $\Z_{(p)}$-algebra, and let $I \subset R$ be any ideal. If $k$ is sufficiently large (and, if $p = 2$, even), the reduction
\[ M_k(\mf{n}, R) \to M_k(\mf{n}, R/I) \]
is surjective.
\end{theorem}

\begin{remark}
This result holds for any prime-to-$p$ level structure $\Gamma$ such that $\M_{\Gamma}$ is an algebraic space.
\end{remark}

\begin{proof}
We may assume that $R$ is $\Z_{(p)}$-flat, as any noetherian $\Z_{(p)}$-algebra is the quotient of a flat noetherian $\Z_{(p)}$-algebra. 

Recall the morphism $\wt{f} \from \wt{\M}^{\mr{tor}} \to \wt{M}^{\mr{min}}$ of Theorem \ref{thm:min-cpt}. We have an exact sequence
\[ 0 \to I \wt{f}_*(\omega^{\tensor k}_{\mr{can}}) \to \wt{f}_*(\omega^{\tensor k}_{\mr{can}}) \to \wt{f}_*(\omega^{\tensor k}_{\mr{can}}) \tensor R/I \to 0. \] Localizing at the cusp points, we see that $\wt{f}_*(\omega^{\tensor k}_{\mr{can}}) \tensor R/I \to \wt{f}_*(\omega^{\tensor k}_{\mr{can}} \tensor R/I)$ has cokernel the same as $\oplus_{\mr{Cusps}} Q_{[C], k}(R) \to Q_{[C], k}(R/I)$, which is zero as long as either $p \neq 2$ or $k$ is even. 

As $I$ is $\Z_{(p)}$-flat, $I \wt{f}_*(\omega^{\tensor k}_{\mr{can}}) = \wt{f}_*(I \omega^{\tensor k}_{\mr{can}})$. The above exact sequence is then the same as
\[ 0 \to \wt{f}_*(I \omega^{\tensor k}_{\mr{can}}) \to \wt{f}_*(\omega^{\tensor k}_{\mr{can}}) \to \wt{f}_*(\omega^{\tensor k}_{\mr{can}} \tensor R/I) \to 0. \] Taking global sections gives:
\[ 0 \to I M_k(\mf{n}, R) \to M_k(\mf{n}, R) \to M_k(\mf{n}, R/I) \to H^1(\wt{M}^{\mr{min}}_{R}, \wt{f}_*(I \omega^{\tensor k}_{\mr{can}})). \]

There exists a $k_0 > 0$ such that $\wt{f}_*(\omega^{\tensor k_0}_{\mr{can}})$ is ample, and if $k = i + j k_0$, we have  \[ \wt{f}_*(I \omega^{\tensor k}_{\mr{can}}) = \wt{f}_*(I \omega^{\tensor i}_{\mr{can}}) \tensor \wt{f}_*(\omega^{\tensor k_0}_{\mr{can}})^{ \tensor j}. \] Hence, for $k$ sufficiently large, $H^1(\wt{M}^{\mr{min}}_R, \wt{f}_*(I \omega^{\tensor k}_{\mr{can}})) = 0$.
\end{proof}

We now prove a variant of this lifting result which takes the action of $G^+_{\mf{n}}$ into account. Fix a noetherian $\cO$-algebra $R$ and an ideal $I \subset R$.
\begin{definition}\label{defn:good-ring}A noetherian $\cO$-algebra $R$ is \textbf{good} if $2N_{F/\Q}(\mf{n})$ is invertible in $R$, as are the orders of all inertia groups of $\M$ (see Corollary \ref{cor:inertia-bound} for a bound on these orders).
\end{definition}

Fix a sign $\epsilon = \pm 1$, and consider a character $\psi \from G^+_{\mf{n}} \to R^*$ such that $\sgn(\psi) = \epsilon$. Denote the reduction of $\psi$ by $\bar{\psi} \from G^+_{\mf{n}} \to (R/I)^*$. Recall that $M_k(\mf{n}, \psi, R) = M_k(\mf{n}, R)^{\psi}$, and similarly define $C_k(\psi, R)$, $Q_k(\psi, R)$, $S_k(\mf{n}, \psi, R)$ to be the $\psi$-isotypic component of the constant terms, $q$-expansions, and cusp forms of weight $k$.
\begin{theorem}\label{thm:equiv-lift}
Assume that $R$ is good. For $k$ sufficiently large with $(-1)^k = \epsilon$,
the map 
\[ M_k(\mf{n}, \psi, R) \to M_k(\mf{n}, \bar{\psi}, R/I) \] has the same cokernel as \[ C_{k}(\psi, R) \to C_{k}(\bar{\psi}, R/I). \] In particular, the reduction of cusp forms, $S_k(\mf{n}, \psi, R) \to S_k(\mf{n}, \bar{\psi}, R/I)$, is surjective.
\end{theorem}

\begin{remark}
It seems likely that the hypothesis on the primes invertible in $R$ can be weakened if the values of modular forms at elliptic points, as well as at the cusps, are taken into account.
\end{remark}

\begin{proof}
It is easy to reduce to the case that $R$ is a $\Z_{(p)}$-algebra, where $p$ is relatively prime to $2N_{F/\Q}(\mf{n})$. As in \S \ref{subsubsec:good-level}, the toroidal compactification $\wt{\M}^{\mr{tor}}_{\Z_{(p)}}$ constructed in Theorem \ref{thm:toroidal-cpt} is proper, as is the morphism $\wt{f} \from \wt{\M}^{\mr{tor}}_R \to \wt{M}^{\mr{min}}_R$. There is a map $f' \from \wt{\M}^{\mr{tor}}_R/G_{\mf{n}} \to \wt{M}^{\mr{min}}_R/G_{\mf{n}}$ where $\wt{\M}^{\mr{tor}}_R/G_{\mf{n}}$ is the quotient stack, while $\wt{M}^{\mr{min}}_R/G_{\mf{n}}$ is the schematic quotient of a projective scheme.

The line-bundle $\omega^{\tensor k}_{\mr{can}} \tensor R(\psi)$ on $\wt{\M}^{\mr{tor}}_R$ is $G^+_{\mf{n}}$-equivariant, descending to a line bundle $\mc{F}_{R}$ on $(\wt{\M}^{\mr{tor}}/G^+_{\mf{n}})_R$. If $(-1)^k = \epsilon$, this line bundle in fact is pulled back from a line bundle on the rigidification $\wt{\M}^{\mr{tor}}/G_{\mf{n}}$, which we will also denote $\mc{F}_{R}$. If we did the same construction over the ring $R/I$, we would obtain a line bundle $\mc{F}_{R/I}$ on $(\wt{\M}^{\mr{tor}}/G_{\mf{n}})_{R/I}$. They are compatible with base-change: $\mc{F}_{R} \tensor_R R/I = \mc{F}_{R/I}$, hence there is an exact sequence of sheaves
\[ 0 \to I \mc{F}_{R} \to \mc{F}_{R} \to \mc{F}_{R/I} \to 0 \]
on the stack $\wt{\M}^{\mr{tor}}_R/G_{\mf{n}}$, whose global sections give
\[ 0 \to M_k(\mf{n}, \psi, I) \to M_k(\mf{n}, \psi, R) \to M_k(\mf{n}, \psi, R/I). \]

Consider the pushforward
\[ 0 \to f'_*(I \mc{F}_{R}) \to f'_*(\mc{F}_{R}) \to f'_*(\mc{F}_{R/I}). \]

Let $\mc{G} = f'_*(\mc{F}_{R/I})$ denote the image of the final map. The same argument as in Theorem \ref{thm:classic-lift} then applies to \[ 0 \to f'_*(I \mc{F}_{R}) \to f'_*(\mc{F}_{R}) \to \mc{G} \to 0, \] using the ampleness of $f'_*(\omega^{\tensor k_0}_{\mr{can}})$ to conclude that 
$H^0( f'_*(\mc{F}_{R})) \to H^0(\mc{G})$ is surjective for $k$ sufficiently large.

It remains to show that the cokernel of \[ M_k(\mf{n}, \psi, R/I) = H^0(f'_*(\mc{F}_R)) \surj H^0(\mc{G}) \to H^0(f'_*(\mc{F}_{R/I})) = M_k(\mf{n}, \psi, R/I) \] equals the cokernel of  \[ C_{k}(\psi, R) \to C_{k}(\bar{\psi}, R/I). \] 

The map $f'$ fails to be an isomorphism at the cusp points $W := \wt{M}^{\mr{min}}/G_{\mf{n}} - \wt{M}/G_{\mf{n}}$ and at the closed points in  $|\wt{M}/G_{\mf{n}}|$ whose preimage in $|\wt{\M}/G_{\mf{n}}|$ has non-trivial inertia group, i.e. the elliptic points.

The order of the inertia groups of $\wt{\M}/G_{\mf{n}}$ are invertible in $R$, since $\wt{\M}/G_{\mf{n}} \to \M/G_{1}$ is a finite \'etale morphism, and up to stabilizers of exponent 2, $G_{\mf{1}}$ acts freely on the connected components of $\M$. We have assumed that $2$, and the orders of the inertia groups of $\M$, are invertible in $R$.

As the orders of the inertia groups of $\wt{\M}/G_{\mf{n}}$ are invertible in $R$, and $f'|_{\wt{\M}/G_{\mf{n}}} \from \wt{\M}/G_{\mf{n}} \to \wt{M}/G_{\mf{n}}$ is the coarse moduli space map, we find that $f'_*(\mc{F}_{R})|_{\wt{M}/G_{\mf{n}}} \to f'_*(\mc{F}_{R/I})|_{\wt{M}/G_{\mf{n}}}$ is surjective. Thus the cokernel of $\mc{G} \to f'_*(\mc{F}_{R/I})$ is supported on $W$, and so the cokernel of $H^0(\mc{G}) \to H^0( f'_*(\mc{F}_{R/I}))$ equals the cokernel of $f'_*(\mc{F}_R)_W \to f'_*(\mc{F}_{R/I})_W$ (the localization at $W$). 

This cokernel can be computed after passing to the formal completion along $W$. Hence the cokernel of $f'_*(\mc{F}_R)_W \to f'_*(\mc{F}_{R/I})_W$ (after tensoring with $Q_0(R)$) equals the cokernel of
\[ Q_{k}(\psi, R) \to Q_{k}(\bar{\psi}, R/I). \] The stabilizer $Stab_{[C]} \subset G^+_{\mf{n}}$ acts freely on $M_{[C]}^+/U_{[C]} - 0$. Since $M_{[C]}^+/U_{[C]}-0$ indexes the non-constant terms of the $q$-expansions in $Q_{[C], k}$, it follows that \[ \ker(Q_{k}(\psi, R) \to C_{k}(\psi, R)) \to \ker(Q_{k}(\bar{\psi}, R/I) \to C_{k}(\bar{\psi}, R/I)) \]
is surjective. Hence the cokernel of $f'_*(\mc{F}_R)_W \surj \mc{G}_W \subset f'_*(\mc{F}_{R/I})_W$ is the same as that of
\[ C_{k}(\psi, R) \to C_{k}(\psi, R/I). \]
Taking global sections of the exact sequence of sheaves
\[ 0 \to \mc{G} \to f'_*(\mc{F}_{R/I}) \to \coker(C_{k}(\psi, R) \to C_{k}(\psi, R/I)) \to 0, \]
we obtain
\[ 0 \to M_k(\mf{n}, \psi, R) \tensor R/I \to M_k(\mf{n}, \bar{\psi}, R/I) \to \coker(C_{k}(\psi, R) \to C_{k}(\psi, R/I)). \]

Theorem \ref{thm:equivar-surj} implies that $M_k(\mf{n}, \bar{\psi}, R/I) \to C_{k}(\psi, R/I)$ is surjective for $k$ sufficiently large, hence the final map in the above exact sequence is surjective. This proves the result.

\end{proof}

\subsection{Lifting the Hasse invariant} \label{s:vk}

 In this section we prove the existence of a Hilbert modular form in level 1 with trivial nebentype that is $p$-adically very close to 1.  
 Results in this direction have appeared in the literature in several places; see for instance  \cite{wilesrep}, \cite{v}, and \cite{ag}.
Unfortunately, none of these lemmas proves the exact result we need as stated---that the form we desire exists in level 1 and has trivial nebentype.  Therefore, for completeness, we give a complete proof.

We begin by describing what we mean by having $q$-expansion equal to $1$, in terms of $q$-expansions at cusps $[C] \in \mr{Cusp}(1)$. Given an even integer $k$, and a cusp label $C$, the module $N_{F/\Q}(\mf{a})^{\tensor -k}$ has a canonical generator, denoted $|N_{F/\Q}(\mf{a})|^{\tensor - k}$, given by taking the $k$-th power of either generator of $N_{F/\Q}(\mf{a}) \isom \Z$. This defines an element $|N_{F/\Q}(\mf{a})|^{-k} \tensor 1 \in C_{C,k}(\Z) \subset Q_{C, k}(\Z)$, and taking this over all cusps, define \[ f_k := (|N_{F/\Q}(\mf{a})|^{-k} \tensor 1)_{[C] \in \mr{Cusp}} \in \oplus_{\mr{Cusp}(1)} C_{[C],k}(\Z) \subset \oplus_{\mr{Cusp}(1)} Q_{[C], k}(\Z). \] This element $f_k \in Q_k(\Z)$ is $G^+_{1}$-invariant. There is also a $G^+_{1}$-invariant element \[ f_1 \in \oplus_{\mr{Cusp}} C_{[C],k}(\F_2) \subset \oplus_{\mr{Cusp}} Q_{[C], k}(\F_2),\] whose square is the reduction of $f_2$. Thus, for all even weights, and for odd weights in characteristic $2$, there is a notion of having a $q$-expansion equal to $1$.

The \textbf{Hasse invariant} $H \in M_{p-1}(1, \Z/p\Z)$ is defined as follows. We define it as an element of $H^0(\M'_{\F_p}, \omega^{\tensor p-1})$. Consider a morphism $S \to \M'_{\F_p}$, corresponding to an \'etale-locally RM-polarizable abelian scheme $A$ over an $\F_p$-scheme $S$. There is a Verscheibung isogeny $V \from A \to A^{(p)}$, giving $V^* \from \omega_A \to \omega_{A^{(p)}} \isom \omega_A^{\tensor p}$, hence an $\cO_S$-linear map $h \from \cO_S \to  \omega_A^{\tensor p - 1}$. The restriction of $H$ to $S$ is defined to be $h(1) \in H^0(S, \omega_A^{\tensor p - 1})$. 

The Hasse invariant has been defined in many contexts, and in those settings the fact that it has $q$-expansion equal to 1 is well-known. That the same is true for our $H$ can be verified via pull-back to $\coprod \M_{\mc{P}}$, adding auxiliary level, and applying (\cite{ag} 7.14). As the $q$-expansion map is injective, this verifies that $H$ has trivial nebentype, i.e. $H \in M_{p-1}(1, (1), \Z/p\Z)$.

\begin{theorem}\label{thm:lifting-hasse}
Fix a prime $p$ and an integer $m > 0$. For $k$ sufficiently large, even, and divisible by $p-1$, there exists a form $A \in M_{k}((1), 1, \Z_{(p)})$ whose $q$-expansion (at all cusps) is equivalent to $1$ modulo $p^m$.
\end{theorem}

While this lifting of the Hasse invariant may seem like an application of Theorem \ref{thm:equiv-lift}, it is not, as $\Z_{(p)}$ is not necessarily a good ring for $\mf{n} = (1)$ due to the existence of elliptic points. The following proof, avoiding this difficulty, was suggested to us by Samit Dasgupta.

\begin{proof}

By Lemma \ref{lem:full-level}, we may choose $N > 0$, coprime to $p$, such that the full level $N$ moduli stack $\M(N)_{\Z[\zeta_N]}$ is an algebraic space. The morphism $c \from \M(N)_{\Z[\zeta_N]} \to \M_{\Z[\zeta_N]}$ is finite \'etale, with Galois group $\GL_2(\cO_F/N)/\cO_F^*$. Theorem \ref{thm:classic-lift} implies that, for $k$ sufficiently large, $c^*(H)^k$ is the reduction modulo $p$ of a form $A \in M_{k(p-1)}(\Gamma(N), \Z_{(p)}[\zeta_N]) = H^0(\M(N)_{\Z[\zeta_N]}, \omega^{\tensor p-1})$. By increasing $k$, we can further assume that \begin{equation}\label{eq:extra-cong} \mr{qexp}(A) \equiv 1 \mod p^{1 + \mr{ord}_p(|\GL_2(\cO_F/N)/\cO^*_F|) + \mr{ord}_p(|G^+_{1}|) + \mr{ord}_p(\Gal(\Q(\zeta_N)/\Q)}. \end{equation}

We will now average the form $A$ in a few ways, with the additional congruence of (\ref{eq:extra-cong}) cancelling with the $p$-power denominators incurred by averaging.  First, average $A$ over the group $\GL_2(\cO_F/N)/\cO^*_F$ to obtain an element of $M_{k}((1), \Z_{(p)}[\zeta_N])$, and then average over $G^+_{1}$ to obtain an element of $M_{k}((1), 1, \Z_{(p)}[\zeta_N])$. As all cusps of $\M$ are defined over $\Z$ (Theorem \ref{thm:cusp-defn}), we can further average by $\Gal(\Q(\zeta_N)/\Q)$ to obtain an element of $M_{k}((1), 1, \Z_{(p)})$. The $q$-expansion of this averaged lift is congruent to $1$ modulo $p^{m}$, as the element $1 \in Q_{k(p-1)}(\Z/p^m\Z) \subset Q_{k(p-1)}(\Z[\zeta_N]/(p^m))$ is invariant under the action of the groups $\GL_2(\cO_F/N)/\cO^*_F, G^+_{1}$, and  $\Gal(\Q(\zeta_N)/\Q)$.

\end{proof}

\subsection{Application to group ring forms}\label{s:gpdef}

Write $G = G^+_{\mf{n}}$. Fix a sign $\epsilon = \pm 1$. Let $\cO$ denote the ring of integers of a finite extension of $\Q_p$ containing the values of all characters $\psi \in \hat{G}$. Let $R$ denote the image of $\cO[G] \to \prod_{\sgn(\psi) = \epsilon} \cO$. We denote by
 \[ \bpsi \colon G \rightarrow R^* \]
 the canonical character. Given a character $\psi \in \hat{G}$ of $\sgn(\psi) = \epsilon$, there is a corresponding map $\psi \from R \to \cO$. Let $(\#G) e_{\psi} \in R$ denote the image of the element $\sum \psi(g)^{-1} [g] \in \cO[G]$, with \[ \psi((\#G) e_{\psi}) = \#G,\ \psi'((\#G) e_{\psi}) = 0\text{ if }\psi \neq \psi'.\]
 
A form $h(\bpsi) \in M_k(\fn, \bpsi, R)$ is a \textbf{group ring form}. For any character of $G$ of $\sgn = \epsilon$, there is a corresponding ring homomorphism $\psi \from R \to \cO$, and \emph{specialization map} $M_k(\fn, \bpsi, R) \to M_k(\fn, \psi, \cO)$, and $h(\psi)$ is the \emph{specialization} of $h(\bpsi)$, the image of $h(\bpsi)$ under $\psi \from R \to \cO$.

  An immediate consequence of Theorem \ref{thm:equivar-surj} is:
\begin{theorem}\label{t:chai} With notation as above, for $k$ sufficiently large and $(-1)^k = \epsilon$, the map
 \[ M_k(\fn, \bpsi, R) \rightarrow C_{p, k}(\bpsi, R), \]
is surjective.
\end{theorem}

We will construct a specific constant term, and use the above theorem to lift it to a group ring valued form:
\begin{theorem} \label{t:auxform} 
For sufficiently large positive integers $k$ with $(-1)^k = \epsilon$, there exists a group ring valued form $G_{k}(\bpsi) \in M_k(\fn, \bpsi, R)$ with normalized constant terms 
at cusp labels $C_{(A, \lambda)} = (A^{-1} \alpha_{\lambda}, L_{\infty})$ with $[C_{(A, \lambda)}] \in \mr{Cusp}_{\infty}(\fn)$ equal to $\bpsi(\det(A)^{-1} (a+ c \mf{t}_{\lambda}^{-1} \mf{d}^{-1}))$, and normalized constant terms at cusps in $\mr{Cusp}_p(\fn) \setminus \mr{Cusp}_\infty(\fn)$ equal to $0$. \end{theorem}
\begin{proof}
Write $C_{\lambda} = (\alpha_{\lambda}, L_{\infty})$ for $\lambda \in \Cl^+(F)$. The cusp label $C_{\lambda} \tensor N^{-1}$ has associated fractional ideal $\mf{a} = N$, and so $C_{C_{\lambda} \tensor N^{-1}, k}(R) = N_{F/\Q}(N)^{\tensor k} \tensor R$. By Lemma \ref{lem:unram-cusp}, every unramified cusp label is isomorphic to $C_{\lambda} \tensor N^{-1}$ for some $\lambda$ and $N$. 

We define a constant term $B = (B_{[C]})_{[C]} \in C_{p,k}(R) = \oplus_{[C] \in \mr{Cusp}_p(\mf{n})} C_{[C],k}(R)$ as follows. If $[C] \in \mr{Cusp}_p - \mr{Cusp}_{\infty}$, we let $B_{[C]} = 0$. If $C = C_{\lambda} \tensor N^{-1}$, we let \[ B_{[C]} := |N_{F/\Q}(N)|^{k} \tensor \bpsi(N) \in N_{F/\Q}(N)^{\tensor k} \tensor R = C_{C,k}(R). \]

\textbf{Well-defined:}

If $C_{\lambda_1} \tensor N_1^{-1} \isom C_{\lambda_2} \tensor N_2^{-1}$, then $N_1 = \alpha N_2$ for some $\alpha \in F^*_{1,\mf{n}}$. Thus $|N_{F/\Q}(N_1)|^{k} \tensor \bpsi(N_1)$ is sent, under $C_{C_{\lambda_1} \tensor N_1^{-1},k}(R) \isom C_{C_{\lambda_2}  \tensor N_2^{-1}, k}(R)$, to \begin{align*}
|N_{F/\Q}(N_1)|^{k} N_{F/\Q}(\alpha)^{-k} \tensor \bpsi(N_1) &=  |N_{F/\Q}(N_2)|^{k} \sgn(N_{F/\Q}(\alpha)^{-k}) \bpsi(N_1) \\ &= |N_{F/\Q}(N_2)|^{k} \bpsi(N_2), \\
\end{align*} where the final equality holds as $\sgn(N_{F/\Q}(\alpha)^{-k}) = \bpsi(\alpha)$ for $\alpha \in F_{1, \mf{n}}^*$ by assumption.

\textbf{$\bpsi$-isotypic:}

Recall the action of $G^+_{\mf{n}}$ on constant terms, described in \S \ref{subsubsec:formal-action}. Given $[N] \in G^+_{\mf{n}}$, $[N] \cdot B = (B_{[C \tensor N^{-1}]})_{[C]}$, where we consider $B_{C \tensor N^{-1}} \in C_{[C \tensor N^{-1}], k}(R)$ as an element of $C_{[C], k}(R)$ via the isomorphism $N_{F/\Q}(\mf{a} N^{-1})^{-k} \tensor R \isom N_{F/\Q}(\mf{a})^{-k} \tensor R$ given by division by $|N_{F/\Q}(N)|^k$.

For $C = C_{\lambda} \tensor N^{-1}$, \[ B_{C \tensor (N')^{-1}} = |N_{F/\Q}(N N')|^{\tensor k} \tensor \bpsi([N N']) \in N_{F/\Q}(N N')^{\tensor k} \tensor R \] corresponds to the element \[ |N_{F/\Q}(N)|^{\tensor k} \tensor \bpsi([N N'])^{-1} = \bpsi([N']) B_{[C]} \in N_{F/\Q}(N)^{\tensor k} \tensor R = C_{[C], k}(R). \] Therefore $[N'] \cdot B = \bpsi([N']) B$, as desired.

\textbf{Formula for $B$:}

We verify that \[ B_{C_{(A, \lambda)}} = |N_{F/\Q}(\mf{a})|^{\tensor -k} \tensor \bpsi(\det(A)^{-1} (a+ c \mf{t}_{\lambda}^{-1} \mf{d}^{-1})). \] We have $C_{(A, \lambda)} = C_{\lambda'} \tensor N^{-1}$ for some $[\lambda'] \in \Cl^+(F)$ and some $N$. This cusp label has associated fractional ideal $\mf{a} = N^{-1}$. For the cusp labels $C_{(A, \lambda)}$, we calculated $\mf{a} = \det(A) (a+ c \mf{t}_{\lambda}^{-1} \mf{d}^{-1})^{-1}$ in Lemma \ref{lem:normalize}. Therefore, for the constant term $B$ defined above, $B_{C_{(A, \lambda)}} = |N_{F/\Q}(\mf{a})|^{\tensor -k} \tensor \bpsi(\det(A)^{-1} (a+ c \mf{t}_{\lambda}^{-1} \mf{d}^{-1}))$.

\textbf{Construction of $G_k(\bpsi)$:}

Apply Theorem \ref{t:chai} to the constant term $B \in C_{p,k}(\bpsi, R)$ defined above to construct the form $G_k(\bpsi)$. 

\end{proof}

Recall that $\mr{const}_p \from M_k(\mf{n}, \cO) \to C_{p,k}(\cO)$ denotes the total constant term map at all $p$-unramified (admissible) cusps. Write $\mf{P} = \gcd(\mf{n}, p^{\infty})$.

\begin{theorem} \label{t:auxform2} Fix a positive integer $m \ge \ord_p(\#G)$.  The following holds for all sufficiently large odd integers $k$.
Let $f_\psi \in M_k(\fn, \psi, \cO)$ be a collection of modular forms for the characters $\psi$ of $G$ of $\sgn(\psi) = -1$,
with the property that $\mr{const}_p(f_\psi) \in p^m C_{p,k}(\psi, \cO)$. There exists a group ring form \[ h(\bpsi) \in M_k(\fn, \bpsi, R) \] such that each specialization $h(\psi)$ satisfies the property that \[ \tilde{f}_\psi = f_\psi - (p^{m}/\#G) h(\psi) \] has constant term 0 at all cusps in $\mr{Cusp}_p(\fn)$.  If $\mf{P} = 1$, then $\tilde{f}_{\psi}$ is cuspidal.  If $\mf{P} \neq 1$, then the ordinary projection $e_{\mf{P}}^{\ord}(\tilde{f}_\psi)$ is cuspidal.
\end{theorem}

We will not recall the definition of the ordinary projector $e_{\mf{P}}^{\ord}$ here---see \cite{dka}, \S 5 for the definition of $e_{\mf{P}}^{\ord}$.

\begin{proof}  
There is an element of $\sum_{\psi} \frac{\#G}{p^m} e_{\psi} \mr{const}_p(f_\psi) \in C_{p, k}(\bpsi, R)$ whose specialization to $C_{p, k}(\psi, R)$ equals $\frac{\#G}{p^m} \mr{const}_p(f_{\psi}) \in C_{p, k}(\psi, R)$ for all $\psi$ (of $\sgn(\psi) = -1$). Apply Theorem~\ref{t:chai} to obtain a form $h(\bpsi) \in M_k(\fn, \bpsi, R)$ such that $\mr{const}_p(h(\psi)) =\frac{\#G}{p^m}\mr{const}_p(f_{\psi}) \in C_{p, k}(\psi, R)$.

Then $h(\bpsi)$ has the desired property, that $\tilde{f}_{\psi}$ has constant term 0 at all cusps in $\mr{Cusp}_p(\fn)$.
 If $\fP = 1$, then $\mr{Cusp}_p(\fn)= \mr{Cusp}(\fn)$ so $\tilde{f}_{\psi}$ is cuspidal.
If $\fP \neq 1$, Theorem 5.1 of \cite{dka} gives the desired result.
\end{proof}

\appendix

\section{Background on stacks}\label{sec:stack}

\subsection{Stacks}\label{subsec:stacks}
Let $\AlgSp/S_0$ (resp. $\AnSp/S_0$) denote the category of algebraic spaces (resp. analytic spaces) over a fixed algebraic (resp. analytic) space $S_0$. For us, $S_0$ will always be an affine scheme (resp. $S_0 = \star$ a point). We will discuss stacks in both algebraic spaces and complex-analytic spaces over a base $S_0$, i.e  categories fibered in groupoids $\mc{X} \to \AlgSp/S_0$ or $\mc{X} \to \AnSp/S_0$, satisfying descent in the \'etale/classical topology. In particular, a stack $\mc{X}$ associates, to any space $S/S_0$, a groupoid $\mc{X}(S)$. We will define stacks by simply describing $S \mapsto \mc{X}(S)$, as the fiber structure will always be clear. To define a stack, it suffices to define it on a basis for the Grothendieck topology on $\AlgSp/S_0$ (resp. $\AnSp/S_0$), such as $\Sch/S_0 \subset \AlgSp/S$. For $S_0 = \star$, we can choose the basis consisting of the simply-connected complex analytic spaces in $\AnSp_{\star}$.

An \emph{algebraic stack}  $\mc{X}$ is a stack over $\AlgSp$ such that:
\begin{enumerate}
\item the diagonal $\Delta \from \mc{X} \to \mc{X} \times \mc{X}$ is representable, quasi-compact, and separated;
\item there exists a scheme $X$ and a surjective smooth morphism $X \to \mc{X}$.
\end{enumerate} 
A \emph{Deligne--Mumford algebraic stack} is an algebraic stack with a surjective \'etale morphism from a scheme. There is a similar notion of a Deligne--Mumford analytic stack, admitting a surjective \'etale morphism from a complex-analytic space.

There is a notion of \emph{equivalence} of stacks, $\mc{X}_1 \isom \mc{X}_2$, corresponding to an equivalence between the underlying categories fibered in groupoids. An equivalence of stacks gives rise to equivalences of categories $\mc{X}_1(S) \to \mc{X}_2(S)$ for all $S$. We will describe equivalences of stacks by describing these equivalences of categories. We will frequently treat equivalent stacks as if they are identical.

Suppose that one is given a group $U$, as well as maps $\iota_x \from U \to \Aut(x)$ for all $x \in \mc{X}(S)$, such that for all $f \from x \to y$, $\iota_x \circ f = f \circ\iota_y$ (also satisfying certain compatibilities with morphisms $S' \to S$). The \emph{rigidification} $\mc{X} \myfatslash U$ is defined by \[ \mr{Ob}(\mc{X}\myfatslash U)(S) := \mr{Ob}(\mc{X})(S),\qquad  \mr{Mor}_{(\mc{X}\myfatslash U)(S)}(x,y) := \mr{Mor}_{\mc{X}(S)}(x,y)/U. \] When $U$ is a finite group, this in fact defines a stack (\cite{romagny} 5.1). 

For the rigidifications we care about, however, $U$ is an infinite group. Nevertheless, this defines a stack -- this is a consequence of the fact that, for these rigidifications, a finite index subgroup of $U$ acts trivially on $\mc{X}$, in the sense of Definition \ref{def:triv} below. 

Most of the stacks we consider will be either Deligne--Mumford stacks (in either $\AlgSp$ or $\AnSp$), or stacks whose rigidification is Deligne--Mumford. The main exception will be the stack $\Lie(\mc{T})$ over $\AnSp$ (\S\ref{subsec:rm-ab}) and an analogous algebraic stack $\Lie(\mc{A})$ over $\AlgSp$ (\S\ref{subsec:hodge-bundle}). These stacks are the quotient of a complex-analytic space/$\Z$-scheme by the action of an infinite group $U^+$, and they fail to be Deligne--Mumford since the diagonal is not quasi-compact.

We will discuss the analytification of Deligne--Mumford algebraic stacks, as well as the analytification of quotients of such stacks by group actions. In our situation, there is no difficulty seeing that the necessary quotients commute with analytification.  (We are not aware of any general statement in this direction.)

Given an element $x \in \mc{X}(S)$, the {\em  inertia group} at $x$ is defined to be the automorphism group $\mr{Mor}_{\mc{X}(S)}(x, x)$. If a Deligne--Mumford algebraic stack has all inertia groups trivial, then it is in fact an algebraic space. In fact, it suffices to check that that all inertia groups at geometric points are trivial (\cite{conrad} Thm. 2.2.5).

\subsection{Group actions on stacks}\label{subsec:grp-act}
We will consider the action of group spaces on stacks. A reference for this is in the algebraic setting is Romagny \cite{romagny}, but the analytic setting is entirely similar. Roughly, a group space $G$ acting on a stack $\mc{X}$ is given by specifying, for all $S$ and all $g \in G(S)$, an isomorphism of groupoids $g_* \from \mc{X}(S) \to \mc{X}(S)$. There is a technical distinction between a \emph{strict} action, where $g_* h_* = (gh)_*$, and a \emph{non-strict} action, where one specifies natural isomorphisms $t_{g,h} \from g_* h_* \isom (gh)_*$ between these functors, satisfying an additional compatibility relating $t_{g_1,g_2 g_3}$ and $t_{g_1 g_2, g_3}$.

If a group $G$ acts on a stack $\mc{X}$, and we have an equivalence of stacks $\mc{X}' \to \mc{X}$, there may not be a natural action of $G$ on $\mc{X}'$.

For example, in \S\ref{subsec:lie-compare}, we will have a complex analytic space $X_{\mf{n}}$ on which a complex Lie group $H_{\mf{n}} = F_{1,\mf{n}}^* \backslash (I_{\mf{n}} \times (F \times \C)^*)$ acts (necessarily strictly). We will have an equivalence $\Lie(\mc{T})^* \isom X_{\mf{n}}$, and we will define a non-strict action of $I_{\mf{n}} \times (F \times \C)^*$ on $\Lie(\mc{T})^*$, which is \emph{equivalent} to the corresponding action on $X_{\mf{n}}$. However, it is not the case that this non-strict action literally factors through the quotient by $F_{1,\mf{n}}^*$! Nevertheless, any $\alpha \in F_{1,\mf{n}}^*$ will act by sending objects $x \in \mr{Ob}(\Lie(\mc{T})^*(S))$ to isomorphic objects, where we specify the isomorphism $x \isom \alpha_*(x)$.

We base the following definition off of the example above. 

\begin{definition}\label{def:triv}A non-strict action of a group $H$ on a stack $\mc{X}$ is \textbf{trivialized}, if for all $h \in H$, there is a natural isomorphism $v_h \from h_* \to id$ between the functors $h_* \from \mc{X} \isom \mc{X}$ and $id \from \mc{X} \isom \mc{X}$, such that:
\begin{itemize}
\item $v_h$ is compatible with the fiber structure, restricting to a natural isomorphism between the functors $h_*, id \from \mc{X}(S) \to \mc{X}(S)$ for $S/S_0$;
\item for all $g, h \in H$, $v_g \circ (v_h)^{g} \circ  t_{g,h} = v_{gh}$, where $v_h^g \from g_* h_* \to g_*$ is the natural transformation which associates to $x \in \mr{Ob}(\mc{X})$ the isomorphism $v_{h}^g(x) \from g_* h_*(x) \nmto{g_*(v_h(x))} g_*(x)$.
\end{itemize}
\end{definition}

We refer to a $G$-stack with a trivialized normal subgroup $H$ as a \textbf{$(G,H)$-stack}. It is possible to replace a $(G,H)$-stack $\mc{X}$ by an equivalent strict $G$-stack $\mc{X}'$, such that the action of $G$ on $\mc{X}'$ factors through $G/H$, as follows. First, replace $\mc{X}$ with the equivalent strict $G$-stack $\mc{X}^{str}$ (\cite{romagny} 1.5). The group $G$ acts freely on $\mr{Ob}(\mc{X}^{str}(S))$, and the action of $H$ on $\mc{X}^{str}$ is still trivialized. Given a groupoid $\mc{C}$ on which $H$ acts strictly, freely on objects, and with a trivialization, there is a quotient groupoid $\mc{C}/H$, with $\mr{Ob}(\mc{C}/H) = \mr{Ob}(\mc{C})/H$. The morphisms are defined by \[ \mr{Mor}_{\mc{C}/H}(x \mod H, y \mod H) = \mr{Mor}_{\mc{C}}(x, y) \] for any $x \in (x \mod H), y \in (y \mod H)$. To see that this is well-defined and to define a composition of morphisms, use the trivialization of the action of $H$. The functor $\mc{C} \to \mc{C}/H$ is an equivalence of categories. We define the desired stack by $\mc{X}'(S) = \mc{X}^{str}(S)/H$.

\subsection{Quotient stacks}\label{subsec:quotient-stacks}

Given a stack $\mc{X}$ in $\AlgSp$ on which a group $G$ acts strictly, the \emph{quotient stack} $\mc{X}/G$ is defined by \[ (\mc{X}/G)(S) = \{ (p, f) : p \from P \to S, f \from P \to \mc{X} \}, \]
where $p$ is an \'etale $G$-torsor over $S$ and $f$ is a $G$-equivariant morphism. Recall that an \'etale $G$-torsor is a morphism of algebraic spaces, where $G$ acts on $P$, $p \circ g = p$, and for some surjective \'etale morphism $S' \to S$, there is a $G$-equivariant isomorphism $P \times_{S} S' \isom G \times S'$.

If $G$ is a finite group, and $\mc{X}$ is an algebraic stack, then $\mc{X}/G$ is an algebraic stack (\cite{romagny} Thm 4.1).

A similar notion of quotient stack may be defined for stacks in $\AnSp$. 

If the groupoid $\mc{X}(S)$ ``is a set" (i.e. the only morphisms are the identity morphisms), and if $S$ is a space such that every $G$-torsor $P \to S$ is trivial, then $(\mc{X}/G)(S)$ equals the \emph{groupoid quotient} of the set $\mr{Ob}(\mc{X}(S))$ by the group $G$. Given a set $Y$ and a group $G$, the groupoid quotient $[Y/G]$ is the groupoid defined by \[ \mr{Ob}([Y/G]) := Y,\ \mr{Mor}_{[Y/G]}(x_1, x_2) := \{ g \in G: x_1 = g(x_2) \} \subset G. \]

For example, if $X$ is a complex-analytic space, and $G$ is a group acting on $X$, the quotient stack $[X/G]$ has $\C$-points equal to $[X(\C)/G]$. More generally, for any simply-connected complex analytic space $S$, we have $[X/G](S) = [X(S)/G]$.

\subsection{Equivariant line bundles}\label{subsec:equiv-line}
We discuss \emph{line bundles} on stacks, i.e. locally free $\cO_{\mc{X}}$-modules of rank 1. To define a line bundle on a stack $\mc{X}/S_0$, we need, for each $S_0$-space $S$, a map of groupoids \[ \mc{X}(S) \to \mc{P}ic(S) = \{ \text{invertible } \cO_S\text{-modules} \},\] compatible with $S \to S'$. In other words, for each object $x \in \mc{X}(S)$, we must specify an invertible $\cO_S$-module $\cL(x)$; given a morphism $f \from (S', x') \to (T, x)$ in $\mc{X}$, there must be a \emph{specified} base-change isomorphism $f^*(\cL(x')) \isom \cL(x)$, compatible with composition in $\mc{X}$.

\begin{definition}
Given a group $G$, a \textbf{$G$-equivariant line bundle} $\cL$ on a $G$-stack $\mc{X}$ is a line bundle $\cL$ such that, for each $g \in G$ and each $x \in \mc{X}(S)$, we specify an $\cO_S$-linear isomorphism $g^* \from \cL(g_*(x)) \isom \cL(x)$. Via the isomorphism \[  \Hom_{\cO_S}(\cL(g_*h_*(x)), \cL(x)) \isom \Hom_{\cO_S}(\cL((gh)_*(x)), \cL(x)) \] defined using $t_{g,h} \from g_* h_* \isom (gh)_*$, $h^* g^*$ must be sent to $(gh)^*$. Moreover, the isomorphisms $g^*$ must be compatible with base-change. To be precise, given a morphism $f \from (S', x') \to (T, x)$ in $\mc{X}$, the following diagram must commute:
\[ \begin{tikzcd}
\cL(g_*(x)) \arrow{r}{g^*} \arrow{d} & \cL(x) \arrow{d} \\
f^*\cL(g_*(x')) \arrow{r}{f^*(g^*)} & f^*\cL(x') \\
\end{tikzcd}, \] where the vertical arrows are the base-change isomorphisms associated to $g_*(f) \from (S', g_*(x')) \to (T, g_*(x))$ and $f \from (S', x') \to (T, x)$ respectively.
\end{definition}

If $\mc{L}$ is a $G$-equivariant line bundle, then $H^0(\mc{X}, \mc{L})$ is a $G$-module, where $g \in G$ acts on a section $s \in H^0(\mc{X}, \mc{L})$ via \begin{equation} g \cdot s := (g^{-1})^*(s). \end{equation} If $G$ were abelian, one could instead consider the action $g \circ s := g^*(s)$. Even though $G$ is abelian in our cases of interest, we do not use this action.

Given a line bundle $\cL$ on $\mc{X}$, its \emph{total space} $\mr{Tot}(\cL)$ is the stack \[ S \mapsto \{ (x, v): x \in \mc{X}(S),\ v \in \cL(x) \}. \] If $G$ acts on the line bundle $\cL$, then $G$ acts on the stack $\mr{Tot}(\cL)$, and $\mr{Tot}(\cL) \to \mc{X}$ is a morphism of $G$-stacks. 

Suppose that the action of a normal subgroup $H \subset G$ on $\mr{Tot}(\cL)$ admits a trivialization. We refer to such a line bundle as a \emph{$(G,H)$-equivariant line bundle}. Then the procedure that gives a stack $\mc{X}'$ on which $G/H$ acts strictly also gives a $G/H$-equivariant line bundle $\cL'$ on $\mc{X}'$, such that $H^0(\mc{X}', \mc{L}') = H^0(\mc{X}, \mc{L})$ as $G$-modules. In particular, $H^0(\mc{X}, \mc{L})$ is naturally a $G/H$-module.

\section{Background on ample line bundles}\label{sec:ample}

Let $R$ be a noetherian ring. 

\begin{definition}
 A line bundle $\cL$ on an algebraic stack $X \to \Spec(R)$ is:
\begin{enumerate}
\item \textbf{globally generated} if the natural map $H^0(X, \cL) \tensor_R \cO_X \to \cL$ is surjective, (or equivalently, if for all geometric points $x \from \Spec(k)  \to X$ there exists $s \in H^0(X, \cL)$ such that $x^*(s) \in x^*(\cL)$ is non-zero), 
\item \textbf{semi-ample} if $X$ is proper over $\Spec(R)$, and $\cL^{\tensor k_0}$ is globally generated for some $k_0 > 0$.
\end{enumerate}
\end{definition}

We will not consider any notion of ampleness on algebraic stacks, only the usual notion for proper schemes.

Globally generated line bundles give rise to maps $X \to \Proj(\Sym_R^*(H^0(X, \cL^{\tensor k_0})))$. If $X$ is proper and normal, semi-ample line bundles give rise to maps $f_{\cL} \from X \to X_{\cL}$, where $X_{\cL}$ is a certain projective scheme and $f_{\cL}$ has connected fibers. More precisely:
\begin{prop}[\cite{lan-compact} \S 7.2.2, 7.2.3]\label{prop:semiample}
Consider a semi-ample line bundle $\cL$ on a proper, normal algebraic stack $X \to \Spec(R)$, such that $\cL^{\tensor k_0}$ is globally generated. Then:
\begin{enumerate}
\item $\oplus_{k \geq 0} H^0(X, \cL^{\tensor k})$ is a finitely generated $R$-algebra.
\item The map $X \to \Proj(\Sym_R^*(H^0(X, \cL^{\tensor k_0})))$ factors through a map 
\[ f_{\cL} \from X \to X_{\cL} := \Proj(\oplus_{k \geq 0} H^0(X, \cL^{\tensor k})), \] and the map $X_{\cL} \to \Proj(\Sym_R^*(H^0(X, \cL^{\tensor k_0})))$ has finite fibers.
\item $X_{\cL}$ is a normal projective scheme of finite-type over $\Spec(R)$.
\item $f_{\cL}$ has connected fibers.
\item $(f_{\cL})_*(\cL^{\tensor k_0})$ is an ample line bundle on $X_{\cL}$ and $f^*_{\cL}(f_{\cL})_*(\cL^{\tensor k_0}) = \cL^{\tensor k_0}$.
\end{enumerate}
\end{prop}

\begin{remark}
Note that $(f_{\cL})_*(\cL)$ is not necessarily a line bundle.
\end{remark}

The above result is an application of Stein factorization and Zariski's connectedness theorem.

\begin{lemma} \label{l:section}
Let $R$ be a (noetherian) $W(\bar{\F}_p)$-algebra. Consider a normal algebraic stack $X$, proper over $\Spec(R)$. If $\cL$ is a semi-ample line bundle on $X$, then there exists $m > 0$ such that, for all finite sets $\{x_i \}$ of geometric points of $X$, there exists $s \in H^0(X, \cL^{\tensor m})$ such that $x_i^*(s) \in x_i^*(\cL^{\tensor m})$ is non-zero for all $i$.
\end{lemma}
\begin{proof}
By Proposition~\ref{prop:semiample} (5), it  suffices to check this for a very ample line bundle $\cL$ on a normal proper scheme $X \to \Spec(R)$. For any $x_i$ the map $H^0(X, \cL) \to x_i^*(\cL)$ is non-zero.  Let $s_i$ denote a section with $x_i^*(s_i) \neq 0$.

A simple inductive argument shows that there exists roots of unity $u_i \in W(\bar{\F_p})$ such that $s = \sum u_i s_i$ is the desired section.
\end{proof}

This implies:

\begin{corollary}\label{cor:semiample-descent}
Let $X$ be a normal algebraic stack, proper over $\Spec(R)$, and let $\cL$ be a line bundle on $X$. If there exists a surjective finite \'etale morphism $f \from X' \to X$ such that $f^*(\cL)$ is semi-ample, then $\cL$ is semi-ample.
\end{corollary}
\begin{proof}
As semi-ampleness can be tested after a faithfully flat base-change \[ \coprod_i \Spec(R_i) \to \Spec(R), \] we may assume $R$ is a $W(\bar{\F}_p)$-algebra. Since the pull-back of a semi-ample line bundle is semi-ample, we can replace $X'$ by its Galois closure, and assume that $f \from X' \to X$ is a finite \'etale Galois morphism with Galois group $G$.

Given a geometric point $x$ of $X$, write $f^{-1}(x) = \{ x_1, \ldots, x_n \}$. By Lemma~\ref{l:section}, there is a section $s '\in H^0(X' , f^*(\cL)^{\tensor m})$ such that $x_i^*(s') \neq 0$. Then the norm \[ \prod_{g \in G} g^*(s') \in H^0(X', f^*(\cL)^{\tensor |G| m}) \] is $G$-invariant,  and thereby descends to a section $s \in H^0(X, \cL^{\tensor |G| m})$. Clearly $x^*(s) \neq 0$. 
\end{proof}

We will also want the following facts about quotients of projective schemes. Recall that if $X \to \Spec(R)$ is a quasiprojective scheme on which a finite group $G$ acts, there exists a quotient scheme $X/G$ and a finite morphism $p \from X \to X/G$. The map $p$ induces the quotient map on the underlying topological spaces, and $\cO_{X/G} = p_*(\cO_X)^G$. More can be said when $X$ is projective:

\begin{lemma}[\cite{sga} 4.1, 5.1]\label{lem:proj-quot}
Let $X \to \Spec(R)$ be a projective scheme on which a finite group $G$ acts, and let $\cL$ be an ample line bundle on $X$. Define $N_G(\cL) = \bigotimes_{g \in G} g^*(\cL)$. Then:
\begin{enumerate}
\item $X/G$ is projective;
\item $\cL' := (p_*(N_G(\cL)))^G$ is an ample line bundle on $X/G$;
\item $N_G(\cL) = p^* (\cL')$.
\end{enumerate}
\end{lemma}

Note that if $\cL$ is $G$-equivariant, then $N_G(\cL) \isom \cL^{\tensor |G|}$.

\end{document}